%% file: LinearGensPaper.tex
\documentclass[reqno]{amsart}
\usepackage{Preambling}
\usepackage[foot]{amsaddr}
\usepackage[backend=biber,bibstyle=alphabetic,citestyle=alphabetic]{biblatex}
\usepackage{upref}
\crefdefaultlabelformat{#2\textup{#1}#3}

\title[\texorpdfstring{$w$}{w}-Linear Generators]{Triangulated Categories Admitting \texorpdfstring{$w$}{w}-Linear Generators}
\author{Marina Godinho$^1$}
\address{$^1$ Department of Mathematics and Statistics, University of Glasgow, UK}

\author{Dave Murphy$^{2,\ast}$}
\address{$^2$ Liverpool, UK}
\address{* Corresponding author}
\date{}
\subjclass{18G80,16E45,16E35}
\keywords{perfect derived category, classical generators, linear generators, discrete cluster categories.}

\begin{document}

\begin{abstract}
	We introduce the notion of a $w$-linear generator in a $k$-linear, Hom-finite, Krull-Schmidt algebraic triangulated category $\cT$ over a field $k$, for some positive integer $w$.
	We show that the existence of a $w$-linear generator gives us a complete classification of indecomposable objects and morphisms spaces in $\cT$, as well as the factorisation of morphisms.
	We compute the graded endomorphism ring $\Lambda$ of a $w$-linear generator viewed as a differential graded algebra with trivial differential, and the Hochschild cohomology of $\Lambda$.
	It turns out that the graded endomorphism ring of a $w$-linear generator is intrinsically formal, and so $\cT$ is triangle equivalent to the perfect derived category of $\Lambda$.
	
	As a application of these results, we find a description of the Paquette--Y\i ld\i r\i m completion of discrete cluster categories of type $A_{\infty}$ as a perfect derived category.
	Further, we confirm a conjecture of Franchini in her PhD thesis, and show that the $(w+1)$-Calabi--Yau candidate categories she introduces are in fact triangulated categories.
\end{abstract}

\maketitle

\tableofcontents

\input{wLinearIntro}

\input{wLinearGenerator}

\input{TheAlgebraLambda}

\input{Bardzell}

\input{CombinatorialModel}

\appendix

\input{ChoiceOfGeneratorsII}

\printbibliography

\end{document}

%% file: wLinearIntro.tex
\section{Introduction}

Classical generators of triangulated categories are natural objects to look at when we care about studying representation theory.
Indeed, much celebrated results in tilting theory depend upon classical generators; such as Rickard's theorem \cite{DerivedMorita} and a generalisation due to Keller \cite{KellerDeriving}.

In an abstract sense, a classical generator can provide us with large amounts of information about the triangulated category that it generates.
We can determine objects via iterated extensions, Hom-spaces by applications of Hom-functors to these extensions, etc.
However, in a practical sense, this information can be extremely difficult to obtain, if not impossible in full generality.

Fortunately, there are classical generators which can be friendly, and allow us to practically determine all of the possible information that they contain.
One such example is an hereditary algebra $H$ in its bounded derived category $\mathrm{D}^{\mathrm{b}}(H)$, where a complete understanding of the module category of $A$ quickly leads us to an understanding of $\mathrm{D}^{\mathrm{b}}(H)$.

Another example comes from the polynomial ring in a single variable, $\Lambda_1 = k[x]$, viewed as a differential graded algebra (dga) by placing $x$ in degree $-1$ and equipping $k[x]$ with a trivial differential.
The category $\mathrm{D}^{\mathrm{f}}(\Lambda_1)$, the subcategory of complexes with finite dimensional total homology in $\mathrm{D}(\Lambda_1)$, was studied by Holm--J\o rgensen in \cite{Holm2009}.
It was shown that $\mathrm{D}^{\mathrm{f}}(\Lambda_1)$ admits a combinatorial model, where indecomposable objects are in bijection to arcs connecting non-adjacent vertices of an $\infty$-gon, and the category $\mathrm{D}^{\mathrm{f}}(\Lambda_1)$ is classically generated by the complex with $k$ placed in degree zero, and is trivial in all other degrees.

Further, it was shown in \cite{ACFGS} that the perfect derived category of $A$ is equivalent as triangulated categories to the Paquette--Y\i ld\i r\i m completion of the discrete cluster category of type $A_{\infty}$ with a single accumulation point \cite{Paquette2020}.
The indecomposable object corresponding to $A$ is particularly nice, generating $\mathrm{perf}\, \Lambda_1$ in a single step \cite{Generators}, and providing us with a complete description of $\mathrm{perf}\, \Lambda_1$ with standard techniques in homological algebra.

However, this raises an interesting question.
There exists a Paquette--Y\i ld\i r\i m completion for any positive number $n$ of accumulation points, yet we do not have a description in terms of a perfect derived category of a dga.
These categories are algebraic triangulated categories with a classical generator \cite{Generators}, so they do admit a description as the perfect derived category of a dga by a result of Keller \cite{KellerDeriving}.
So, can we construct such a dga for an arbitrary number of accumulation points?

With this question in mind, we introduce in \Cref{sec:ww Linear Gens} the concept of triangulated categories with $w$-\textit{linear generators} a positive integer $w$.
A $w$-linear generator abstractly generalises the homological properties of $\Lambda_1$ that makes $\mathrm{perf}\, \Lambda_1$ so nice to work with, allowing us large amounts of control over our ambient triangulated category.

Given a triangulated category $\cT$, and an object $G \in \cT$, let $\lang[1]{G} \subset \cT$ denote the smallest full subcategory containing $G$ and closed under direct sums, direct summands, and suspension. Roughly, a linear generator $G$ of $\cT$ is a classical generator $G \in \cT$ such that the set $\mathrm{ind}( \lang[1]{G})$ of indecomposable objects in $\lang[1]{G}$ admits a total order where $Q < P \in \mathrm{ind}(\lang[1]{G})$ if and only if $\hcT{P}{Q} \cong k$ and $\hcT{Q}{P} =0$. Moreover, we place some conditions on the total ordering concerning the factorisation of morphisms. 

Throughout, we work within the following setup.

\begin{setup*}\label[set]{setup: cluster cat intro}
	Let $k$ be a field and let $\cT$ be a $k$-linear Hom-finite, Krull-Schmidt triangulated category.   
\end{setup*}  

It turns out that the existence of a linear generator provides us a large degree of control over $\cT$. For example, in \Cref{sec:ww Linear Gens}, we describe all indecomposable objects of $\cT$ in terms of $G$, we describe all triangles in $\cT$ with at least two indecomposable terms, and we compute the Rouquier dimension of $\cT$. Most importantly, the properties of linear generators allow us to prove the main result of \Cref{sec:ww Linear Gens}. 

\begin{theorem*}[\Cref{Thm:w Perf 2 is Tri} and \Cref{prop:w Ind in langle 2}]\label{wIntro Thm1}
	Let $\cT$ be a triangulated category satisfying \Cref{setup:w cluster cat}, and equipped with a $w$-linear generator $G$. Then the category $\langle G \rangle_2$ is triangulated, and hence equivalent to $\cT$. Equivalently, the generation time of $G$ is at most $1$.
	
	Moreover, all indecomposable objects in $\cT$ are either in $\lang[1]{G}$, or isomorphic to the cone of a morphism between indecomposable objects in $\lang[1]{G}$.
\end{theorem*}

That is, the existence of a $w$-linear generator in $\cT$ allows us to completely determine the indecomposable objects in $\cT$, and builds these indecomposable objects via distinguished triangles where all of the terms are indecomposable.

This already shows that the existence of a $w$-linear generator is particularly powerful in studying $\cT$.
However to achieve our goal of a description of the Paquette--Y\i ld\i r\i m completions as a perfect derived category of a dga, we must start looking at dgas.
To do this, we will compute the graded endomorphism algebra of a $w$-linear generator, and equip this graded algebra with a trivial differential to view it as a dga.

Let $G \cong \bigoplus_{i=1}^m G_i$ be a $w$-linear generator for a triangulated category $\cT$, and suppose that $\mathrm{Hom}_{\cT}(G_i,G_j) \cong k$ if and only if $i \leq j$, and $\mathrm{Hom}_{\cT}(G_i,G_j) =0$ otherwise.
Then let $Q^G$ be a linear graded quiver of Dynkin type $A_m$, with arrows in degree zero from vertex $i$ to $i+1$, and additional loops placed in degree $-w$ at all vertices.
Further, if $G_i \cong G_i[w]$, we place another loop of degree $w$ at vertex $i$.
The set of relations $I^G$ are constructed in such a way to guarantee that two parallel paths are equivalent if and only if they are of equal degrees, and the placement of loops in the path is not considered.
Let $\Lambda^G$ be the graded endomorphism ring of $G \in \cT$.

\begin{theorem*}[\Cref{Thm:path algebras}]\label{wIntro Thm2}
	Let $G$ be a $w$-linear generator of a triangulated category $\cT$ satisfying \Cref{setup:w cluster cat}.
	Then there exists an isomorphism of graded $k$-algebras,
	\[
	\Lambda^G \xlongrightarrow{\sim} kQ^G/I^G.
	\]
\end{theorem*}

The Hochschild cohomology of a dga is a derived invariant that provides a lot of information about the dga, however it is notoriously difficult to compute, and again is not even necessarily possible to compute for certain algebras.
Thankfully, given we are in a situation closely related to Dynkin quivers of type $A_m$, so we do have some techniques available to us to compute the Hochschild cohomology of $\Lambda^G$.
In particular, the Bardzell resolution has been used to compute the Hochschild cohomology of graded (skew-)gentle algebras \cite{OpperGentle,BSSWW}, and even the entire Tamarkin--Tsygen calculus for a gentle algebra \cite{CSSS}.

Therefore, we introduce an adaptation of the Bardzell resolution which accounts for the additional loops in $(Q^G,I^G)$, called the \textit{looped Bardzell resolution}.
We show that the looped Bardzell resolution is a projective resolution of $\Lambda^G$ over its enveloping algebra, and so we can compute the Hochschild cohomology via the looped Bardzell resolution.
This brings us to the main result of this work.

\begin{theorem*}[\Cref{Thm: HH} and \Cref{Cor: Intrinsically formal}]\label[thm]{wIntro Thm3}
	Let $G$ be a $w$-linear generator of a triangulated category $\cT$ satisfying \Cref{setup:w cluster cat}, and let $\Lambda^G$ be the graded endomorphism algebra of $G$ viewed as a dga.
	Then we have a complete list of elements that generate $\coHH{r,q}(\Lambda^G)$.
	
	In particular, $\coHH{r,q}(\Lambda^G) = 0$ for all $r \geq 2$, and so $\Lambda^G$ is intrinsically formal.
\end{theorem*}

As a direct result of \Cref{wIntro Thm3}, we now know that an algebraic triangulated category $\cT$ satisfying \Cref{setup:w cluster cat} with a $w$-linear generator is equivalent as triangulated categories to the perfect derived category of $\Lambda^G$ due to a result by Keller \cite{KellerDeriving}.
This provides a positive answer to our motivating question: we can construct a dga $\Lambda_n$ such that $\mathrm{perf}\, \Lambda_n$ is equivalent as triangulated categories to the Paquette--Y\i ld\i r\i m completion with $n$ accumulation points.
The construction of $\Lambda_n$ can be found in \Cref{Sec: Lambda}.

We finish by constructing a combinatorial model for $\mathrm{perf}\, \Lambda^G$, and showing that such a model exists for all $w$-linear generators.
These combinatorial models generalise the combinatorial models for Paquette--Y\i ld\i r\i m categories, and for discrete cluster categories of type $A_{\infty}$ \cite{Igusa2013,Gratz2017}, where indecomposable objects are a in bijection to a subset of arcs in $(w+2)$-angulations of the $\infty$-gon with $n$ accumulation points.

One particular application of the construction of these combinatorial models is to construct a $w$-Calabi--Yau version of the discrete cluster categories of Dynkin type $A_{\infty}$.
We do this by confirming a conjecture by Franchini in her PhD thesis \cite{FranchiniThesis}, showing that her candidate categories are in fact triangulated.
A thorough investigation into the representation theory of these categories is provided in \cite{FranchiniThesis}, under the assumption that the candidate categories are triangulated.

The organisation of this work is as follows.
In \Cref{sec:ww Linear Gens}, we introduce the notion of a $w$-linear generator for a triangulated category $\cT$, and explore the consequences of the existence of a $w$-linear generator.
This is where we determine indecomposable objects, morphisms, and factorisation properties in $\cT$.

In \Cref{Sec: Graded Endo} we compute the graded endomorphism algebras of a $w$-linear generator, first as a graded matrix algebra with entries being a polynomial ring or Laurent polynomial ring in a single variable.
Next we compute the graded endomorphism algebra as the path algebra of a graded quiver with relations, and finally we construct the algebra $\Lambda_n$.

We introduce the looped Bardzell resolution, and compute the Hochschild cohomology of $\Lambda^G$ as a dga in \Cref{Sec: HH}.
Finally, we construct the combinatorial model and confirm the conjecture by Franchini in \Cref{Sec: comb}.

\subsubsection*{Conventions}

Throughout, we let $k$ be an algebraically closed field.
All modules will be considered as right modules, unless otherwise stated, and we will compose paths from left to right.
That is, if $a$ and $b$ are paths, then we write $ab$ to mean first apply $a$ then $b$. 
However, morphisms will be composed right to left, \ie if $f : X \rightarrow Y$ and $g :Y \rightarrow Z$, then $gf : X \rightarrow Z$.
For a category $\cC$ and two objects $X,Y \in \cC$, we write $\cC(X,Y) = \Hom[\cC]{X}{Y}$, and we shall denote the set of isomorphism classes of indecomposable objects in $\cC$ by $\mathrm{ind}\, \cC$

We will assume that the suspension functor $[1]_{\cT}$ for a triangulated category $\cT$ has a natural inverse, $[-1]_{\cT}$, and we will simply denote the suspension functor by $[1]$ if the category is clear from context.
We will always take triangle to mean a distinguished triangle in a triangulated category.

For a graded/differential graded algebra $A$, we denote the degree $m \in \mathbb{Z}$ component of $A$ by $A^m$.

\subsubsection*{Acknowledgements}

In the process of writing this work, the authors had helpful conversations with with many people, and so it is a pleasure to thank these people for their kind help. Special thanks goes to Greg Stevenson, David Paukstello, Rosie Laking, and Jenny August for their encouragement and valuable suggestions. 

\subsubsection*{Funding}

The first named author was supported by ERC Consolidator Grant 101001227 (MMiMMa).
The second named author was supported by project LAVIE - Large views of small phenomena: decompositions, localizations, and representation type, FIS 00001706, funded by Program FIS2021 of Italian Ministry of University and Research.

\subsubsection*{Open Access} For the purpose of open access, the author has applied a Creative Commons Attribution (CC:BY) licence to any Author Accepted Manuscript version arising from this submission.

%% file: wLinearGenerator.tex
\section{Triangulated Categories with a \texorpdfstring{$w$}{w}-Linear Generator}\label{sec:ww Linear Gens}

This section introduces the concept of triangulated categories with $w$-linear generators. It turns out that the existence of such generators imply very strong and interesting properties on the triangulated category and its morphisms. We are interested in these types of categories as the main examples of categories in this paper satisfy this setup and, moreover, the abstract properties of this setup that are proved in this section will play a key role in the remainder of the paper.

Throughout, we will be interested in triangulated categories satisfying the following setup.

\begin{setup}\label[set]{setup:w cluster cat}
	Let $k$ be a field and let $\cT$ be a $k$-linear, Hom-finite, Krull-Schmidt triangulated category.   
\end{setup} 

\subsection{Linear Generators}

We begin by recalling some standard notation and terminology concerning generators of triangulated categories.

Let $\cT$ be a triangulated category and consider a subcategory $\cS$ of $\cT$. Recall that $\cS$ is \textit{thick} if it is a triangulated subcategory of $\cT$ closed under direct summands. Moreover, given an object $X \in \cT$, we are interested in the following subcategories associated to $X$.

\begin{enumerate} 
	\item The \textit{thick closure} $\lang{X}$ of $X$ is the smallest thick subcategory of $\cT$ containing $X$.
	\item The category $\lang[1]{X}$ is the smallest additive full subcategory of $\cT$ which contains $X$, and is closed under direct summands and the suspension functor $[1]$ of $\cT$.
	\item The category $\lang[m]{X}$ is defined inductively as the full subcategory of $\cT$ consisting of objects $A \in \cT$ which fit into a triangle
	\[
	B \rightarrow A \rightarrow C \rightarrow B[1],
	\]
	where $B \in \lang[1]{X}$ and $C \in \lang[m-1]{X}$. In other words, $\lang[m]{X} := \lang[1]{X} \star \lang[m-1]{X}$.
\end{enumerate}

A \textit{classical generator} of $\cT$ is an object $G \in \cT$ such that $\lang{G} = \cT$. A classical generator is, moreover, a \textit{strong generator} if $\lang[m]{G} = \cT$ for some integer $m \geq 1$. Given a strong generator $G$, let $n$ be the smallest positive integer such that  $\lang[n]{G} = \cT$. Then, the \textit{generation time} of $G$ is $n-1$.

\medskip
This section investigates the properties of triangulated categories satisfying \Cref{setup:w cluster cat}, equipped with a classical generator $G$ such that $\lang[1]{G}$ satisfies certain criteria.

\begin{definition} \label[def]{def:w linear gen}
	Let $\cT$ be a triangulated category satisfying \Cref{setup:w cluster cat} and suppose that $G \in \cT$ is a classical generator.
	Let $w$ be a positive integer.
	Then, we say that $G \cong \bigoplus_{i=1}^n$ is a $w$-\textit{linear generator} if it satisfies the following properties:
	\begin{enumerate}[label={(G\arabic*)}]
		\item The set $\mathrm{ind}\, \lang[1]{G}$ admits a strict partial ordering defined as:  \label{wG1}
		
		For objects $Q, P \in \mathrm{ind}\, \lang[1]{G}$, then $Q \leq P$ if and only if $\hcT{P}{Q} \cong k$. Moreover, if $Q \not\cong P$, then $\hcT{Q}{P}=0$.
		
		\item Given an indecomposable object $P \in \mathrm{ind}\, \lang[1]{G}$, then $P \leq P[m]$ if and only if $0 \leq m \in w\mathbb{Z}$.  \label{wG2}
		
		\item For objects $P, Q \in \mathrm{ind}\, \lang[1]{G}$, then $P \leq Q \leq P[w]$ implies that $Q$ is isomorphic to at least one of $P$ and $P[w]$.  \label{wG4}
		
		\item For objects $R \leq Q \leq P$ in $\mathrm{ind}\, \lang[1]{G}$, then any morphism from $P$ to $R$ factors non-trivially through $Q$.  \label{wG3}
		
		\item Given two indecomposable objects $P,Q \in \lang[1]{G}$, such that $Q \not\cong P[i]$ for $i \in \mathbb{Z}$, and $Q <P$, then $Q < P[m]$ if and only if $m \in w\mathbb{Z}$. \label{wG5}
		
		\item $P_i < P_j$ for all $i > j$.
	\end{enumerate}
\end{definition}

\begin{lemma}\label[lem]{lem:w Pa cone is ind}
	Let $\cT$ be a triangulated category satisfying \Cref{setup:w cluster cat} and admitting a $w$-linear generator $G$. Suppose that $f \colon P \rightarrow Q$ is a non-zero morphism between non-isomorphic, indecomposable objects in $\lang[1]{G}$. Then, $\mathrm{cone} \, f$ is indecomposable in $\cT$.
\end{lemma}

\begin{proof}
	We will prove the statement by proving that $\cone{f}$ has a local endomorphism ring, as $\cT$ is Krull--Schmidt. 
	
	Given that $f$ is non-zero and $G$ is a $w$-linear generator, then it must be that $Q \leq P$ by property \ref{wG1} in \Cref{def:w linear gen}. If $Q \cong P$, then, since $\hcT{P}{Q} \cong k$, it follows that $f$ is an isomorphism and we may conclude that $\cone{f} = 0$. Hence, it suffices to consider the case where $P$ and $Q$ are not isomorphic.
	
	Consider the triangle
	\[
	P \xrightarrow{f} Q \xrightarrow{g} \mathrm{cone} \,  f \rightarrow P[1].
	\]
	Applying the covariant functor $\hcT{P[1]}{-}$ induces the long exact sequence
	\[
	\hcT{P[1]}{P}\xrightarrow{\hcT{P[1]}{f}} \hcT{P[1]}{Q} \rightarrow
	\hcT{P[1]}{\mathrm{cone} \, f} \rightarrow
	\hcT{P[1]}{P[1]} \xrightarrow{\hcT{P[1]}{f[1]}} \hcT{P[1]}{Q[1]}.
	\]
	Properties \ref{wG1} and \ref{wG2} in \Cref{def:w linear gen} imply that $\hcT{P[w]}{P} \cong k$ and $\hcT{P[w]}{Q} \cong k$.
	
	If $w = 1$, then the morphism $\hcT{P[1]}{f}$ is nonzero as $Q < P \leq P[1]$, hence by \ref{wG3} any morphism $g$ from $P[1]$ to $Q$ factors through $P$ as $g = f'h$, and $f$ is a nonzero scalar multiple of $f'$ as $\hcT{P}{Q} \cong k$.
	The morphism $\hcT{P[1]}{f[1]}$ is nonzero by an analogous argument, and both morphisms are isomorphisms by a dimensional argument.
	
	Further, observe that if $w>1$ then $\hcT{P[1]}{f}$ and $\hcT{P[1]}{f[1]}$ are the zero map since both domains and codomains are trivial.
	
	As a result, we may conclude from the long exact sequence that $\hcT{P[1]}{\cone{f}} = 0$. Similarly, applying $\hcT{P}{-}$ to the triangle defining $\cone{f}$ shows that $\hcT{P}{\cone{f}}$ vanishes. 
	
	Moreover, since $P \not\cong Q$, then $\hcT{Q}{P} = 0$ by \Cref{def:w linear gen} \ref{wG1}. Therefore, applying the functor $\hcT{Q}{-}$ to the triangle defining $\cone{f}$ shows that $\hcT{Q}{\cone{f}} \cong k$.
	
	Finally, the contravariant functor $\hcT{-}{\mathrm{cone} \, f}$ applied to the triangle defining $\cone{f}$ induces the exact sequence
	\[
	\hcT{P[1]}{\mathrm{cone} \, f} \rightarrow \hcT{\mathrm{cone} \, f}{\mathrm{cone} \, f} \rightarrow
	\hcT{Q}{\mathrm{cone} \, f} \rightarrow
	\hcT{P}{\mathrm{cone} \, f}.
	\]
	Since $\hcT{P}{\cone{f}}$ and $\hcT{P[1]}{\cone{f}}$ vanish, and $\hcT{Q}{\cone{f}} \cong k$, then it must be that $\hcT{\cone{f}}{\cone{f}} \cong k$. That is, $\cone{f}$ has a local endomorphism ring, whence $\mathrm{cone} \, f$ is indecomposable.
\end{proof}

Given indecomposable objects $P, Q \in \lang[1]{G}$, it is worth noting that the cones $\cone{f}$ and $\cone{g}$ of any two nonzero morphisms $f \colon P \rightarrow Q$ and $g \colon P \rightarrow Q$ are isomorphic. To see this simply observe that, since there exists nonzero morphisms $P \to Q$, then $\hcT{P}{Q} \cong k$. Therefore, there exists a scalar $\lambda \in k$ such that $f = \lambda \cdot g$. In other words, since composition is $k$-linear, $f = g \circ \lambda \cdot \id_P$. We may thus construct the following commutative diagram
\[
\begin{tikzcd}
	P \arrow[rr, "f"] \arrow[d, "{\lambda \cdot \id_P}"', "\sim" labl1] && Q \arrow[rr] \arrow[d, equals] && \cone{f} \arrow[rr] \arrow[d, dashed] && P[1] \arrow[d] \\ 
	P \arrow[rr, "g"] && Q \arrow[rr] && \cone{g} \arrow[rr] && P[1] \\ 
\end{tikzcd}
\]
showing that $\cone{f} \cong \cone{g}$. 

\begin{notation} \label{not:w XPQ}
	In view of the remark above, given any non-zero morphism $f \colon P \rightarrow Q$, the indecomposable object $\cone{f}$ will be denoted $X_{P, Q}$.
\end{notation}

\subsection{Morphism Spaces}

Given a triangulated category $\cT$ as in \Cref{setup:w cluster cat} and a $w$-linear generator $G$ in $\cT$, this subsection aims to describe morphisms in $\hcT{X}{Y}$ where $X$ and $Y$ are indecomposable objects in $\lang[2]{G}$. In order to do so, we first describe morphisms whose domain or codomain object is $X_{P, Q}$. 

\begin{lemma}\label[lem]{lem:w Hom X to P}
	With notation as above, then $\hcT{X_{P,Q}}{R} \cong k$ if and only if $Q[1] < R \leq P[1]$. Otherwise, $\hcT{X_{P,Q}}{R}=0$.
\end{lemma}
\begin{proof}
	Applying the functor $\hcT{-}{R}$ to the triangle 
	\[
	P \xrightarrow{f} Q \rightarrow X_{P,Q}  \rightarrow P[1]
	\]
	defining $X_{P, Q}$ induces the long exact sequence
	\[
	\hcT{Q[1]}{R} \xrightarrow{\hcT{f[1]}{R}} \hcT{P[1]}{R} \rightarrow \hcT{X_{P,Q} }{R} \rightarrow \hcT{Q}{R} \xrightarrow{\hcT{f}{R}} \hcT{P}{R}.
	\]
	Observe that $Q \leq P$ since $X_{P, Q}$ is the cone of a non-zero morphism $P \to Q$ by assumption.
	
	By the properties of a long exact sequence, $\hcT{X_{P,Q}}{R} \neq 0$ if and only if at least one of the following holds: $\hcT{f[1]}{R}$ is not surjective, $\hcT{f}{R}$ is not injective.
	
	Consider the morphism $\hcT{f[1]}{R}: \hcT{Q[1]}{R} \rightarrow \hcT{P[1]}{R}$.
	As $\hcT{Q[1]}{R}$ and $\hcT{P[1]}{R}$ are at most one-dimensional, and $\hcT{f[1]}{R}$ is an isomorphism if both are one-dimensional, then $\hcT{f[1]}{R}$ is not surjective if and only if $\hcT{Q[1]}{R} = 0$ and $\hcT{P[1]}{R} \cong k$.
	It is clear by \Cref{def:w linear gen} that this holds if and only if $Q[1] < R \leq P[1]$, and so $\hcT{W_{P,Q}}{R} \cong k$ when $Q[1] < R \leq P[1]$.
	
	Now consider the morphism $\hcT{f}{R}: \hcT{Q}{R} \rightarrow \hcT{P}{R}$.
	By a similar argument as above, $\hcT{f}{R}$ is not injective if and only if $\hcT{P}{R} = 0$ and $\hcT{Q}{R} \cong k$.
	However, $\hcT{Q}{R} \cong k$ implies that $P < Q \leq R$, and so $\hcT{P}{R} \cong k$, a contradiction.
	
	Hence $\hcT{X_{P,Q}}{R} \cong k$ if and only if $Q[1] < R \leq P[1]$. Otherwise, $\hcT{X_{P,Q}}{R}=0$.
\end{proof}

\begin{lemma}\label[lem]{lem:w Hom P to X}	
	
	Let $\cT$ be a triangulated category satisfying \Cref{setup:w cluster cat} and equipped with a $w$-linear generator $G$. Then, $\hcT{R}{X_{P,Q}} \cong k$ if and only if $Q \leq R < P$. Else, $\hcT{R}{X_{P,Q}} =0$.
\end{lemma}

\begin{proof}
	The proof is analogous to that of \Cref{lem:w Hom X to P} by considering the functor $\hcT{R}{-}$.
\end{proof}

\begin{corollary} \label[cor]{cor:w unique triangle} \label[cor]{Cor:w Xs arent iso}
	
	Let $\cT$ be a triangulated category satisfying \Cref{setup:w cluster cat} and equipped with a $w$-linear generator $G$. Then, $Q$ and $P$ are the unique (up to isomorphism) indecomposable objects in $\lang[1]{G}$ which sit in a triangle
	\[ P \rightarrow Q \rightarrow X_{P, Q} \rightarrow P[1]. \]
	In other words, $X_{P, Q} \cong X_{P', Q'}$ if and only $P \cong P'$ and $Q \cong Q'$.
\end{corollary}
\begin{proof}
	Suppose that there is a nonzero morphism $g \colon P' \to Q'$ such that $\cone{g} = X_{P', Q'} \cong X_{P, Q}$. We will show that $P' \cong P$ and $Q' \cong Q$. Well, since $X_{P', Q'} \cong X_{P, Q}$ by assumption, there are nonzero morphisms $Q' \to X_{P, Q}$ and $Q \to X_{P', Q'}$. By \Cref{lem:w Hom P to X}, it must be that $Q \leq Q'$ and $Q' \leq Q$. That is, $Q \cong Q'$. Hence, we may construct the following commutative diagram 	
	\[
	\begin{tikzcd}
		P \arrow[rr, "f"] \arrow[d, dashed] && Q \arrow[rr] \arrow[d, "\sim" labl1] && X_{P, Q} \arrow[rr] \arrow[d, equals] && P[1] \arrow[d] \\ 
		P' \arrow[rr, "g"] && Q' \arrow[rr] && X_{P, Q} \arrow[rr] && P[1] \\ 
	\end{tikzcd}
	\]
	showing that $P \cong P'$. 
\end{proof}

\begin{lemma}\label[lem]{lem:w Pa factors to X}
	Let $\cT$ be a triangulated category satisfying \Cref{setup:w cluster cat} and equipped with a $w$-linear generator $G$. Suppose that $P, Q, R, S$ are indecomposable objects in $\lang[1]{G}$ such that $Q \leq S \leq R < P$. Then, a morphism $f \colon R \rightarrow X_{P, Q}$ factors through S.
\end{lemma}

\begin{proof}
	Applying the functor $\hcT{R}{-}$ to the triangle
	\[
	P \rightarrow Q  \xrightarrow{g} X_{P,Q} \rightarrow P[1],
	\]
	induces the exact sequence
	\[
	\hcT{R}{P} \rightarrow \hcT{R}{Q} \xrightarrow{\hcT{R}{g}} \hcT{R}{X_{P,Q}} \rightarrow \hcT{R}{P[1]}.
	\]
	Since $R < P$, it follows from the definition of a $w$-linear generator that $\hcT{R}{P} = \hcT{R}{P[1]} =0$, hence $\hcT{R}{g}$ is an isomorphism. Therefore, any morphism $f \colon R \rightarrow X_{P,Q}$ factors through a morphism $q \colon R \rightarrow Q$. Moreover, because $Q \leq S \leq R$, then property \ref{wG3} in \Cref{def:w linear gen} implies that $q$ (and thus $f$) also factors through $S$.
\end{proof}

An analogous proof to \Cref{lem:w Pa factors to X}, with the application of the functor $\hcT{-}{S}$ proves the following.

\begin{lemma}\label[lem]{lem:w X factors to Pa}
	Let $\cT$ be a triangulated category satisfying \Cref{setup:w cluster cat} and equipped with a $w$-linear generator $G$. Suppose that $P, Q, R, S$ are indecomposable objects in $\lang[1]{G}$ such that $Q[1] < S \leq R \leq P[1]$. Then, a morphism $f \colon X_{P, Q} \to S$ factors through $R$.
\end{lemma}

These lemmas allow us to completely describe the indecomposable objects in $\lang[2]{G}$ in terms of indecomposable objects $\lang[1]{G}$ and the cones $X_{P, Q}$.

\begin{proposition}\label[prop]{prop:w Ind in langle 2}
	Let $\cT$ be a triangulated category satisfying \Cref{setup:w cluster cat} and equipped with a $w$-linear generator $G$. Suppose that $X \in \lang[2]{G}$ is an indecomposable object.
	Then $X$ is either isomorphic to an object $P \in \lang[1]{G}$, or to $X_{P, Q}$, where $P, Q$ are indecomposable objects in $\lang[1]{G}$.
\end{proposition}

\begin{proof}
	It suffices to show that any morphism $f \colon P \rightarrow Q \in \lang[1]{G}$ is such that $\cone{g} \cong A \oplus B$, with $A \in \lang[1]{G}$ and $B \cong \bigoplus_{(P', Q') \in I} X_{P',Q'}$ where $I$ is a set containing tuples $(P', Q')$ of indecomposable objects in $\lang[1]{G}$. We will do so by considering three cases.
	Throughout, consider that $R \not\leq P \not\leq R$ and $Q < P$ implies that $R \not\leq Q \not\leq R$ for all $R,P,Q \in \lang[1]{G}$, and so may only consider indecomposable objects in $\lang[1]{G}$ that are related in the strict partial ordering.
	
	\begin{enumerate}
		\item \label{wPgtP1} Consider first a morphism $f \colon P \rightarrow \bigoplus_{i=1}^n P_i$, where $P$ and $P_i$ are indecomposable objects in $\lang[1]{G}$. We may write $f = (f_i)_{i =1}^n$. If there is a $j$ where $f_j \colon P \to P_j$ is the zero morphism, then $\cone{f} \cong \cone{(f_i)_{i \neq j}} \oplus P_j$. Hence, we may reduce to the case where none of the $f_i$ are the zero morphism. Thus, by reindexing if necessary, we may assume that $P_i \leq P_1 \leq P$ for all $1 \leq i \leq n$. By property \ref{wG3} in \Cref{def:w linear gen}, it follows that $f_i \colon P \to P_i$ factors through a morphism $g_i \colon P \to P_1$ for  all $2 \leq i \leq n$. By rescaling the morphisms $g_i$ if necessary, we may assume $f_i = g_i \circ f_1$. We may thus construct the following morphism of triangles.
		\[
		\begin{tikzcd}[column sep=4.5em]
			P \arrow[r,"{\begin{pmatrix} f_1 \\ 0 \end{pmatrix}}"] \arrow[d,"\sim" labl1 ,swap] & \bigoplus_{i =1}^n P_{i} \arrow[r] \arrow[d,"h"] & X_{P, P_1} \oplus \bigoplus_{i =2}^n P_{i} \arrow[r]\arrow[d] & P_1[1] \arrow[d] \\
			P \arrow[r,"f"] & \bigoplus_{i=1}^n P_{i} \arrow[r] & \mathrm{cone} \, f \arrow[r]& P[1].
		\end{tikzcd}
		\]
		Here, $h = \begin{pmatrix}
			\mathrm{id}_{P} & 0 \\
			g & \mathrm{id}_{\bigoplus_{i=2}^n P_i}
		\end{pmatrix}$ and $g = (g_i)_{i=2}^n : P \rightarrow \bigoplus_{i=2}^n P_i$. Observe that $h$ is an isomorphism, as it admits the inverse $h^{-1} = \begin{pmatrix}
			\mathrm{id}_{P} & 0 \\
			-g & \mathrm{id}_{\bigoplus_{i=2}^n P_i}
		\end{pmatrix}$. We may thus conclude that $\cone{f} \cong X_{P, P_1} \oplus \bigoplus_{i=2}^n P_i$.
		
		\item Next, consider a morphism $f = (f_{ij})_{i=1, j=1}^{m, n}: \bigoplus_{j=1}^m Q_{j}  \rightarrow \bigoplus_{i=1}^n P_{i}$, where $f_{ij} \colon Q_{j} \rightarrow P_i$ is a morphism between indecomposable objects in $\lang[1]{G}$ for all $i, j$. We will show by induction on $m$ that $\cone{f} \cong A \oplus B$ where $A \in \lang[1]{G}$ and $B \cong \bigoplus_{(P',Q') \in I} X_{P', Q'}$. Observe that \eqref{wPgtP1} proves the statement for $m=1$, and suppose by induction that it holds for $m-1$.
		
		By reindexing if necessary, assume that $Q_j \leq Q_1$ for all $1 \leq j \leq n$. If there is an $i_0$ such that $P_{i_0} > Q_1$, then $f_{i_0, j} = 0$ for all $j$ and $\cone f \cong \cone (f_{ij})_{i \neq i_0, j} \oplus P_{i_0}$. Hence, we may reduce to the case where $P_i, Q_j \leq Q_1$ for all $1 \leq i \leq n$ and $1 \leq j \leq m$. 
		Consider the diagram 
		\[
		\begin{tikzcd}
			Q_{1}  \arrow[r, "{\begin{pmatrix} 1 \\ 0 \end{pmatrix}}"] \arrow[d,"\sim" labl1 ,swap] & \bigoplus_{j =1}^m Q_{j}  \arrow[r] \arrow[d,"f",swap] & \bigoplus_{j =2}^n Q_{j}  \arrow[r] \arrow[d,"z \oplus v", dashed] & Q_{1}[1] \arrow[d]\\
			Q_{1}  \arrow[r,"{(f_{i1})_{i=1}^n}"] \arrow[d]  & \bigoplus_{i=1}^n P_{i} \arrow[r] \arrow[d] & X_{Q_1,P_1}  \oplus \bigoplus_{i=2}^n P_{i} \arrow[d, dashed] \arrow[r] & Q_{1}[1] \arrow[d]\\
			0 \arrow[r, dashed] \arrow[d] & \mathrm{cone} \, f \arrow[r,"\sim", dashed] \arrow[d] & \mathrm{cone} \, (z \oplus v) \arrow[r, dashed] \arrow[d,"u", dashed] & 0\arrow[d]\\
			Q_{1}[1] \arrow[r]  & \bigoplus_{j=1}^m Q_{j}[1] \arrow[r] & \bigoplus_{j=2}^m Q_j[1] \arrow[r]  & Q_{1}[2].
		\end{tikzcd}
		\]
		The top left square of this diagram clearly commutes. The bottom two rows in this diagram, as well as the dashed morphisms, are obtained via the octahedral axiom. In fact, it follows from this axiom that the diagram commutes and all four columns and rows are distinguished triangles. The fact that the $\cone{(f_{i1})_{i=1}^n} = X_{Q_1,P_1} \oplus \bigoplus_{i=2}^n P_i$ follows from \eqref{wPgtP1} above. 
		
		Furthermore, observe that $z$ can be written as $z = (z_{j} \colon Q_j \to X_{Q_1, P_1})_{j=1}^m$ and, by \Cref{lem:w Hom P to X}, $\hcT{Q_j}{X_{Q_1,P_1}} = 0$ since $Q_j \leq Q_1$. Thus, $z = 0$, and so there is an isomorphism $\cone{f} \cong \cone{v} \oplus X_{Q_1,P_1}$. However, by the inductive hypothesis, the cone of the morphism $v \colon \bigoplus_{j=2}^m Q_j \to \bigoplus_{i=2}^n P_i$ is of the required form. Hence, our claim follows by induction. \qedhere
	\end{enumerate}
\end{proof}

\medskip
In view of \Cref{prop:w Ind in langle 2}, it is possible to separate nonzero morphisms between indecomposable objects in $\lang[2]{G}$ into two classes, which will play an important role in controlling the behaviour of the Hom-spaces.

\begin{definition}
	Let $\cT$ be a triangulated category satisfying \Cref{setup:w cluster cat} and equipped with a $w$-linear generator $G$. Let $Y,Z \in \lang[2]{G}$ be indecomposable objects, and let $f: Y \rightarrow Z$ be a non-zero morphism. Then, $f$ is a \textit{forwards morphism} if there exists a pair of triangles
	\[
	\begin{tikzcd}
		P \arrow[d,dashed, "f_1"] \arrow[rr] && Q \arrow[d,dashed, "f_0"] \arrow[rr] && Y \arrow[d,"f"] \arrow[rr] && P[1] \arrow[d,dashed]\\
		P' \arrow[rr] && Q' \arrow[rr] && Z \arrow[rr] && P'[1],
	\end{tikzcd}
	\]
	where $P,P',Q,Q' \in \lang[1]{G}$ and a pair of morphisms $f_1$ and $f_0 \neq 0$ making the above diagram a morphism of triangles. A morphism which is not a forwards morphism is called a \textit{backwards morphism}.
\end{definition}

\begin{lemma}\label[lem]{lem:w Forward and Back morphisms}
	Let $\cT$ be a triangulated category satisfying \Cref{setup:w cluster cat} and equipped with a $w$-linear generator $G$. Suppose that $P \in \mathrm{ind}\, \lang[1]{G}$ and $X \in \mathrm{ind}\, \lang[2]{G}$. Then,
	\begin{enumerate}
		\item a nonzero morphism $f \colon P \rightarrow X$ is a forwards morphism, \label{w forward from P}
		\item a nonzero morphism $g: X \rightarrow P$ is a backwards morphism if and only if $X \not\in \mathrm{ind}\, \lang[1]{G}$. \label{w backward to P}
	\end{enumerate}
\end{lemma}
\begin{proof}
	We begin by proving (1). By \Cref{lem:w X factors to Pa}, it suffices to consider the cases $f \colon P \to Q$ or $f \colon P \to X_{R, Q}$ with $P, Q, R$ indecomposable objects in $\lang[1]{G}$.
	Suppose $X \cong Q$, then the natural morphism of triangles
	\[
	\begin{tikzcd}
		0 \arrow[d] \arrow[rr] && P \arrow[d,"f"] \arrow[rr,"1"] && P \arrow[d,"f"] \arrow[rr] && 0 \arrow[d]\\
		0 \arrow[rr] && Q \arrow[rr,"1"] && Q \arrow[rr] && 0,
	\end{tikzcd}
	\]
	proves that $f$ is a forwards morphism.
	
	Suppose next that $X \cong X_{R, Q}$. Since $f \colon P \to X_{R, Q}$ is a nonzero morphism, then $R < P \leq Q$ by \Cref{lem:w Hom P to X}. Observe that by\ \Cref{lem:w Pa factors to X}, there are nonzero morphisms $g \colon P \to Q$ and $h \colon Q \to X_{R, Q}$ such that $f = h \circ g$. Moreover, by rescaling if necessary, we may assume that $h \colon Q \to X_{R, Q}$ is the morphism in the triangle that defines $X_{R, Q}$. Hence, we may construct the following commutative diagram
	\[
	\begin{tikzcd}
		0 \arrow[d] \arrow[rr] && P \arrow[d, "g"] \arrow[rr,"1"] && P \arrow[d,"f"] \arrow[rr] && 0 \arrow[d]\\
		R \arrow[rr] && Q \arrow[rr, "h"] && X_{R,Q} \arrow[rr] && R[1],
	\end{tikzcd}
	\]
	which is a morphism of triangles, proving that $f$ is a forwards morphism.
	
	To prove (2), consider a nonzero morphism $g: X \rightarrow P$.
	If $X \in \lang[1]{G}$, then, since $X$ is indecomposable, we are in the situation of (1), and we have already proved that in such cases $g$ is a forwards morphism.
	For the converse, suppose that $X \not\in \lang[1]{G}$. Then, $X \cong X_{R, Q}$ for indecomposable objects $R, Q \in \lang[1]{G}$ by \Cref{lem:w X factors to Pa}. By \Cref{cor:w unique triangle}, it suffices to check that there does not exist a morphism of triangles as in the diagram 
	\[
	\begin{tikzcd}
		R \arrow[d,dashed] \arrow[rr] && Q \arrow[d, "h", dashed] \arrow[rr] && X_{R,Q} \arrow[d,"g"] \arrow[rr] && R[1] \arrow[d,dashed]\\
		0 \arrow[rr] && P \arrow[rr,"1"] && P \arrow[rr] && 0.
	\end{tikzcd}
	\]
	with $h \neq 0$. Suppose for a contradiction that such diagram exists. Because $g$ is nonzero, then $Q[1] < P \leq R[1]$ by \Cref{lem:w Hom X to P}, and so $Q < P$. That is, $h = 0$, since there does not exist a nonzero morphism $Q \to P$. This contradicts the assumption that $g$ is a forwards morphism, whence it is a backwards morphism. 
\end{proof}

Due to \Cref{lem:w Forward and Back morphisms}, it is possible to complete the classification of Hom-spaces between indecomposable objects in $\cT$.

\begin{lemma}\label[lem]{lem:w Morphisms X to X}
	Let $\cT$ be a triangulated category satisfying \Cref{setup:w cluster cat} and equipped with a $w$-linear generator $G$. Let $P, Q, R, S$ be indecomposable objects in $\lang[1]{G}$. Then, $\hcT{X_{P,Q}}{X_{R,S}} \cong k$ if and only if one of the following holds;
	\begin{enumerate}
		\item \label{En:w Morph 1} $S \leq Q < R \leq P$,
		\item \label{En:w Morph 2} $Q[1] < S \leq P[1] < R$.
	\end{enumerate}
	Else $\hcT{X_{P,Q}}{X_{R,S} } =0$.
	Moreover, a morphism of type (\ref{En:w Morph 1}) is a forwards morphism, and a morphism of type (\ref{En:w Morph 2}) is a backwards  morphism.
\end{lemma}

\begin{proof}
	Applying the functor $\hcT{-}{X_{R,S}}$ to the triangle
	\[ P \xrightarrow{\alpha} Q \xrightarrow{\beta} X_{P, Q} \rightarrow P[1]\]
	defining $X_{P,Q}$ induces the long exact sequence
	\begin{align} \label{les:w XPQ XRS} 
		\hcT{Q[1]}{X_{R,S} } \xrightarrow{a} \hcT{P[1]}{X_{R,S} } \rightarrow \hcT{X_{P,Q}}{X_{R,S} } \rightarrow \hcT{Q}{X_{R,S} } \xrightarrow{b} \hcT{P}{X_{R,S} }. 
	\end{align} 	
	Again, we consider that by the properties of a long exact sequence, $\hcT{X_{P,Q}}{X_{R,S}} \neq 0$ if and only if at least one of $a$ is not surjective and $b$ is not injective holds.
	
	Consider the morphism $a: \hcT{Q[1]}{X_{R,S} } \rightarrow \hcT{P[1]}{X_{R,S} }$, if $\hcT{P[1]}{X_{R,S} } = 0$ then $a$ is automatically surjective, so we also consider when $\hcT{P[1]}{X_{R,S} } \cong k$, which by \Cref{lem:w Hom P to X} implies $S \leq P[1] <R$.
	However, if $\hcT{Q[1]}{X_{R,S}} \cong k$, then we have $S \leq Q[1] < P[1] < R$, which implies that $a$ is an isomorphism by \Cref{lem:w Pa factors to X}.
	Hence we must only consider when $\hcT{Q[1]}{X_{R,S}} = 0$, which implies that $Q[1] < S \leq P[1] <R$.
	
	As $Q[1] < S \leq P[1] <R$, then either $Q < Q[1] < S < R$ if $w=1$, or $Q$ is not related to $R,S$ if $w>2$, and so $\hcT{Q}{X_{R,S}} =0$.
	Therefore, $\hcT{X_{P,Q}}{X_{R,S}} \cong k$ if $Q[1] < S \leq P[1] <R$.
	To see that $v \in \hcT{X_{P,Q}}{X_{R,S}}$ is a backwards morphism if $Q[1] < S \leq P[1] <R$, consider the following diagram:
	\[
	\begin{tikzcd}
		P \arrow[d,dashed] \arrow[rr] && Q \arrow[d,dashed,"f_0"] \arrow[rr] && X_{P,Q} \arrow[d,"v"] \arrow[rr] && P[1] \arrow[d,dashed]\\
		R \arrow[rr] && S \arrow[rr] && X_{R,S} \arrow[rr] && R[1].
	\end{tikzcd}
	\]
	It is clear that $f_0 =0$, as either $Q<R$ if $w=1$, or $Q$ and $R$ are not related, hence $v$ is not a forwards morphism, and so must be a backwards morphism.
	
	An analogous argument holds when considering the morphism $b$ to see that $\hcT{X_{P,Q}}{X_{R,S}} \cong k$ if $R \leq Q < S \leq P$.
	In this case, we see that a morphism $u \in \hcT{X_{P,Q}}{X_{R,S}}$ is a forwards morphism via the commutative diagram,
	\[
	\begin{tikzcd}
		P \arrow[d,dashed] \arrow[rr] && Q \arrow[d,dashed,"f_0"] \arrow[rr] && X_{P,Q} \arrow[d,"v"] \arrow[rr] && P[1] \arrow[d,dashed]\\
		R \arrow[rr] && S \arrow[rr] && X_{R,S} \arrow[rr] && R[1],
	\end{tikzcd}
	\]
	where $f_0$ is nonzero as $R \leq Q$.
\end{proof}

\begin{remark}\label[rem]{rem:w unique up to canonical iso}
	As a corollary of \Cref{lem:w Morphisms X to X}, we see that there is a canonical isomorphism between cones of a morphism $P \to Q \in \lang[1]{G}$, when $P, Q$ are indecomposable. Consider the morphism of triangles
	\[ \begin{tikzcd}
		P \arrow[d, equal] \arrow[r, "f"] & Q \arrow[r, "g"] \arrow[d, equal] & X \arrow[r, "h"] \arrow[d, "u", dashed, xshift=-0.5em]  \arrow[d, "v", dashed, xshift=0.5em] & P[1] \arrow[d,equal] \\
		P \arrow[r, "f"] & Q \arrow[r, "g'"] & X' \arrow[r, "h'"] & P[1].
	\end{tikzcd}
	\]
	Since $\hcT{X}{X'}$ is one dimensional, there is a $\lambda \in k^\times$ such that $v = \lambda u$. However, commutativity of the diagram implies that $\lambda = 1$.
\end{remark}

\subsection{Rouquier Dimension}

This section uses the properties of $w$-linear generators established in the previous sections, in order to compute the Rouquier dimension of a triangulated category admitting a $w$-linear generator. 

\begin{lemma}\label[lem]{lem:w cone Pa into Z}
	Let $\cT$ be a triangulated category satisfying \Cref{setup:w cluster cat} and equipped with a $w$-linear generator $G$. Let $P$ be an indecomposable object in $\lang[1]{G}$ and let $f \colon P \rightarrow Z$ be a nonzero morphism with $Z \in \lang[2]{G}$. Then, $\cone{f} \in \lang[2]{G}$. 
\end{lemma}

\begin{proof}
	Since $\cT$ is Krull-Schmidt, we may write $Z= \bigoplus_{i=1}^n Z_i$ where $Z_i$ are indecomposable objects in $\lang[2]{G}$. Moreover, we may write $f = (f_i \colon P \rightarrow Z_i)_{i=1}^n$, and we may assume without loss of generality that $f_i \colon P \rightarrow Z_i$ is nonzero for all $i$. 
	
	By \Cref{prop:w Ind in langle 2}, $Z_i \cong Q \in \lang[1]{G}$ or $Z_i \cong X_{R, S}$ for $R, S \in \lang[1]{G}$. For each $1 \leq i \leq n$, let $Q_i = Z_i$ if $Z_i \in \lang[1]{G}$. Otherwise, let $Q_i = S$ if $Z_i \cong X_{R,S}$. Similarly, for each $1 \leq i \leq n$, let $P_i = 0$ if $Z_i \in \lang[1]{G}$ and let $P_i = R$ if $Z_i \cong X_{R, S}$. Observe, then, that there is a natural morphism $h = (h_i)_{i=1}^n \colon \bigoplus_{i=1}^n Q_i \to Z$ whose cone is $\bigoplus_{i=1}^n P_i[1]$.
	
	For all $1 \leq i \leq n$, we claim that there are nonzero morphisms $g_i \colon P \to Q_i$ such that $h \circ (g_i)_{i=1}^n = f$. To construct such $g_i$, if $Z_i \in \lang[1]{G}$, let $g_i = f_i$, which is nonzero by assumption. On the other hand, if $Z_i = X_{R, S}$, then, by \Cref{lem:w Pa factors to X}, there is a nonzero morphism $g_i \colon P \to S$ such that $f = h_i \circ g_i$. Thus, an application of the octahedral axiom induces the commutative diagram
	\[
	\begin{tikzcd}[column sep=4em]
		P \arrow[d,"\sim" labl1] \arrow[r, "{g = (g_i)_{i=1}^n}"] & \bigoplus_{i =1}^n Q_{i}  \arrow[d, "h"] \arrow[r] & X_{P,Q_1}  \oplus \bigoplus_{i=2}^n Q_{i}  \arrow[d] \arrow[r] & P[1] \arrow[d]\\
		P \arrow[d] \arrow[r,"f"] & Z \arrow[d] \arrow[r] & \mathrm{cone} \, f \arrow[d] \arrow[r] & P[1] \arrow[d]\\
		0 \arrow[d] \arrow[r] & \bigoplus_{i=1}^n P_{i}[1] \arrow[d] \arrow[r] & \bigoplus_{i=1}^n P_{i}[1] \arrow[d,"{(u,v)}"] \arrow[r] & 0 \arrow[d]\\
		P[1] \arrow[r] & \bigoplus_{i=1}^n Q_{i}[1] \arrow[r] & X_{P,Q_1} [1] \oplus \bigoplus_{i=2}^n Q_{i}[1] \arrow[r] & P[2].
	\end{tikzcd}
	\]
	which has exact rows and columns. The fact that $\cone{g} \cong X_{P,Q_1} \oplus \bigoplus_{i=2}^n Q_i$ follows from \Cref{prop:w Ind in langle 2} \ref{wPgtP1}. If $u =0$, then $\cone f \cong X_{P, Q_1} \oplus \cone v$, where $v \colon \bigoplus_{i=1}^n P_i[1] \to \bigoplus_{i=2}^n Q_i[2]$ is a morphism in $\lang[1]{G}$, whence $\cone v \in \lang[2]{G}$ by definition of $\lang[2]{G}$. That is, if $u=0$, then $\cone f \cong X_{P, Q_1} \oplus \cone v \in \lang[2]{G}$.
	
	To see that $u = (u_i \colon P_i[1] \to X_{P, Q_1})_{i=1}^n = 0$, observe that, if $Z_i \in \lang[1]{G}$, then $P_i = 0$ by construction, and so $u_i =0$. Otherwise, $P_i = R$, and the assumption that $f_i \colon P \to X_{R,S}$ is nonzero means that $S \leq P < R = P_i$ by \Cref{lem:w Hom P to X}. By the same lemma and since $P < P_i$ implies that $P[1] \leq P_i$ if $w=1$, and $P[1]$ is not related to $P_i$ if $w > 1$, then $\hcT{P_i[1]}{X_{P, Q_1}} = 0$. Whence, $u_i = 0$. 
\end{proof}

\begin{theorem}\label{Thm:w Perf 2 is Tri}
	Let $\cT$ be a triangulated category satisfying \Cref{setup:w cluster cat}, and equipped with a $w$-linear generator $G$. Then, category $\langle G \rangle_2$ is triangulated, and hence equivalent to $\cT$. Equivalently, the generation time of $G$ is at most $1$.
\end{theorem}

\begin{proof}
	To show that $\lang[2]{G}$ is triangulated, it suffices to show that $\lang[3]{G} = \lang[2]{G}$. That it is, it is enough to show that the cone of any morphism $f \colon P \rightarrow Z$ with $P \in \lang[1]{G}$ and $Z \in \lang[2]{G}$ is in $\lang[2]{G}$.
	Since $\cT$ is Krull-Schmidt, we may write $P = \bigoplus_{i=1}^n P_i$, where $P_i$ are indecomposable objects, and $f = (f_i \colon P_i \rightarrow Z)_{i=1}^n$. We will show that $\cone{f} \in \lang[2]{G}$ by induction on $n$. The $n=1$ case is \Cref{lem:w cone Pa into Z}. Suppose by induction that $\cone{f'} \in \lang[2]{G}$ for all $f' \colon Q \to Z$ where $Q = \bigoplus_{i=1}^{n-1} Q_i$ for indecomposable objects $Q_i \in \lang[1]{G}$.
	
	By reindexing if necessary, assume $P_i \leq P_1$ for all $1 \leq i \leq n$. By applying the octahedral axiom, it is possible to construct the following commutative diagram with exact rows and columns.	
	\[
	\begin{tikzcd}
		P_1  \arrow[d,"\sim" labl1] \arrow[r, "{(1, 0, \hdots)}"] & P \arrow[d,"f"] \arrow[r] & P/P_{1}  \arrow[d, "h"] \arrow[r] & P_{1}[1] \arrow[d]\\
		P_{1}  \arrow[d] \arrow[r,"{g=(f_1, 0, \hdots)}"] & Z \arrow[d] \arrow[r] & \mathrm{cone} \, g \arrow[d] \arrow[r] & P[1] \arrow[d]\\
		0 \arrow[d] \arrow[r] & \mathrm{cone} \, f \arrow[d] \arrow[r,"\sim"] & \mathrm{cone} \, h \arrow[d] \arrow[r] & 0 \arrow[d]\\
		P_{1}[1] \arrow[r] & P[1] \arrow[r] & P/P_{1}[1] \arrow[r] & P_{1}[2].
	\end{tikzcd}
	\]
	By \Cref{lem:w cone Pa into Z}, $\mathrm{cone} \, g \in \lang[2]{G}$. By the inductive hypothesis, since $P/P_1$ has exactly $n-1$ indecomposable summands, $\cone{h} \in \lang[2]{G}$. Since $\cone{f} \cong \cone{h}$, it follows that $\cone{f} \in \lang[2]{G}$.
	
	Therefore $\lang[2]{G}$ is the smallest triangulated subcategory of $\cT$ containing $G$. Hence $\lang[2]{G} = \cT$.
\end{proof}

As an immediate corollary of \Cref{Thm:w Perf 2 is Tri}, we are able to compute the Roquier dimension of certain triangulated categories equipped with a $w$-linear generator. Recall that the \textit{Orlov spectrum} $\cO(\cS)$ of a triangulated category $\cS$ is the collection of generation times across all classical generators in $\cS$ \cite{OrlovSpectra}. Further, the \textit{Rouquier dimension} of a triangulated category is the infimum of $\cO(\cS)$ \cite{Rouquier}.

\begin{corollary}\label[cor]{cor:w Rouquier dimension}
	Let $\cT$ be a triangulated category satisfying \Cref{setup:w cluster cat}, and equipped with a $w$-linear generator $G$. Then, $\cT$ has Rouquier dimension 1 if and only if $G \not\cong G[w]$, else the Rouquier dimension of $\cT$ is 0.
\end{corollary}

\begin{proof}
	Suppose that $G \not\cong G[w]$. In view of \Cref{Thm:w Perf 2 is Tri}, it is suffices to show that the Rouquier dimension of $\cT$ is greater than zero. That is, it suffices to show that there does not exist an object $G' \in \cT$ such that $\lang[1]{G'} = \cT$.
	
	Observe first that, since $\cT$ is Krull-Schmidt, $G'$ is a finite direct sum of indecomposable objects in $\cT$, and so $\lang[1]{G'}$ has finitely many indecomposable objects up to suspension. However, when $G \not\cong G[w]$, then we can fix $P \in \lang[1]{G}$ such that $P \not\cong P[1]$, which in turn, implies that $P[i-w] < P[i]$ for all $i \in \bZ$. Thus, by \Cref{cor:w unique triangle} the collection of objects $\{X_{P,P[iw]}\}_{i <0}$ is infinite. Again, by \Cref{cor:w unique triangle}, since $P$ is fixed, then $X_{P,P[iw]} \not\cong X_{P,P[i'w]}[j]$, for all $i,i' <0$ and $j \in \bZ$. Hence, the collection $\{X_{P,P[iw]}\}_{i < 0}$ is an infinite collection of indecomposable objects which are not isomorphic and are not suspensions of each other. It thus follows that this collection is not in $\lang[1]{G'}$, from which we conclude that $\lang[1]{G'} \subsetneq \cT$ for all objects $G' \in \cT$, and so the Rouquier dimension of $\cT$ is greater than zero.
	
	Now consider the case $G \cong G[w]$. Then, $\lang[1]{G}$ has finitely many isomorphism classes of indecomposable objects. Hence, it follows from \Cref{prop:w Ind in langle 2} and \Cref{Thm:w Perf 2 is Tri} that $\cT$ has finitely many isomorphism classes of indecomposable objects $\{P_i\}_{i=1}^n$. Hence, by letting $G' = \bigoplus_{i=1}^n P_i$, then $\cT = \lang[1]{G'}$ and we may conclude that $\cT$ has Rouquier dimension 0 if there exists a $w$-linear generator $G \cong G[w]$.
\end{proof}

\begin{corollary}\label[cor]{cor:w Ind in a tri with PQ}
	Let $\cT$ be a triangulated category satisfying \Cref{setup:w cluster cat}, and equipped with a $w$-linear generator $G$.
	For each object $X \in \mathrm{ind}\, \cT$, there exists a triangle
	\[
	P \rightarrow X \rightarrow Q \rightarrow P[1],
	\]
	such that $P,Q \in \lang[1]{G}$, with $P$ indecomposable, and $Q$ either indecomposable or trivial.
\end{corollary}

\begin{proof}
	Suppose $X \in \lang[1]{G}$, then we may take $X \cong P$ and $Q=0$ to obtain the identity triangle for $X$.
	Now suppose that $X \not\in \lang[1]{G}$, then $X \cong X_{Q[-1],P}$ by \Cref{prop:w Ind in langle 2} and \Cref{Thm:w Perf 2 is Tri}.
\end{proof}

\subsection{Cones} \label{subsec:ww cones}

Throughout this section, let $\cT$ be a triangulated category satisfying \Cref{setup:w cluster cat}, and equipped with a $w$-linear generator $G$. The properties of $w$-linear generators are strong enough to afford control over cones of morphisms between indecomposable objects in $\cT$. In this section, we present such results. 

\begin{lemma}\label[lem]{lem:w P to X cone}
	Let $Q \leq R < P$ be indecomposable objects in $\lang[1]{G}$, and let $g: R \rightarrow X_{P,Q} $ be a non-zero morphism.
	Then, $\mathrm{cone} \, g \cong P[1] \oplus X_{R,Q}$.
\end{lemma}

\begin{proof}
	Consider the diagram
	\[
	\begin{tikzcd}
		R \arrow[d,"\sim" labl1 ] \arrow[r] & Q \arrow[d] \arrow[r] & X_{R,Q} \arrow[d] \arrow[r] & R[1] \arrow[d]\\
		R \arrow[d] \arrow[r,"g"] & X_{P,Q}  \arrow[d] \arrow[r] & \mathrm{cone} \, g \arrow[d] \arrow[r] & R[1] \arrow[d]\\
		0 \arrow[d] \arrow[r] & P[1] \arrow[d] \arrow[r,"\sim"] & P[1] \arrow[d,"u"] \arrow[r] & 0 \arrow[d]\\
		R[1] \arrow[r] & Q[1] \arrow[r] & X_{R,Q}[1] \arrow[r] & R[2].
	\end{tikzcd}
	\]
	where the top leftmost square exists and commutes by \Cref{lem:w Pa factors to X}. Thus, the diagram commutes and has exact rows and columns by the octahedral axiom. Moreover, the morphism $u$ is trivial, as $\hcT{P[1]}{X_{R,Q}[1]}= 0$ by \Cref{lem:w Hom P to X}, hence $\mathrm{cone} \, g \cong P[1] \oplus X_{R,Q}$.
\end{proof}

\begin{lemma}\label[lem]{lem:w X to P cone}
	Let $Q[1] < R \leq P[1]$, and let $g: X_{P,Q}  \rightarrow R$ be a non-zero morphism.
	Then $\mathrm{cone} \, g \cong Q[1] \oplus X_{P[1],R}$.
\end{lemma}

\begin{proof}
	By \Cref{lem:w X factors to Pa}, the top leftmost square in the diagram 
	\[
	\begin{tikzcd}
		X_{P,Q}  \arrow[r] \arrow[d,"\sim" labl1 ] & P[1] \arrow[d,"h"] \arrow[r] & Q[1] \arrow[d] \arrow[r] & X_{P,Q} [1] \arrow[d] \\
		X_{P,Q}  \arrow[r,"g"] \arrow[d] & R \arrow[d] \arrow[r] & \mathrm{cone} \, g \arrow[d] \arrow[r] & X_{P,Q} [1] \arrow[d]\\
		0 \arrow[r] \arrow[d] & X_{P[1],R} \arrow[r,"\sim"] \arrow[d] & X_{P[1],R} \arrow[r] \arrow[d] & 0 \arrow[d]\\
		X_{P,Q} [1] \arrow[r] & P[2]  \arrow[r] & Q[2] \arrow[r] & X_{P,Q} [2].
	\end{tikzcd}
	\]
	exists and commutes. Thus, the diagram exists, commutes and has exact rows and columns by the octahedral axiom. 	
	
	Since $\hcT{X_{P[1],R}}{Q[2]}=0$ by \Cref{lem:w Hom X to P}, then $\mathrm{cone} \, g \cong Q[1] \oplus X_{P[1],R}$.
\end{proof}

Understanding the cone of a morphism from $R \to X_{P,Q}$ allows us to establish further factorization properties. 

\begin{lemma} \label[lem]{lem:w P to Q via X}
	Let $P, Q \in \lang[1]{G}$ and $X_{R,S} \in \cT$ be indecomposable objects. Then, a nonzero morphism $f \colon P \to Q$ does not factor through $X_{R,S}$.
\end{lemma}

\begin{proof}
	First, note $f$ cannot factor through $X_{R,S}$ if either morphism $g \colon P \to X_{R,S}$ or $h \colon X_{R,S} \to Q$ are zero morphisms.
	Hence, we may assume that such nonzero morphisms exists, and by \Cref{lem:w Hom P to X} and \Cref{lem:w Hom X to P} that $S[1] < Q \leq R[1]$ and $S \leq P < R$.
	However, if $w \neq 1$, this is absurd as $Q < P$ by assumption, and combining this with both of the above inequalities holding would imply that $S[1] < R$, a contradiction of the definition of a $w$-linear generator. 
	
	Now we suppose that $w=1$, and we apply the functor $\hcT{-}{Q}$ to the triangle 
	\[
	P \xrightarrow{g} X_{R,S} \rightarrow R[1] \oplus X_{P,S} \rightarrow P[1],
	\]
	induced by \Cref{lem:w P to X cone} yields the exact sequence	
	\[
	\hcT{P[1]}{Q} \rightarrow \hcT{R[1] \oplus X_{P,S}}{Q} \rightarrow \hcT{X_{R,S}}{Q} \xrightarrow{\hcT{g}{Q}} \hcT{P}{Q}.
	\]
	By \Cref{def:w linear gen}, \Cref{lem:w Hom X to P}, and \Cref{lem:w Morphisms X to X} this simplifies to the following exact sequence of vector spaces.
	\[ k \rightarrow k^2 \rightarrow k \xrightarrow{\hcT{g}{Q}} k \]
	A dimensional argument implies that this is only possible if $\hcT{g}{Q} = 0$. That is, for any pair of maps $g \colon P \to X_{R,S}$ and $h \colon X_{R,S} \to Q$, then $h \cdot g = 0 \neq f$, and we may conclude that $f$ does not factor through $W$.
\end{proof}

\begin{lemma}\label[lem]{lem:w P to X via X}
	Let $R,X_{P,Q}  \in \cT$ be indecomposable objects, such that there exists a non-zero morphism $f: R \rightarrow X_{P,Q} $.
	Then, $f$ factors through $X_{P',Q'}$ if and only if $Q \leq Q' \leq R < P \leq P'$.
\end{lemma}

\begin{proof}
	Suppose first that $Q \leq Q' \leq R < P \leq P'$. Then, by \Cref{lem:w Hom P to X}, there is a nonzero morphism $g \colon R \to X_{P', Q'}$. Moreover, by \Cref{lem:w P to X cone}, there is a triangle
	\[
	R \xrightarrow{g} X_{P',Q'} \rightarrow P'[1] \oplus X_{R,Q'} \rightarrow R[1].
	\]
	By applying the functor $\hcT{-}{X_{P,Q}}$ to this triangle, we obtain the long exact sequence
	\begin{equation} \label{les:w X factors through X}
		\begin{tikzcd}
			\hcT{P'[1] \oplus X_{R,Q'}}{X_{P,Q} }  \arrow[d, phantom,""{coordinate,name=Z}] \arrow[r] & \hcT{X_{P',Q'}}{X_{P,Q}} \arrow[r, "\hcT{g}{X_{P,Q}}"]  &  \hcT{R}{X_{P,Q}} \arrow[dll, rounded corners, to path={ -- ([xshift=2ex]\tikztostart.east) |- (Z) [near end]\tikztonodes -| ([xshift=-2ex]\tikztotarget.west) -- (\tikztotarget)}] \\
			\hcT{P' \oplus X_{R,Q'}[-1]}{X_{P,Q}} \arrow[r] & \hcT{X_{P',Q'}}{X_{P,Q}[-1]} & {} 
		\end{tikzcd}
	\end{equation}
	We claim that $\hcT{g}{X_{P,Q}}$ is an isomorphism, so that we may conclude that $f$ factors through $g$. Well, by  \Cref{lem:w Hom P to X,lem:w Morphisms X to X}, the Hom-spaces $\hcT{P'[1] \oplus X_{R,Q'}}{X_{P,Q}}$ and $\hcT{P' \oplus X_{R,Q'}[-1]}{X_{P,Q}}$ vanish, and so we may conclude that $\hcT{g}{X_{P, Q}}$ is an isomorphism. 
	
	To prove the converse, suppose that $Q \leq Q' \leq R < P \leq P'$ does not hold. Observe that since $f$ is nonzero, it must be that $Q \leq R < P$ by \Cref{lem:w Hom P to X}. Moreover, if $\hcT{R}{X_{P', Q'}} = 0$, then $f$ cannot factor through $X_{P', Q'}$. Hence, it suffices to consider cases where $\hcT{R}{X_{P', Q'}} \neq 0$. By \Cref{lem:w Hom P to X}, this means that we may assume $Q' \leq R < P$. 
	
	Similarly, it suffices to consider cases where $\hcT{X_{P',Q'}}{X_{P,Q}} \neq 0$. By \Cref{lem:w Morphisms X to X}, we may assume that either (1) $Q \leq Q' < P \leq P'$ or (2) $Q'[1] < Q \leq P'[1] < P$. Combining (1) with the inequality $Q' \leq R < P$ implies that $Q \leq Q' \leq R < P \leq P'$, which is a contradiction since we assumed that this does not hold. Hence, (2) must hold.
	Combining (2) and the fact that $Q \leq R < P'$ implies $Q'[1] < Q \leq R < P' \leq P'[1] < P$ which may only occur if $w=1$, hence if $w>1$ then we are done.
	
	However, in this case and where $w=1$, \Cref{lem:w Hom P to X,lem:w Morphisms X to X} imply that the exact sequence \eqref{les:w X factors through X}, seen as a sequence of vector spaces, becomes
	\[ k \xrightarrow{\hcT{g}{X_{P,Q}}} k \xrightarrow{} k^2 \xrightarrow{} k. \]
	A dimensional argument shows that this is only possible if $\hcT{g}{X_{P,Q}}= 0$. Since this is true for any $g \colon R \to X_{P',Q'}$, given any pair of morphisms $g \in \hcT{R, X_{P', Q'}}$ and $h \in \hcT{X_{P', Q'}}{X_{P,Q}}$, then $h \cdot g = 0 \neq f$. Thus, $f$ does not factor through $X_{P',Q'}$, as required. 
\end{proof}

\begin{lemma}\label[lem]{lem:w X to P via X}
	Let $R,X_{P,Q}  \in \cT$ be indecomposable objects, such that there exists a non-zero morphism $f: X_{P,Q}  \rightarrow R$.
	Then $f$ factors through $X_{P',Q'}$ if and only if $Q'[1] \leq Q[1] < R \leq P'[1] \leq P[1]$.
\end{lemma}

\begin{proof}
	The proof is analogous to that of \Cref{lem:w P to X via X}, by applying the functor $\hcT{X_{P,Q} }{-}$ to the triangle
	\[
	X_{P',Q'} \rightarrow R \rightarrow Q'[1] \oplus X_{P',R} \rightarrow X_{P',Q'}[1].
	\]
\end{proof}

\begin{lemma}\label[lem]{lem:w X to X via P}
	Let $X_{P,Q} ,X_{P',Q'} \in \cT$ and $R \in \lang[1]{G}$ be indecomposable objects.
	Let $f: X_{P,Q}  \rightarrow X_{P',Q'}$ be a non-zero morphism.
	Then $f$ factors through $R$ if and only if $Q[1] < Q' \leq R \leq P[1] < P'$. 
\end{lemma}
\begin{proof}
	Suppose that $Q[1] < Q' \leq R \leq P[1] < P'$. Then, since $Q' \leq R < P'$, it follows from \Cref{lem:w Hom P to X} that there is a nonzero morphism $g \colon R \to X_{P', Q'}$. Therefore, applying $\hcT{X_{P, Q}}{-}$ to the triangle
	\[
	R \xrightarrow{g} X_{P',Q'} \rightarrow P'[1] \oplus X_{R, Q'} \rightarrow R[1],
	\]
	given by \Cref{lem:w P to X cone}, induces the long exact sequence 
	\[
	\dotsb \to \hcT{X_{P,Q} }{P' \oplus X_{R,Q'}[-1]} \rightarrow \hcT{X_{P,Q} }{R} \xrightarrow{\hcT{X_{P,Q}}{g}} \hcT{X_{P,Q}}{X_{P',Q'}} \rightarrow \hcT{X_{P,Q} }{P'[1] \oplus X_{R,Q'}} \to \dotsb.
	\]
	Therefore, it suffices to prove that $\hcT{X_{P,Q}}{g}$ is an isomorphism. Well, it follows from \Cref{lem:w Hom P to X} and \Cref{lem:w Morphisms X to X} that under the assumptions  $Q[1] < Q' \leq R \leq P[1] < P'$, then the Hom-spaces $\hcT{X_{P,Q} }{P' \oplus X_{R,Q'}[-1]}$ and $\hcT{X_{P,Q} }{P'[1] \oplus X_{R,Q'}}$ vanish, so that we may conclude that $\hcT{X_{P,Q}}{g}$ is an isomorphism.
	
	To prove that converse, suppose that $Q[1] < Q' \leq R \leq P[1] < P'$ does not hold.
	Observe that we may assume $\hcT{R}{X_{P',Q'}} \neq 0$, $\hcT{X_{P,Q}}{R} \neq 0$ and $\hcT{X_{P,Q}}{X_{P',Q'}} \neq 0$, since otherwise, the statement is clear. Thus, we may assume that $Q' \leq R < P'$, $\, Q[1] < R \leq P[1]$ and that either (1) $\, Q' \leq Q < P' \leq P$ or (2) $\, Q[1] < Q' \leq P[1] < P'$.
	Combining (2) with $Q' \leq R \leq P[1]$ implies that $Q[1] < Q' \leq R \leq P[1] < P'$, which is a contradiction since we assumed this does not hold. Hence, (1) must hold.
	Combining (1) and $Q[1] \leq R < P'$ implies that $Q' \leq Q \leq Q[1]  < R < P' \leq P$ holds, and again implies that $w=1$, else these relations cannot hold as $Q \not< Q[1]$.
	Similarly to the proof of \Cref{lem:w P to X via X}, this implies that there is an exact sequence of vector spaces
	\[k \xrightarrow{\hcT{X_{P,Q}}{g}} k \to k^2 \to k \]
	and so  $\hcT{X_{P,Q}}{g} = 0$. That is, the morphism $f$ cannot factor through $g$. 
\end{proof}

We may now complete our computations of the cones of morphisms between indecomposable objects.

\begin{lemma}\label[lem]{lem:w X to X cone}
	Let $X_{P,Q} ,X_{P',Q'} \in \cT$ be indecomposable objects and let $f: X_{P,Q}  \rightarrow X_{P',Q'}$ be a non-zero morphism.
	Then,
	\[
	\mathrm{cone} \, f \cong \begin{cases*}
		X_{P,P'}[1] \oplus X_{Q,Q'} & if $Q' \leq Q < P' \leq P$,\\
		X_{P',Q[1]} \oplus X_{P[1],Q'} & if $Q[1] < Q' \leq P[1] < P'$.
	\end{cases*}
	\]
	Equivalently, $\cone{f} \cong X_{P, P'}[1] \oplus X_{Q, Q'}$ if $f$ is a forwards morphism. Else, $f$ is a backwards morphism and $\cone{g} \cong X_{P', Q[1]} \oplus X_{P[1], Q'}$.
\end{lemma}
\begin{proof}
	Let $Q' \leq Q < P' \leq P$, then by \Cref{lem:w P to X via X}, the top left square of the diagram
	\[
	\begin{tikzcd}
		Q \arrow[d,"\sim" labl1 ] \arrow[r] & X_{P,Q}  \arrow[d,"f"] \arrow[r] & P[1] \arrow[d,"{u = (u_1,u_2)}"] \arrow[r] & Q[1] \arrow[d]\\
		Q \arrow[r] \arrow[d] & X_{P',Q'} \arrow[r, "h"] \arrow[d] & P'[1] \oplus X_{Q,Q'} \arrow[r] \arrow[d] & Q[1] \arrow[d] \\
		0 \arrow[d] \arrow[r] & \mathrm{cone} \, f \arrow[r,"\sim"] \arrow[d] & \mathrm{cone} \, f \arrow[d] \arrow[r] & 0 \arrow[d]\\
		Q[1] \arrow[r] & X_{P,Q} [1] \arrow[r] & P[2] \arrow[r] & Q[2].
	\end{tikzcd}
	\] 
	commutes. Hence, the full diagram with exact rows and columns exists and commutes by the octahedral axiom. By \Cref{lem:w Hom P to X}, $\hcT{P[1]}{X_{Q,Q'}} = 0$. Hence, if $u \neq 0$, then $\cone{f} \cong \cone{u_1} \oplus X_{Q, Q'} \cong X_{P,P'}[1] \oplus X_{Q, Q'}$ by the definition of $X_{P, P'}$.	
	To see that $u \neq 0$, observe that applying $\hcT{X_{P,Q}}{-}$ to the second row of the diagram yields the exact sequence
	\[\hcT{X_{P,Q}}{Q} \rightarrow \hcT{X_{P,Q}}{X_{P',Q'}} \xrightarrow{\hcT{X_{P,Q}}{h}} \hcT{X_{P,Q}}{P'[1] \oplus X_{Q,Q'}} \rightarrow \hcT{X_{P,Q}}{Q[1]} \]
	It therefore follows from \Cref{lem:w Hom X to P} that $\hcT{X_{P,Q}}{h}$ is an isomorphism. Hence, $0 \neq h \cdot f = u \cdot g$ and we may conclude that $u \neq 0$.
	
	Next, suppose that $Q[1] < Q' \leq P[1] < P'$, then by \Cref{lem:w X to X via P}, the top left square of the following diagram commutes.
	\[
	\begin{tikzcd}
		X_{P,Q}  \arrow[d,"\sim" labl1 ] \arrow[r] & Q'  \arrow[d] \arrow[r] & Q[1] \oplus X_{P[1],Q'} \arrow[d] \arrow[r] &  X_{P,Q} [1] \arrow[d]\\
		X_{P,Q}  \arrow[d] \arrow[r,"f"] & X_{P',Q'} \arrow[d] \arrow[r] & \mathrm{cone} \, f \arrow[d] \arrow[r] &  X_{P,Q} [1] \arrow[d]\\
		0 \arrow[d] \arrow[r] & P'[1] \arrow[d] \arrow[r] & P'[1]  \arrow[d, "{v = (v_1, v_2)}"] \arrow[r] &  0 \arrow[d]\\
		X_{P,Q} [1] \arrow[r] & Q'[1] \arrow[r] & Q[2] \oplus X_{P[1],Q'}[1] \arrow[r] &  X_{P,Q} [2].
	\end{tikzcd}
	\]
	Therefore, this diagram commutes and has exact columns and rows by the octahedral axiom. By assumption, $P[1] < P'$ and so \Cref{lem:w Hom P to X} implies that $\hcT{P'}{X_{P[1],Q'}} = 0$. Hence, $v_2 = 0$. It follows that $\cone{f} \cong \cone{v_1}[-1] \oplus X_{P[1],Q'}$. That is, $\cone{f} \cong X_{P', Q[1]} \oplus X_{P[1], Q'}$, as required.
	
	The final statement follows directly from \Cref{lem:w Morphisms X to X}. 
\end{proof}

\subsection{Factorisation Properties, Composition, and Triangles}

In this section, we consolidate the results established thus far into a handful of propositions that are easier to track. Our first result summarises the factorisation properties of morphisms between indecomposable objects.

\begin{proposition}\label[prop]{prop:w Factoring arcs}
	Let $Y,Z \in \cT$ be non-isomorphic indecomposable objects such that $\hcT{Y}{Z} \cong k$.
	Then, any nonzero morphism $f: Y \rightarrow Z$ factors through an indecomposable object $W$ if and only if there exist non-zero morphisms $g: Y \rightarrow W$ and $h: W \rightarrow Z$, such that
	\begin{enumerate}
		\item \label{for for} if $f$ is a forwards morphism, then $g$ and $h$ are both forwards morphisms,
		\item \label{back for} if $f$ is a backwards morphism, then one of $g$ or $h$ is a backwards morphism and the other is a forwards morphism.
	\end{enumerate}
\end{proposition}

\begin{proof}
	There are eight cases for us to consider, where each of $Y,Z,W$ are either of the form $R \in \lang[1]{G}$ or of the form $X_{P,Q}$. However, seven of these cases have been considered. 
	
	First, when $Y, Z, W \in \lang[1]{G}$, then the statement follows from \Cref{def:w linear gen} and \Cref{lem:w Forward and Back morphisms} \eqref{w forward from P}. Further, for $Y, W \in \lang[1]{G}$ and $Z = X_{P,Q}$, then the statement is a consequence of \Cref{lem:w Pa factors to X,lem:w Forward and Back morphisms} \eqref{w forward from P}. The third case, where $Y = X_{P,Q}$ and $Z, W \in \lang[1]{G}$, is implied by \Cref{lem:w X factors to Pa,lem:w Forward and Back morphisms} \eqref{w backward to P}. Next, when $Y \in \lang[1]{G}$, $Z=X_{P,Q}$ and $W=X_{P',Q'}$, then the statement is entailed by \Cref{lem:w Forward and Back morphisms,lem:w Morphisms X to X,lem:w P to X via X}. The fifth case, where $Y = X_{P,Q}$, $Z \in \lang[1]{G}$ and $W = X_{P',Q'}$, can be deduced by applying \Cref{lem:w Forward and Back morphisms,lem:w Morphisms X to X,lem:w X to P via X}. Additionally, the situation where $Y = X_{P, Q}$, $Z = X_{P', Q'}$ and $W \in \lang[1]{G}$ results from applying \Cref{lem:w Forward and Back morphisms,lem:w Morphisms X to X,lem:w X to X via P}. The seventh case, where $Y, Z \in \lang[1]{G}$ and $W \cong X_{P,Q}$ follows from \Cref{lem:w Forward and Back morphisms,lem:w P to Q via X}.
	
	Consider the final case, where $Y \cong X_{P,Q}$, $Z \cong X_{P',Q'}$ and $W \cong X_{R,S}$. If there does not exists a nonzero morphism $g \colon Y \to W$, then the statement is clear. Hence, we may assume that such a morphism exists. The functor $\hcT{-}{Z}$ induces the exact sequence
	\begin{align} \label{les:w X to X via X}
		\hcT{Y[1]}{Z} \rightarrow \hcT{\cone{g}}{Z} \rightarrow \hcT{W}{Z} \xrightarrow{\hcT{g}{Z}} \hcT{Y}{Z} \rightarrow \hcT{\cone{g}[-1]}{Z} \rightarrow \hcT{W[-1]}{Z} 
	\end{align}

	We begin by showing that if either \eqref{for for} or \eqref{back for} hold, then $f$ factors through $W$. Suppose first that \eqref{for for} holds. That is, $f \colon Y \to Z$ is a fowards morphism and there exist forwards morphisms $g \colon Y \to W$ and $h \colon W \to Z$. We will show that $f$ factors through $W$ by proving that $\hcT{g}{Z}$ is an isomorphism. Due to \Cref{lem:w Morphisms X to X}, it must be that $Q' \leq S \leq Q < P' \leq R \leq P$. Moreover, by \Cref{lem:w X to X cone}, $\cone{g} \cong X_{P, R} \oplus X_{Q, S}$. It thus follows from \Cref{lem:w Morphisms X to X} that $\hcT{\cone{g}}{Z} = 0$, and so \eqref{les:w X to X via X} implies that there is an exact sequence
	\[ 0 \rightarrow k \xrightarrow{\hcT{g}{Z}} k \]
	Hence, we may conclude that $\hcT{g}{Z}$ is an isomorphism, as required.
	
	Next, assume that $f \colon Y \to Z$ is a backwards morphism and that there exist a forwards morphism $g \colon Y \to W$ and a backwards morphism $h \colon W \to Z$. Again, we will show that $f$ factors through $W$ by showing that $\hcT{g}{Z}$ is an isomorphism. Our assumptions imply that $S \leq Q[1] < Q' \leq R[1] \leq P[1] < P'$ by \Cref{lem:w Morphisms X to X} and that $\cone{g} \cong X_{P,R} \oplus X_{Q,S}$ by \Cref{lem:w X to X cone}. As before, this implies that $\hcT{\cone{g}}{Z} = 0$ and we may deduce that $\hcT{g}{Z}$ is an isomorphism.
	
	Finally, assume that $f \colon Y \to Z$ is a backwards morphism and that there exist a backwards morphism $g \colon Y \to W$ and a forwards morphism $h \colon W \to Z$. Then, $Q[1] < Q'\leq S \leq P[1] < P' \leq R$ by \Cref{lem:w Morphisms X to X} and $\cone{g} \cong X_{R,Q} \oplus X_{P,S}$ by \Cref{lem:w X to X cone}. Therefore, $\hcT{\cone{g}}{Z} = 0$ and $\hcT{g}{Z}$ is an isomorphism. Consequently, if either \eqref{for for} or \eqref{back for} hold, then $f$ factors through $W$. 
	
	\medskip
	To prove the converse, suppose that $f \colon Y \to Z$ is such that $f = h \circ g$ where $g \colon Y \to W$ and $h \colon W \to Z$. Assume first that $f$ is a forwards morphism. We will show that \eqref{for for} holds by showing that all other cases lead to a contradiction. If $f$ is forwards and both $g$ and $h$ are backwards, then \Cref{lem:w Morphisms X to X} implies that $S < Q' \leq Q \leq Q[1] < S$, which is a contradiction. Whence, $g$ and $h$ cannot both be backwards. 
	
	On the other hand, if $f$ is forwards, $g$ is backwards and $h$ is forwards, then $Q' \leq Q[1] < S < P' \leq P[1] < R$ by \Cref{lem:w Morphisms X to X}. By the same lemma, \eqref{les:w X to X via X} induces the exact sequence
	\[ \label{les:w vector spaces X to X via X} 
	\hcT{Y[1]}{Z} \rightarrow k^2 \rightarrow k \xrightarrow{\hcT{g}{Z}} k. 
	\]
	Since $\hcT{Y[1]}{Z}$ is at most one dimensional, then $\hcT{g}{Z}$ cannot be an isomorphism, whence it is zero. We conclude that $f = h \circ g = 0$, which is a contradiction. 
	
	By the same argument, it cannot be that $f$ is a forwards morphsm, $g$ is a forwards morphism and $h$ is a backwards morphism. We thus conclude that if $f$ is a forwards morphism, then both $g$ and $h$ are forwards morphisms. That is, \eqref{for for} holds.
	
	Next, suppose that $f$ is a backwards morphism. We will show that  \eqref{back for} holds by showing that the two other possible cases lead to contradictions. If both $g$ and $h$ are backwards, then by \Cref{lem:w Morphisms X to X}, \eqref{les:w X to X via X} induces again the exact sequence \eqref{les:w vector spaces X to X via X}. Therefore, it must be that $\hcT{g}{Z} = 0$ and $f = h \circ g = 0$, which is a contradiction. Furthermore, if $g$ and $h$ are both forwards, then \Cref{lem:w Morphisms X to X} implies that $S \leq Q \leq Q[1] < Q' \leq S$, which is a contradiction. Consequently, if $f$ is a backwards morphism, then \eqref{back for} holds. \qedhere	
\end{proof}

With \Cref{prop:w Factoring arcs}, we are able to prove that for all indecomposable objects $X, Y \in \cT$, there is a choice of generators $\alpha^{X, Y} \in \hcT{X}{Y}$ which respects composition in a suitable sense. This is a subtle result, but it is the key technology we use to write down functors between triangulated categories with $w$-linear generators. It is also helpful for describing certain triangles in $\cT$.

\begin{proposition}\label[prop]{prop:w Generators of Hom}
	Let $\mathrm{ind} \, \cT$ denote the set of indecomposable objects in $\cT$. Then, there is a choice of morphisms $\{ \alpha^{Y, Z} \colon Y \to Z \}_{Y, Z \in \mathrm{ind} \, \cT}$ such that for any pair of indecomposable objects $Y,Z \in \cT$ 
	\[
	\alpha^{Y,Z} = \alpha^{W,Z} \alpha^{Y,W},
	\]
	whenever $\alpha^{Y, Z} \colon Y \rightarrow Z$ factors through an indecomposable object $W$.
\end{proposition}

\begin{proof}
	The proof of this result is quite long and subtle, so we postpone it to \Cref{sec: choice of gen}.
\end{proof}

Finally, we may completely determine the triangles containing at least two indecomposable terms up to isomorphism of triangles.

%% file: TheAlgebraLambda.tex
\section{Graded Endomorphism Algebras of $w$-Linear Generators}\label{Sec: Graded Endo}

\subsection{The Algebra \texorpdfstring{$\Lambda^G$}{LG}}\label{Subsec: LambdaG}

In this subsection, we are interested in computing the graded endomorphism algebra of a \textit{minimal} $w$-linear generator $G$. of $\cT$.
Recall from \cite{Generators} that a generator is minimal if there exists no proper direct summand $H$ of $G$, such that $H$ is a generator of $\cT$.

\begin{definition}
	Let $\cT$ be a triangulated category. Recall that the \textit{graded endomorphism ring} of an object $Y\in \cT$ is the graded ring $\Endo[\cT]{\ast}{Y}$ which, at degree $i$, is given by 
	\[
	\Endo[\cT]{i}{Y} = \Ext[\cT]{i}{Y}{Y} = \hcT{Y}{Y [i]}.
	\]
	For $f \in \Endo[\cT]{i}{Y}$ and $g \in \Endo[\cT]{j}{Y}$, multiplication is given by the rule
	\[
	g \cdot f := g [i] \circ f \in \Endo[\cT]{i+j}{M}.
	\]
\end{definition}

Notice that if $\cT$ is Krull-Schmidt and $Y = \bigoplus_{i=1}^m Y_i$, then $\Endo[\cT]{\ast}{Y}$ can be written as a graded generalised $m\times m$-matrix ring
\[
\Endo[\cT]{\ast}{Y} = \begin{pmatrix} \Endo[\cT]{\ast}{Y_1} & \Ext[\cT]{\ast}{Y_2}{Y_1} & \dotsb & \Ext[\cT]{\ast}{Y_m}{Y_1} \\ \Ext[\cT]{\ast}{Y_1}{Y_2} & \Endo[\cT]{\ast}{Y_2} & \dotsb & \Ext[\cT]{\ast}{Y_m}{Y_2} \\ \vdots & \vdots & \ddots & \vdots \\ \Ext[\cT]{\ast}{Y_1}{Y_m} & \Ext[\cT]{\ast}{Y_2}{Y_m} & \dotsb & \Endo[\cT]{\ast}{Y_m} \end{pmatrix}
\]
where $\Endo[\cT]{\ast}{Y_i}$ are seen as graded rings and $\Ext[\cT]{\ast}{Y_i}{Y_j}$ are seen as graded $\Endo[\cT]{\ast}{Y_j}$-$\Endo[\cT]{\ast}{Y_i}$-bimodules.

\begin{proposition}\label[prop]{Prop: Endo Alg}
	Let $\cT$ be a triangulated category satisfying \Cref{setup:w cluster cat} and admitting a minimal $w$-linear generator $G = \bigoplus_{i = 1}^m G_i$ with $m$ indecomposable summands. Then, $\Endo[\cT]{\ast}{G}$ is isomorphic to the upper triangular $m \times m$-matrix algebra $\chi^G$, with
	\[ \chi^G_{i, j} := \begin{cases} k[x] &  \text{if $i = j$ and $G_i < G_i[w]$} \\ 0 & \text{if $j < i$}  \\  k[x^\pm]  & \text{otherwise.} \end{cases} \]
	where $x$ is concentrated in degree $-w$. 
\end{proposition}

\begin{proof}
	Let $\{ \alpha^{Y, Z} \}_{Y, Z \in \mathrm{ind} \, \cT}$ be a choice of generators of Hom-spaces as in \Cref{prop:w Generators of Hom}. Moreover, by re-indexing if necessary, we assume that $G_j < G_i$ for all $j < i$. By the definition of a $w$-linear generator, these assumptions imply that $G_j[t] < G_i$ for all $t \in w\bZ$. To simplify notation, let $\Lambda := \Endo[\cT]{\ast}{G}$. Throughout the proof, let $k[x]$ and $k[x^{\pm}]$ be graded rings with grading defined by considering $x$ as a homogeneous element of degree $-w$.
	
	We begin by constructing graded isomorphisms $\Lambda_{i, j} \cong \chi^G_{i, j}$. Observe first that the definition of a $w$-linear generator implies that 
	\[ \Lambda_{i, j}^t = \Hom[\cT]{G_j}{G_i[t]} = k \cdot \alpha^{G_j, G_i[t]} \]
	for all $t \in w\bZ$ and $i,j \in \{1,\ldots,m\}$, where $\alpha^{G_j, G_i[t]}$ might possibly be zero. 

	If $j < i$, then $G_j < G_i[t]$ for all $t \in w\bZ$ and so $\alpha^{G_j, G_i[t]} = 0$ by the definition of a $w$-linear generator. Hence, $\Lambda_{i, j} = 0 = \chi^G_{i,j}$.
	
	Consider next the case $j = i$, $t \in w\bZ$, and $G_i < G_i[w]$. If $0<t$, then $G_i < G_i[t]$ so that $\Lambda_{i, i}^t = 0$. If, however, $t \leq 0$, then $\Lambda_{i, i}^t \neq 0$. In fact, For $t, s \leq 0$, observe that 
	\[ \alpha^{G_i, G_i[t]} \cdot \alpha^{G_i,G_i[s]} = \alpha^{G_i,G_i[t]}[s] \circ \alpha^{G_i, G_i[s]} = \alpha^{G_i, G_i[t+s]}, \]
	and we may construct a graded algebra isomorphism
	\[ \Lambda_{i,i} \to k[x] \colon \alpha^{G_i, G_i[t]} \mapsto  x^{\frac{t}{w}} \]
	for all $t \leq 0$.
	
	On the other hand, if $j = i$ and $G_i \cong G_i[w]$, then $\Lambda_{i, j}^t \neq 0$ for all $t \in w\bZ$ and $\Lambda_{i, j}^t = 0$ if $t \not\in \bZ$.
	So there is an algebra isomorphism 
	\[ \Lambda_{i,i} \to k[x^\pm] \colon \alpha^{G_i, G_i[t]} \mapsto x^{\frac{t}{w}} \]
	for all $t \in w\bZ$.
	
	Next, suppose $i < j$, so that $G_i < G_j$ and $G_i[t] < G_j$ for all $t \in w\bZ$. Let $s, \ell \in w\bZ$. If $s \leq 0$ or $G_i \cong G_i[w]$, and if $\ell \leq 0$ or $G_j \cong G_i[w]$, then
	\begin{align*}
		& \alpha^{G_i, G_i[s]} \cdot \alpha^{G_j, G_i[t]} = \alpha^{G_i[t], G_i[s+t]} \circ \alpha^{G_j, G_i[t]} = \alpha^{G_j,G_i[s+t]} \\
		& \alpha^{G_j, G_i[t]} \cdot  \alpha^{G_j, G_j[\ell]}  = \alpha^{G_j[\ell], G_i[\ell+t]} \circ \alpha^{G_j, G_j[\ell]} = \alpha^{G_j,G_i[\ell+t]}.
	\end{align*}
	Hence, under the identification $\Lambda_{i,i} \cong k[x]$ or $k[x^\pm]$ and $\Lambda_{j, j} \cong k[x]$ or $k[x^\pm]$, then there is a graded $\Lambda_{i,i}$-$\Lambda_{j,j}$-bimodule isomorphism 
	\[ \Lambda_{i, j} \to k[x^\pm] \colon \alpha^{G_j, G_i[t]} \mapsto x^{\frac{t}{w}}
	\]
	for all $t \in \bZ$.
	
	Finally, we let $\phi \colon \Lambda \to \chi^G$ be the graded ring morphism specified by 
	\[ [\phi([\alpha^{G_i, G_j[t]}]_{i, j})]_{i, j} := \begin{cases} x^t & \text{if $i =j$, $G_i < G_i[w]$ and $t \leq 0$}, \\
		0 & \text{if $i =j$, $G_i < G_i[w]$ and $0 < t$}, \\ 
		0 & \text{if $j < i$}, \\
		x^t & \text{otherwise}.
	\end{cases}  
	\]
	It is straightforward to check that this is indeed a graded ring isomorphism. 
\end{proof}

Let $G \in \cT$ be a $w$-linear generator of a triangulated category satisfying \Cref{setup:w cluster cat}.
Then we define $\Lambda^G$ to be the differential graded algebra with underlying graded algebra $\chi^G$ and with a trivial differential. Our goal for the remainder of this section is to show that $\Lambda^G$ is a $w$-linear generator of the triangulated category $\perf(\Lambda^G)$, so that we may construct an equivalence $\cT \xrightarrow{\sim} \perf(\Lambda^G)$. We first recall a well-known lemma that will be useful to us.

\begin{lemma}[{See e.g.\ \cite[Prop. 2.3]{KrauseKS} or \cite[A.1 (iv)]{Jor}}] \label[lem]{Lemma: projectivisation}
     Let $\cA$ be an additive idempotent complete category and consider objects $X, Y, Z \in \cA$. If $Y \in \add_\cA X$, then functor $\cA(X,-)$ induces an isomorphism
    \[ \cA({Y},{Z}) \to \Hom[\End{\cA}{X}]{\cA({X},{Y})}{\cA({X},{Z}) }. \]
    which is natural in $Y$ and $Z$. \label{yoneda}
\end{lemma}

\begin{proposition} \label[prop]{Prop: Ext is equivalence on lang}
Let $G \in \cT$ be a minimal $w$-linear generator of a triangulated category satisfying \Cref{setup:w cluster cat}. Then, the functor
\[ \Ext[\cT]{\ast}{G}{-} \colon \lang[1]{G} \to \lang[1]{\Lambda^G} \]
is an equivalence. 
\end{proposition}

\begin{proof}
	Let $G = \bigoplus_{i=1}^m G_i$. We begin by showing that the functor $\Ext[\cT]{\ast}{G}{-}$ is fully faithful. This amounts to showing that the assignment
	\[ \hcT{G_i}{G_j[t]} \to \Hom[\Lambda^G]{\Ext[\cT]{\ast}{G}{G_i}}{\Ext[\cT]{\ast}{G}{G_j[t]}} \] 
	sending $f \colon G_i \to G_j[t]$ to $f \circ (-) $ is bijective for $t \in q\bZ$. This is essentially a consequence \Cref{Lemma: projectivisation}, as we will explain below. 

	Observe first that morphism $f \in \Hom[\Lambda^G]{\Ext[\cT]{\ast}{G}{G_i}}{\Ext[\cT]{\ast}{G}{G_j[t]}}$ is equivalent to the data 
	\[ (f^s \colon \hcT{G}{G_i[s]} \to \hcT{G}{G_j[s+t]})_{s \in \bZ}, \]
	where for all $\alpha \in \hcT{G}{G_i[s]}$ and $\lambda \in (\Lambda^G)^{\ell}$, $f^s(\alpha) \cdot \lambda = f^{s+\ell}(\alpha \cdot \lambda)$ if and only if $s,l \in w\bZ$, else $f^s = 0$.

	We claim that $f$ is uniquely determined by $f^0$. To see this, note that, under the graded algebra isomorphism of \Cref{Prop: Endo Alg}, then $\alpha \in \hcT{G}{G_i[s]}$ can be written as $\alpha = (a_\ell x^s)_{\ell=1}^m$ with $a_\ell \in k$ for all $1 \leq \ell \leq m$. Let $I_{-s} \in \Lambda^G$ denote the $m \times m$ diagonal matrix with entries $x^{-s}$ along the diagonal. Then, 
	\[ f^s(\alpha) \cdot I_{-s} = f^0( (\alpha_\ell x^s)_{\ell=1}^m I_{-s}  ) = f^0( (\alpha_\ell)_{\ell=1}^m ), \]
	and so
	\[ f^s(\alpha) =  f^0( (\alpha_\ell)_{\ell=1}^m ) I_s. \]
	Hence, the nonzero morphisms $f^s$ are all uniquely determined by $f^0$. That is, $f$ is determined by its component $f^0$.

	Furthermore, morphisms $f^0$ are precisely the morphisms in $\Hom[(\Lambda^G)^0]{\hcT{G}{G_i}}{\hcT{G}{G_j[t]}}$. By definition, $(\Lambda^G)^0 = \hcT{G}{G}$ and $G_i \in \add_\cT G$, so that by \Cref{Lemma: projectivisation}, there are isomorphisms
	\[ \hcT{G_i}{G_j[t]} \xrightarrow{\sim} \Hom[(\Lambda^G)^0]{\hcT{G}{G_i}}{\hcT{G}{G_j[t]}} \]
	sending $h \colon G_i \to G_j[t]$ to $h \circ (-)$.

	Since morphisms  $f \in \Hom[\Lambda^G]{\Ext[\cT]{\ast}{G}{G_i}}{\Ext[\cT]{\ast}{G}{G_j[t]}}$ are determined by the component $f^0$, this is enough to conclude that the assignment
	\[ \hcT{G_i}{G_j[t]} \to \Hom[\Lambda^G]{\Ext[\cT]{\ast}{G}{G_i}}{\Ext[\cT]{\ast}{G}{G_j[t]}} \]
	sending $f \colon G_i \to G_j[t]$ to $f \circ (-) $ is bijective.

	To conclude that $\Ext[\cT]{\ast}{G}{-} \colon \lang[1]{G} \to \lang[1]{\Lambda^G}$ is an equivalence, it remains to show that it is essentially surjective. This is essentially the same argument as in the proof of \cite[2.3]{KrauseKS}. An object $Y \in \lang[1]{\Lambda^G}$ is a summand of $(\Lambda^G)^s[r]$ for some integers $s, r$. Hence, there is an idempotent morphism $\bar{e} \colon (\Lambda^G)^s \to (\Lambda^G)^s$and morphisms $\bar{p} \colon G^s \to Y[-r]$, $\bar{i} \colon Y[-r] \to G^s$ such that $\bar{i} \circ \bar{p} = \bar{e}$ and $\bar{p} \circ \bar{i} = \id_{Y[-r]}$. 

	Since the functor $\Ext[\cT]{\ast}{G}{-}$ is fully faithful and additive, there is an idempotent $e \colon G^s \to G^s \in \cT$ such that $\bar{e} = e \circ (-)$. Because $\cT$ is idempotent complete (since it is Krull-Schmidt), $e$ splits, and so there are morphisms $p \colon G^s \to X$ and $i \colon X \to G^s$ with $i \circ p = e$ and $p \circ i = \id_X$.   
	
	It is straightforward to check that the morphisms
	\begin{align*}
		& (p \circ (-)) \circ \bar{i} \colon Y[-r] \to \Ext[\cT]{\ast}{G}{X}, \\
		& (i \circ (-)) \circ \bar{p} \colon \Ext[\cT]{\ast}{G}{X} \to Y[-r]
	\end{align*}
	are inverses. Hence, since $\Ext[\cT]{\ast}{G}{-}$ commutes with suspension, $ \Ext[\cT]{\ast}{G}{X[r]} \cong Y$, as required.
\end{proof}

\subsection{The Algebra \texorpdfstring{$\Lambda^w_n$}{Ln}}\label{Sec: Lambda}

With applications to Paquette--Y\i ld\i r\i m completions of discrete cluster categories of type $A_\infty$ \cite{Paquette2020} in mind, we consider a triangulated category $\cT$ satisfying \Cref{setup:w cluster cat} and equipped with a minimal $w$-linear generator $G \cong \bigoplus_{i=1}^{2n-1} P_i$ such that $P_i \cong P_i[w]$ if and only if $i \equiv 0 \, \mathrm{mod}\, 2$. For such a generator $G$, we know by \Cref{Prop: Endo Alg}, that $\Lambda^G$ is the upper triangular $(2n-1) \times (2n-1)$-matrix algebra
\[
\Ext[\cT]{\ast}{P_i}{P_j} = \begin{cases} k[x] &  \text{if $i = j$ and $i$ odd,} \\ 0 & \text{if $j < i$,}  \\  k[x^\pm]  & \text{otherwise.} \end{cases} 
\]
From now on, we shall denote the graded endomorphism $k$-algebra of the $w$-linear generator $G \in \ocC$ by $\Lambda^w_n$. Let $e_i \in \Lambda^w_n$ correspond to the idempotent matrix whose entries are zero everywhere except at index $(i, i)$, in which case the entry is $1$.

The algebra $\Lambda^w_n$ admits a particularly nice presentation as the
path algebra of a graded quiver. More precisely, let $Q$ be the quiver
\[
\begin{tikzcd}
	Q : & 1 \arrow[r,"\delta_1"] \arrow[loop right, in=290,out=250,looseness=15,"\alpha_1",swap] & 2 \arrow[r,"\delta_2"] \arrow[loop right, in=290,out=250,looseness=15,"\alpha_2",swap] \arrow[loop,in=70,out=110,looseness=12,"\beta_2"] & 3 \arrow[r,"\delta_3"] \arrow[loop right, in=290,out=250,looseness=15,"\alpha_3",swap] & 4 \arrow[loop right, in=290,out=250,looseness=15,"\alpha_4",swap] \arrow[loop,in=70,out=110,looseness=12,"\beta_4"] \arrow[r] & \cdots \arrow[r] & 2n-2 \arrow[loop right, in=290,out=250,looseness=15,"\alpha_{2n-2}",swap] \arrow[loop,in=70,out=110,looseness=12,"\beta_{2n-2}"] \arrow[rr,"\delta_{2n-2}"] && 2n-1 \arrow[loop right, in=290,out=250,looseness=15,"\alpha_{2n-1}",swap].
\end{tikzcd}
\]
where we place $\delta_i$ in degree zero, the loops $\alpha_i$ are placed in degree $-w$, and $\beta_i$ is in degree $w$.

Let $Q_0$ denote the set of vertices of $Q$ and for each $a \in Q_0$ write $\iota_a$ for the trivial idempotent path in $Q$ at vertex $a$. Moreover, let $B \subset Q_0$ be the set of vertices such that there is no loop $\beta_b$ for $b \in B$. 
Consider the ideal $I$ is given by 
\[
I := \langle \alpha_a \beta_a  - \iota_a,\, \beta_a \alpha_a  - \iota_a,\, \alpha_{b} \delta_b - \delta_b \alpha_{b+1},\, \beta_a  \delta_a \delta_{a+1} - \delta_{a} \delta_{a+1} \beta_{a+2}\rangle,
\]
for $1 \leq a,b \leq 2n-1$, and $a \notin B$.

\begin{remark}
	The differential graded algebra $\Lambda_n$ is in fact the canonical example of a \textit{piano algebra}, a class of dg algebras introduced and explored in further detail by the authors in \cite{PianoAlg}.
\end{remark}

The goal of this section is to prove the following result.

\begin{theorem}\label{Thm:path algebras}
	Let $G$ be the minimal $w$-linear generator of $\cT$ with graded endomorphism algebra $\Lambda^w_n$.
	Then there is an isomorphism of graded $k$-algebras,
	\[
	\Lambda^w_n \xrightarrow{\sim} kQ/I.
	\]
\end{theorem}

To prove this theorem, we begin with some lemmas concerning submodules of $kQ/I$.

\begin{lemma}\label[lem]{Lem: equiv classes of paths}
	Let $a,b$ be different vertices in $Q_0$, and let $m \in \mathbb{Z}$.
	Then $(\iota_a(kQ/I)\iota_b)^m \cong (\iota_a(kQ/I)\iota_b)^0$ as $k$-vector spaces if $m \in w \bZ$, and trivial otherwise.
	Moreover, the dimension of $(\iota_a(kQ/I)\iota_b)^m$ as a $k$-vector space is at most 1. In fact, if there is a nonzero path $\sigma \in (\iota_a(kQ/I)\iota_b)^m$, then it is equivalent to  
	\[
	\sigma \equiv \begin{cases*}
		\delta \cdots \delta \delta \alpha^{-q}  & if $q \leq 0$,\\
		\delta \cdots \delta \delta \beta^q  & if $q \geq 0$ and $b \not\in B$,\\
		\delta \cdots \delta \beta^q \delta & if $q \geq 0$ and $b \in B$.
	\end{cases*}
	\]
\end{lemma}

\begin{proof}
	In order to avoid cluttered notation, given arrows $\delta_a$, $\alpha_a$ and $\beta_a$, we may suppress the indices $a$ and $b$ when there is no ambiguity. 
	
	Suppose $\sigma = \rho \delta \rho \delta \cdots \rho \delta \rho$ is a path of degree $q \in w\bZ$, starting at $a$ and ending at $b$ where each $\rho$ is some path consisting combinations of cycles $\alpha_c$, $\beta_c$, and $\iota_c$ for a $c \in Q_0$.

	By the relations in $I$, we know that $\alpha_a \delta_a = \delta_a \alpha_b$ and $\alpha_a \beta_a = \beta_a \alpha_a$, and so we may "commute" $\alpha$ with $\beta$ and $\delta$. Hence
	\[
	\sigma \equiv \sigma' = \beta^{i_1} \delta \beta^{i_2} \delta \cdots \beta^{i_{p-1}} \delta \beta^{i_p} \alpha^i ,
	\]
	where $q = \sum_{l=1}^p (i_l) - i$.
	
	Similarly, according to the relations in $I$, for a vertex $c \notin B$, $\beta_c \omega =  \omega  \beta_d$, and so we may "commute" $\beta$ with paths starting and ending at vertices not in $B$. Hence, we have either
	\[
	\sigma \equiv \sigma_1'' = \delta \cdots \delta\delta \alpha^i \beta^j ,
	\]
	or
	\[
	\sigma \equiv \sigma_2'' = \delta \cdots \delta \beta^j \delta \alpha^i,
	\]
	where, in both $\sigma_1''$ and  $\sigma_2''$, we must have $q=j-i$, as $\sigma$ is in degree $q$ and all paths $\delta$ are in degree $0$.

	Thus, one of the following is true;
	\[
	\sigma \equiv \begin{cases*}
		\delta \cdots \delta \delta \alpha^{-m}  & if $q \leq 0$,\\
		\delta \cdots \delta \delta \beta^m  & if $q \geq 0$ and $b \not\in B$,\\
		\delta \cdots \delta \beta^m \delta & if $q \geq 0$ and $b \in B$.
	\end{cases*}
	\]
	That is, any non-zero path from $a$ to $b$ of degree $q$ is equivalent to one of the above paths in $kQ/I$, determined by $q$, $a$ and $b$.
	Indeed, such a path exists if and only if a path of degree $0$ between $a$ and $b$ exists.
	Hence there is a $k$-vector space isomorphism $(\iota_a(kQ/I)\iota_b)^q \cong (\iota_a(kQ/I)\iota_b)^0$ for $q \in w\bZ$, and it follows that the dimension as a $k$-vector space is at most $1$ as $Q$ is a tree.
\end{proof}

\begin{lemma}\label[lem]{Lem: loop algebras}
	There is an isomorphism of graded $k$-algebras,
	\[
	\iota_a(kQ/I)\iota_a \cong e_a \Lambda^w_n e_a \cong \Endo[\cT]{\ast}{G_a},
	\]
	for any $a \in Q_0$.
\end{lemma}

\begin{proof}
	By \Cref{Lem: equiv classes of paths}, there exists at most one non-zero equivalence class of paths of degree $q \in w\bZ$ starting and ending at vertex $a$: If $q < 0$, then all nonzero paths are equivalent to $\alpha^q$, and so $(\iota_a (kQ/I) \iota_a)^q = k \cdot \alpha^{\frac{-q}{w}}$. On the other hand, if $q = 0$, then all nonzero paths can be represented by $\iota_a$ and so $(\iota_a (kQ/I) \iota_a)^0 = k \cdot \iota_a$. If $q > 0$ and $a \in B$, then there is no loop $\beta_a$ at $a$ and so there are no nonzero paths of degree $q$. However, if $q > 0$ and $a \notin B$, then all nonzero paths are equivalent to $\beta^q$, so that $(\iota_a (kQ/I) \iota_a)^q = k \cdot \beta^{\frac{q}{w}}$. 
	
	Suppose first that $a \in B$, then $a$ is odd and $G_a < G_a[w]$ by definition of the $w$-linear generator $G$. Hence, there is an isomorphism $\Endo[\cT]{\ast}{G_a} \cong k[x]$, with $\lvert x \rvert =-w$ by \Cref{Prop: Endo Alg}. Moreover, since $a \in B$, then, by the above, there are no nonzero paths starting and ending at $a$ of degree $q > 0$. We may thus construct a graded $k$-linear map $\varphi_{a,a} : \iota_a(kQ/I)\iota_a \rightarrow k[x]$ sending $\alpha^{\frac{-q}{w}} \mapsto x^{\frac{-q}{w}}$ for $q < 0$ and sending $\iota \mapsto 1$. This is clearly a graded algebra isomorphism. 
	
	Next, consider the case $a \not\in B$. Then, $a$ is even, $G_a \cong G_a[w]$ and $\Endo[\cT]{\ast}{G_a} \cong k[x^{\pm}]$, with $\lvert x \rvert =-w$ by \Cref{Prop: Endo Alg}. Thus, similarly to before, we may construct the $k$-linear map $\varphi_{a,a} : \iota_a(kQ/I)\iota_a \rightarrow k[x^{\pm 1}]$, taking $\alpha^{\frac{-q}{w}} \mapsto x^{\frac{-q}{w}}$ for $q<0$, $\iota \mapsto 1$, and $\beta^{\frac{q}{w}} \mapsto x^{\frac{q}{w}}$ for $q > 0$. Again, it is clear that this is an isomorphism of graded $k$-algebras given the relations $\alpha\beta -\iota$ and $\beta \alpha - \iota$.
\end{proof}

\begin{lemma}\label[lem]{Lem: path modules}
	Let $a,b$ be distinct vertices in $Q_0$.
	Then there is an isomorphism of graded $\Endo{\ast}{G_a}$-$\Endo{\ast}{G_b}$-bimodules
	\[
	\iota_a(kQ/I)\iota_b \cong e_a \Lambda^w_n e_b \cong \Ext[\cT]{\ast}{G_a}{G_b}.
	\]
\end{lemma}

\begin{proof}
	If $b < a$, then there are no paths from $a$ to $b$, and so $\iota_a (kQ/I) \iota_b = 0$. Similarly, by definition of the $w$-linear generator $G$, $b < a$ implies that $\Ext[\cT]{\ast}{G_a}{G_b} = 0$. Whence, we may assume $a < b$.

	There are four cases to consider: depending on whether $a \in B$ and $b \in B$ or not. Let us first assume that $a \in B$ and $b \notin B$. By \Cref{Prop: Endo Alg} and the of $G$, we know that there are algebra isomorphisms $\Endo{\ast}{G_a} \cong k[x]$, $\Endo{\ast}{G_b} \cong k[x^{\pm}]$. Moreover, $\Ext[\cT]{\ast}{G_a}{G_b} \cong k[ x^{\pm} ]$ as a $k[x]$-$k[x^{\pm}]$-bimodule. Hence, it suffices to show that $\iota_a(kQ/I)\iota_b \cong k[x^{\pm}]$ as a $k[x]$-$k[x^{\pm}]$-bimodule. 

	There is only one degree zero path from $a$ to $b$ given by $\sigma = \delta_{a, a+1} \dotsb \delta_{b-1, b}$, so that $(\iota_a (kQ/I) \iota_b)^0 = k \cdot \sigma$. Moreover, by \Cref{Lem: equiv classes of paths}, any path $a \to b$ of degree $q < 0$ is equivalent to $\sigma \alpha^{\frac{-q}{w}}$, and any path $a \to b$ of degree $q' > 0$ is equivalent to $\sigma \beta^{\frac{q'}{w}}$. Hence, $(\iota_a (kQ/I) \iota_b)^q = k \cdot \sigma \alpha^{\frac{-q}{w}}$ and  $(\iota_a (kQ/I) \iota_b)^{q'} = k \cdot \sigma \beta^{\frac{q'}{w}}$. We may thus construct the graded bimodule isomorphism $\varphi_{a, b} \colon \iota_a (kQ/I) \iota_b \to k[x^{\pm}]$ via the assignment $\sigma \alpha^{\frac{-q}{w}} \mapsto x^{\frac{-q}{w}}$, $\sigma \mapsto 1$, and $\sigma \beta^{\frac{q'}{w}} \mapsto x^{\frac{-q'}{w}}$ for $0 > q \in w\bZ$ and $0 < q' \in w\bZ$. 

	The case where $a \notin B$ and $b \in B$ is similar, and and so we omit it.
\end{proof}

We are now in a position to prove Theorem \ref{Thm:path algebras}.

\begin{proof}[Proof of \Cref{Thm:path algebras}]
	We combine \Cref{Lem: loop algebras,Lem: path modules}, and \Cref{Prop: Endo Alg} to construct a $k$-linear isomorphism $h \colon kQ/I \rightarrow \Lambda^w_n$ sending a path $\sigma$ between vertices $a \to b$ to $\varphi_{a, b}(\sigma)$. Here, we mildly abuse notation and denote the composition of $\varphi_{a, b} \colon \iota_a kQ/I \iota_b \to e_a \Lambda^w_n e_b$ with the natural inclusion $i_{a, b} \colon e_a \Lambda^w_n e_b  \to \Lambda^w_n$ as $\varphi_{a, b} \colon \iota_a kQ/I \iota_b \to \Lambda^w_n$. 

	All that is left to show is that $h$ is an isomorphism of graded $k$-algebras. Let $\sigma \in kQ/I$ be a path $a \to b$ of degree $q$ and let $\tau \in kQ/I$ be a path $b \to c$ of degree $q'$. We would like to show that $h(\sigma \tau) = h(\sigma) h(\tau)$. 

	A case by case analysis of \Cref{Lem: equiv classes of paths} together with the definition of $\varphi_{a,b}$ and $\phi_{a,c}$ implies that \begin{align*}
	& h(\sigma) = \varphi_{a, b}(\sigma) = i_{a, b}(x^{\frac{-q}{w}}) \\ 
	& h(\tau) = \varphi_{b, c}(\tau) = i_{b, c}(x^{\frac{-q'}{w}}). 
 \end{align*}
Thus, $h(\sigma) h (\tau) = i_{a, c}(x^{\frac{-q'-q}{w}})$.

	Moreover, since paths $\alpha$ "commute" with paths $\delta$ due to the relations $I$, then, by \Cref{Lem: equiv classes of paths}, there is a degree zero path $\gamma_0$ such that
\[ \sigma \tau = \begin{cases} 
			\gamma_0 \alpha^{\frac{-q'-q}{w}}  & \text{if $i+j < 0$,} \\
			\gamma_0 \beta^{\frac{q'+q}{w}}  & \text{if $i+j > 0$ and $c \notin B$,} \\
			\gamma_0 \beta^{\frac{q'+q}{w}} \delta \cdots \delta & \text{if $i+j > 0$ and $c \in B$,} \\
			\gamma_0 & \text{if ${\frac{q'+q}{w}}= 0$.}
		\end{cases}
\]
In any of the cases above, 
\[ h(\sigma \tau) = \phi_{a, c}(\sigma \tau) = \iota_{a, c}(x^{\frac{-q'-q}{w}}) =h(\sigma) h(\tau), \]
as required. 
\end{proof}

\begin{proposition}\label[prop]{Prop: morph of lambda}
	Let $w,w$ be positive integers.
	There exists an injective homomorphism of dg algebras
	\[
	\nu \colon \Lambda^w_n \longrightarrow \Lambda^{w'}_n
	\]
	if and only if either $w \in w'\bZ$.
\end{proposition}

\begin{proof}
	Suppose that $w \not\in w'\bZ$, and suppose that $\nu$ exists.
	This immediately presents a contradiction as $(\Lambda^{w'}_n)_w = 0$ and $(\Lambda^{w}_n)_w \neq 0$, hence $\nu_w \colon (\Lambda^{w}_n)_w \rightarrow (\Lambda^{w'}_n)_w$ cannot be an injective morphism.
	
	Now suppose that $w \in w'\bZ$.
	Recall that $(\Lambda^w_n)_0 \cong (\Lambda^{w'}_n)_0 \cong kA_{2n-1}$,
	Then, by \Cref{Lem: equiv classes of paths}, there exists an isomorphism of $kA_{2n-1}$-modules $f_p \colon (\Lambda^{w}_n)_p \rightarrow (\Lambda^{w'}_n)_p$ for all $p \in w\bZ$.
	Then we define $\nu$ to be the dg homomorphism where $\nu_p = f_p$ for $p \in w\bZ$, and $\nu_p = 0$ otherwise.
	
	It is straightforward to verify that this is indeed a homomorphism of dg algebras, and that it is injective as there exists an isomorphism $\nu_p$ for each nonzero degree of $\Lambda^w_n$.
\end{proof}

%% file: Bardzell.tex
\section{The Hochschild Cohomology of \texorpdfstring{$\Lambda^w_n$}{LG}}\label{Sec: HH}

Here we compute the Hochschild cohomology of the graded endomorphism algebra of a linear generator, and show that such an algebra is intrinsically formal.

\subsection{The Bardzell Resolution}

The Bardzell resolution has been used to compute the entire Tamarkin--Tsygen calculus for a gentle algebra, and the Hochschild cohomology for graded (skew-)gentle algebras.
By using the consideration that for a linear generator $G \in \mathcal{T}$ the endomorphism algebra is isomorphic to a linear quiver of Dynkin type $A$, we use an adaptation of the Bardzell resolution to compute the Hochschild cohomology of $\Lambda^w_n$.

Recall that, for a quiver with relations $(Q,I)$, an \textit{anti-path} is a path $a_na_{n-1} \cdots a_2a_1$ in $Q$ such that $a_{i+1}a_i \in I$ for all $i=1,\ldots, n-1$.
Let $P_n$ be the set of anti-paths of length $n$, where $P_0 = Q_0$ and $P_1 = Q_1$.

\begin{definition}
	Let $A \cong k Q/I$ be a gentle algebra, and let $E = k Q_0$.
	Then the \textit{Bardzell complex} $\bbar$ is defined as the complex
	\[
	\cdots \xrightarrow{d_{n+1}} A \otimes_{E} k P_n \otimes_{E} A \xrightarrow{d_n} \cdots \xrightarrow{d_2} A \otimes_{E} k P_1 \otimes_{E} A \xrightarrow{d_1} A \otimes_{E} k P_0 \otimes_{E} A \xrightarrow{\pi} A \rightarrow 0,
	\] 
	where $\pi(a \otimes e_i \otimes b) = ae_ib$, and
	\begin{align*}
		d_n(c \otimes a_n \cdots a_2 a_1 \otimes b) &= ca_n \otimes a_{n-1}\cdots a_2a_1 \otimes b + (-1)^n c \otimes a_n \cdots a_2 \otimes a_1b,\\
		d_1(c \otimes a_1 \otimes b) &= ca_1 \otimes e_{t(a_1)} \otimes b - c \otimes e_{s(a_1)} \otimes a_1b.
	\end{align*}

\end{definition}

The Bardzell resolution $\bbar$ is an $A$-$A$ bimodule projective resolution.

We must consider an adaptation of the Bardzell complex that accounts for the loops in the quiver of $\Lambda^w_n$.
Let $\mathcal{L}_n$ be the set of paths of alternating loops at each vertex in $S$ for $n \geq 2$, that is, when $n$ is even,
\[
\mathcal{L}_{2i} := \{ (\alpha_j \beta_j)^i, (\beta_j\alpha_j)^i \mid j \in S \cap T\},
\]
and when $n$ is odd
\[
\mathcal{L}_{2i+1} := \{ (\alpha_j \beta_j)^i\alpha , (\beta_j\alpha_j)^i\beta  \mid j \in S \cap T\}.
\]
We then let $\mathcal{P}_n := P_n \sqcup \mathcal{L}_n$ for $n \geq 2$, where we keep $\mathcal{P}_1 = Q_1$ and $\mathcal{P}_0 = Q_0$.
Then the \textit{looped Bardzell complex of} $\Lambda$ is the complex $\lbar := (\Lambda \otimes_{E} k \mathcal{P}_{\bullet} \otimes_{E} \Lambda, d_n)$, where the action of $d_n$ is extended to $\mathcal{L}_n$, that is,
\[
d_n(c \otimes (\alpha \beta)^i \otimes b) = c \alpha \otimes (\beta \alpha)^{i-1}\beta \otimes b + (-1)^n c \otimes (\alpha \beta)^{i-1} \alpha \otimes \beta c.
\]
It is clear then that, for $n \geq 2$, the differential $d_n$ takes elements in $\Lambda \otimes_{E} k P_n \otimes_{E} \Lambda$ to elements in $\Lambda \otimes_{E} k P_{n-1} \otimes_{E} \Lambda$, and elements in $\Lambda \otimes_{E} k \mathcal{L}_n \otimes_{E} \Lambda$ to elements in $\Lambda \otimes_{E} k \mathcal{L}_{n-1} \otimes_{E} \Lambda$.

\begin{lemma}\label[lem]{Lem: Bardzell complex}
	The looped Bardzell complex $\lbar (\Lambda)$ is a complex of $k$-vector spaces.
\end{lemma}

\begin{proof}
	To check that $d_{n-1}d_n = 0$ for all $n \geq 2$, we only need to check this is true for paths in $\mathcal{L}_n$, as it is known that the traditional Bardzell resolution $\bbar (\Lambda)$ is an exact complex.
	Consider,
	\begin{align*}
		d_{n-1}d_n (c \otimes (\alpha \beta)^i \otimes b) &= d_{n-1}(c \alpha \otimes (\beta \alpha)^{i-1}\beta \otimes b + (-1)^n c \otimes (\alpha \beta)^{i-1} \alpha \otimes \beta c)\\
		&= c \alpha \beta \otimes (\alpha \beta)^{i-1} \otimes b + (-1)^{n-1} c \alpha \otimes (\beta \alpha)^{i-1} \otimes \beta b +\\
		& \qquad \qquad (-1)^n c \alpha \otimes (\beta \alpha)^{i-1} \otimes \beta b + (-1)^{2n-1}c \otimes (\alpha \beta)^{i-1} \otimes \alpha \beta b\\
		&= c \otimes (\alpha \beta)^{i-1} \otimes b + (-1)^{n-1} c \alpha \otimes (\beta \alpha)^{i-1} \otimes \beta b +\\
		& \qquad \qquad (-1)^n c \alpha \otimes (\beta \alpha)^{i-1} \otimes \beta b - c \otimes (\alpha \beta)^{i-1} \otimes b\\
		&=0,
	\end{align*}
	where the third equality comes by considering $\alpha \beta \in I$.
	Analogous arguments show that $d_{n-1}d_n = 0$ for all $n \geq 2$.
	
	Finally, we must check that $\pi d_1 = 0$, and so now we consider
	\begin{align*}
		\pi d_1 (c \otimes \alpha \otimes b) &= \pi (c \alpha \otimes e \otimes b - c \otimes e \otimes \alpha b)\\
		& = c \alpha b - c \alpha b\\
		& =0.
	\end{align*}
	Hence we are done.
\end{proof}

We impose a grading on $\Lambda \otimes_E k \mathcal{P}_n \otimes_E \Lambda$ by placing $b \otimes a \otimes c$ in degree $m = \lvert b \rvert + \lvert a \rvert + \lvert c \rvert$.
We note that $d_n$ and $\pi$ are all degree preserving morphisms.

\begin{lemma}
	Let $\Lambda$ be the graded endomorphism ring for some $w$-linear generator.
	Then $\lbar$ is a projective $\Lambda^w_n$-$\Lambda^w_n$ bimodule resolution of $\Lambda$.
\end{lemma}

\begin{proof}
	It is clear that $\Lambda^w_n \otimes_{E} k \mathcal{P}_n \otimes_{E} \Lambda^w_n$ is a projective $\Lambda^w_n$-$\Lambda^w_n$ bimodule for all $n$, as it is isomorphic to a finite direct sum of direct summands of $\Lambda^w_n \otimes_{k} \Lambda^w_n$.
	Therefore, it is left to show that $\lbar$ is an exact complex.
	It is classically known that the Bardzell resolution $\bbar (\Lambda^w_n)$ is exact, so we need only focus on the elements in $\mathcal{L}_n$ for $n \geq 2$.
	
	Recall that $c\alpha \neq 0 \neq c\beta$ in $\Lambda^w_n$ for any $c$ such that $t(c) = s(\alpha) = s(\beta)$, and respectively, $\alpha b \neq 0 \neq \beta b$ in $\Lambda^w_n$ for any $b$ such that $s(b) = t(\alpha) = t(\beta)$.
	Therefore, any non-zero element $c \otimes \delta \otimes b$ is not a cocycle for any $\delta \in \{(\alpha\beta)^i,(\beta\alpha)^i\}$ if $n=2i$ (resp. $\delta \in \{(\alpha\beta)^i\alpha,(\beta\alpha)^i\beta\}$).
	
	Now consider a non-zero element $x = c \otimes \delta \otimes b + (-1)^{n+1} c' \otimes \delta' \otimes b'$, and apply $d_n$,
	\[
	d_n(x) = c\delta_1 \otimes \delta_2\cdots \delta_n \otimes b + (-1)^n c \otimes \delta_1 \cdots \delta_{n-1} \otimes \delta_n b - (-1)^n c'\delta'_1 \otimes \delta'_2\cdots \delta'_n \otimes b' - c' \otimes \delta'_1 \cdots \delta'_{n-1} \otimes \delta'_n b'.
	\]
	Then $x$ is a cocycle if $c\delta_1 \otimes \delta_2\cdots \delta_n \otimes b = c' \otimes \delta'_1 \cdots \delta'_{n-1} \otimes \delta'_n b'$.
	One can check that this implies $d_n(x) = 0$ by recognising that $\delta_2 = \delta'_1$, and that both $\delta_i\delta_{i+1}$ and $\delta'_i\delta'_{i+1}$ are equivalent to a trivial path in $\Lambda^w_n$.
	It then follows that $x = d_{n+1} (c' \otimes \delta' \delta'_{n+1} \otimes b)$, and so $x$ is a coboundary.
	
	Further $d_n(x) = 0$ if $c\delta_1 \otimes \delta_2\cdots \delta_n \otimes b = c'\delta'_1 \otimes \delta'_2\cdots \delta'_n \otimes b$, and $n$ is even, or $\mathrm{char}\,k = 2$.
	However, this implies that $b=b'$, and $\delta_2\cdots \delta_n  =\delta'_2\cdots \delta'_n $, hence $\delta_1 = \delta_1'$, so $c =c'$.
	This then means that $x = c \otimes \delta \otimes b + (-1)^{n+1} c \otimes \delta \otimes b$, which is a zero element if $n$ is even, or $\mathrm{char}\, k = 2$, and so $x$ is already a coboundary.
	
	Using the same arguments as above, if $x$ is a cocycle with direct summand $c \otimes \delta \otimes b$, then $(-1)^{n+1} c' \otimes \delta' \otimes b'$ must be a direct summand of $x$. Therefore we may conclude that the class of cocycles and coboundaries coincide for $n > 1$.
	
	It is left to show that $\mathrm{Im}\, \pi = \Lambda^w_n$, and that $\mathrm{Im}\, d_1 = \ker \pi$.
	It is straightforward to check that for $q \in \Lambda^w_n$, there exists $q \otimes t(q) \otimes 1 \in \mathcal{P}_1$ such that $\pi(q \otimes t(q) \otimes 1) = q$, hence $\mathrm{Im}\, \pi = \Lambda$.
	Therefore, we must finally check that $\mathrm{Im}\, d_1 = \ker \pi$.
	By \Cref{Lem: Bardzell complex}, it is sufficient to check that a cocycle $x$ is a couboundary if $\pi(x)$ is a loop relation.
	Suppose $c \otimes e \otimes b$ is a cocycle, then $cb \in I$, and so $c \otimes e \otimes b = d_2(1 \otimes c \otimes b)$, and hence is a coboundary.
	
	Next, suppose $\pi(x) = \alpha r \beta - r \in I$, then it is sufficient to show that $\alpha r' \otimes e \otimes r'' \beta - 1 \otimes e \otimes r$ is a coboundary for all $r',r''$ such that $r'r'' = r$.
	First, consider that $\alpha \otimes e \otimes r \beta - 1 \otimes e \otimes r$ is a coboundary, as it is equal to $d_1(1 \otimes \alpha \otimes r\beta)$.
	Then it is straightforward to check that
	\[
	d_1 \biggl( 1 \otimes \alpha \otimes r \beta + \sum_{j=1}^l \alpha r_j' \otimes r_j \otimes r_j''\beta  \biggr) = \alpha r' \otimes e \otimes r'' \beta - 1 \otimes e \otimes r,
	\]
	where $r_j'r_jr_j'' = r$.
	Hence $\alpha r' \otimes e \otimes r'' \beta - 1 \otimes e \otimes r$ is a coboundary.
	The cases of $\pi(x) = \beta s \alpha - s, \alpha t - t \alpha, \beta u - u \beta \in I$ are all analogous.
	
	Therefore, $\mathrm{Im}\, d_1 = \ker \pi$, and so $\lbar$ is a projective resolution of $\Lambda^w_n$.
\end{proof}

\subsection{Parallel Paths}

Given a projective resolution of $\Lambda^w_n$ as a $\Lambda^w_n$-$\Lambda^w_n$ bimodule, to compute the Hochschild cohomology of $\Lambda^w_n$ we must apply the contravariant Hom functor $\Hom[(\Lambda^w_n)^e]{-}{\Lambda^w_n}$ to the projective resolution.
The Hochschild cohomology of $\Lambda^w_n$ is then the cohomology of the resulting complex, and so we have,
\[
\coHH{\bullet}(\Lambda^w_n) \cong \cohom{\bullet}( \Hom[(\Lambda^w_n)^e]{\lbar}{\Lambda^w_n}).
\]

For a graded gentle algebra $A$, it is classical to consider $\Hom[A^e]{\bbar[n]}{A}$ as the vector space spanned the set of parallel paths $(P_n || \mathcal{B})$, where
\[
(P_n || \mathcal{B}) := \{(p,q) \in P_n \times \mathcal{B} \mid s(p) = s(q) \text{ and } t(p) = t(q)\},
\]
and $\mathcal{B}$ is the basis of the gentle presentation of $A$.
We induce a grading on $(P_n || \mathcal{B})$ by placing $(p,q)$ in degree $\lvert q \rvert - \lvert p \rvert$.
It was shown by Ciblis \cite{Ciblis} that there is an isomorphism of graded $k$-vector spaces
\[
\varphi : k (P_n || \mathcal{B}) \xlongrightarrow{\sim} \Hom[A^e]{\bbar[n]}{A},
\]
that takes a pair $(p,q) \in (P_n || \mathcal{B})$ to the morphism 
\[
f_{(p,q)}(c \otimes a \otimes b) = \begin{cases*}
	(-1)^{\lvert c \rvert (\lvert q \rvert - \lvert p \rvert)} cqb & if $p=a$,\\
	0 & else.
\end{cases*}
\]

For the case of a graded endomorphism algebra $\Lambda^w_n$ of a $w$-linear generator, we must consider the additional loops and their relations.
However, we can still show an analogous result.

\begin{lemma}\label[lem]{Lem: Bardzell differential}
	Consider the dg algebra $\Lambda^w_n$.
	There is an isomorphism of graded $k$-vector spaces from $k(\mathcal{P}_n || \mathcal{B})$ to $\Hom[(\Lambda^w_n)^e]{\lbar[n]}{\Lambda^w_n}$.
	Moreover, there exists an isomorphism of complexes from $k(\mathcal{P}_{\bullet} || \mathcal{B})$ to $\Hom[(\Lambda^w_n)^e]{\lbar}{\Lambda^w_n}$, where $k(\mathcal{P}_{\bullet} || \mathcal{B})$ has a differential $d^n$ given by
	\[
	d^n (p,q) = \sum_{b \in Q_1,bp \in \mathcal{P}_{n+1}} (-1)^{\lvert b \rvert (\lvert q \rvert - \lvert p \rvert)} (bp,bq) - (-1)^n \sum_{a \in Q_1, pa \in \mathcal{P}_{n+1}} (pa,qa).
	\]
\end{lemma}

\begin{proof}
	Consider the morphism
	\begin{align*}
		\varphi_n \colon & k(\mathcal{P}_n || \mathcal{B}) \longrightarrow \Hom[(\Lambda^w_n)^e]{\lbar[n]}{\Lambda}\\
		& (p,q) \longmapsto f_{(p,q)}(c \otimes a \otimes b) = \begin{cases*}
			(-1)^{\lvert c \rvert (\lvert q \rvert - \lvert p \rvert)} cqb & if $p=a$,\\
			0 & else.
		\end{cases*}
	\end{align*}
	We claim that this is an isomorphism for all $n \geq 0$.
	It is clear that $\varphi_n$ has a natural inverse taking $f_{(p,q)}$ to $(p,q)$, so it remains to show that $\mathrm{Im}(\varphi_n) = \Hom[(\Lambda^w_n)^e]{\lbar[n]}{\Lambda^w_n}$.
	
	Let $f \in \Hom[(\Lambda^w_n)^e]{\lbar[n]}{\Lambda^w_n}$, and consider that as $f$ is a $(\Lambda^w_n)^e$ morphism, then $f$ is determined exclusively by its action on elements of the form $1 \otimes a \otimes 1$.
	Similarly, if we have
	\[
	f(c \otimes a \otimes b) = c f(1 \otimes a \otimes 1) b \neq 0,
	\]
	then $f(1 \otimes a \otimes 1) = \mu q \in e_{s(a)} \Lambda e_{t(a)}$, with $\mu \in k^{\times}$.
	Hence $q \in \mathcal{B}$ is a parallel path to $a \in \mathcal{P}_n$, and so $f$ is a linear combination of morphisms of the form $f_{(p,q)}$ for $(p,q) \in (\mathcal{P}_n || \mathcal{B})$.
	
	To check that $\varphi_n$ preserves grading, consider a morphism $f_{(p,q)} \in \Hom[(\Lambda^w_n)^e]{\lbar[n]}{\Lambda^w_n}$ of degree $m$.
	Then,
	\[
	m = \lvert cqb \rvert - \lvert c \otimes a \otimes b \rvert = \lvert c \rvert + \lvert q \rvert + \lvert b \rvert - \lvert c \rvert - \lvert p \rvert - \lvert b \rvert = \lvert q \rvert - \lvert p \rvert,
	\]
	as $p = a$, and so $\varphi_n$ preserves grading.
	
	To prove our final claim, we must check that the following diagram commutes for all $n \geq 0$,
	\[
	\begin{tikzcd}
		k(\mathcal{P}_n || \mathcal{B}) \arrow[rr,"d^n"] \arrow[dd,"\varphi_n"] && k(\mathcal{P}_{n+1} || \mathcal{B}) \arrow[dd,"\varphi_{n+1}"]\\
		\\
		\Hom[(\Lambda^w_n)^e]{\lbar[n]}{\Lambda} \arrow[rr,"d^{\ast}_{n+1}"] && \Hom[(\Lambda^w_n)^e]{\lbar[n+1]}{\Lambda},
	\end{tikzcd}
	\]
	where $d^{\ast}_{n+1} = \Hom[(\Lambda^w_n)^e]{d_{n+1}}{\Lambda^w_n}$.
	Commutativity of the above diagram is equivalent to asking that the following diagram commutes for all $(p,q) \in k(\mathcal{P}_n || \mathcal{B})$ and $c \otimes a \otimes b \in \lbar[n]$,
	\[
	\begin{tikzcd}
		c \otimes a \otimes b \arrow[rrrr, "g =\sum_{y \in Q_1,yp \in \mathcal{P}_{n+1}} (-1)^{\lvert y \rvert (\lvert q \rvert - \lvert p \rvert)} f_{(yp,yq)} - (-1)^n \sum_{x \in Q_1, px \in \mathcal{P}_{n+1}} f_{(px,qx)}"] \arrow[rrdd,"d_{n+1}"] &&&& f_{(p,q)}(ca_1 \otimes a' \otimes b + (-1)^n c \otimes a'' \otimes a_n b)\\
		\\
		&& ca_1 \otimes a' \otimes b + (-1)^n c \otimes a'' \otimes a_n b. \arrow[rruu,"f_{(p,q)}"] &&
	\end{tikzcd}
	\]
	Suppose that $a',a'' \neq p$, then $a \neq yp,px$, and so the diagram commutes as $g(c \otimes a \otimes b) = f_{(p,q)}(ca_1 \otimes a' \otimes b + (-1)^n c \otimes a'' \otimes a_n b) = 0$.
	
	Further, suppose $a' = p$ and $a'' \neq p$, then $a=yp$ for $a_1 = y$, and so 
	\begin{align*}
		g(c \otimes a \otimes b) & = (-1)^{\lvert a_1 \rvert (\lvert q \rvert - \lvert p \rvert)}f_{(a_1p,a_1q)}(c \otimes a \otimes b)\\
		& = (-1)^{\lvert a_1 \rvert (\lvert q \rvert - \lvert p \rvert)}(-1)^{\lvert c \rvert (\lvert a_1q \rvert - \lvert a_1p \rvert)} ca_1qb\\
		&=(-1)^{\lvert ca_1 \rvert(\lvert q \rvert - \lvert p \rvert)} ca_1qb\\
		& = f_{(p,q)}(c a_1 \otimes a' \otimes b)\\
		& = f_{(p,q)}(c a_1 \otimes a' \otimes b +(-1)^n c \otimes a'' \otimes a_n b).
	\end{align*}
	An analogous calculation holds when $a'' = p$ and $a' \neq p$.
	
	Finally, suppose $a'=p=a''$.
	Then a combination of the above calculations shows that the above diagram commutes, and so $\varphi$ is an isomorphism of complexes of graded $k$-vector spaces.
\end{proof}

\subsection{Hochschild Cohomology}

\begin{proposition}\label[prop]{Prop: HH 0}
	The vector space $\mathrm{HH}^{0,q}(\Lambda^w_n)$, for $q \in w\bZ$, is freely generated by the elements
	\begin{itemize}
		\item $\mathbbm{1}_q:  = \sum_{i \in Q_0} (\iota_i,\alpha_i^{\frac{-q}{w}})$ for $q \leq 0$ and even, or $\mathrm{char}\, k = 2$,
		\item $\mathbbm{1}_q:  = \sum_{i \in Q_0} (\iota_i,\beta^{\frac{q}{w}})$ for $q > 0$ and even, or $\mathrm{char}\, k = 2$, and $B = \emptyset$.
	\end{itemize}
	Moreover, if $q \not \in w\bZ$, then $\mathrm{HH}^{0,q}(\Lambda^w_n) = 0$.
\end{proposition}

\begin{proof}
Our second claim is straightforward, as $k (\mathcal{P}_0 || \mathcal{B})_q = 0$ if $q \not\in w\bZ$, and $\mathcal{P}_0$ is concentrated in degree zero and $\mathcal{B}$ is concentrated in degrees in $w\bZ$.
Hence $\mathrm{ker}\, d^0 = 0$.

Therefore, we may consider the case where $q \in w\bZ$, and compute $\mathrm{ker}\, d^0$.
We focus on the case where $q < 0$ for now, and let $p = \frac{-q}{w}$.
The irreducible elements of $k (\mathcal{P}_0 || \mathcal{B})_0$ are of the form $(\iota_i,\iota_i)$, and clearly these are not in the kernel of $d^0$, as
\[
d^0(\iota_i,\alpha_i^p) = (\delta_i,\delta_i\alpha_i^p) - (\delta_{i-1},\alpha_i^p\delta_{i-1}) + ((-1)^q -1)((\alpha_i,\alpha^{p-1}_i) + (\beta_i,\alpha_i^{p-1})),
\]
where $(\delta_0, \alpha_1^p \delta_0) = 0 = (\delta_n, \delta_n \alpha_n^p)$ and $(\beta_i, \alpha_i^{p-1}) = 0$ if $i \in B$.

Now consider a linear combination $u = \sum_{i \in Q_0} x_i (\iota_i,\alpha_i^p)$ for $x_i \in k$.
If $q$ is odd, and $\mathrm{char}\, k \neq 2$, then it is clear all $x_i = 0$ as otherwise
we would have $-2x_i((\alpha_i,\alpha^{p-1}_i) + (\beta_i,\alpha_i^{p-1}))$ as a summand of $d^0(u)$ for all $i \in Q^G_0$.
Hence, we need only consider when $q$ is even, or $\mathrm{char}\, k = 2$.
In this case, we get
\[
d^0(u) = \sum_{i =1}^{n-1} x_i((\delta_i,\delta_i\alpha_i^p)) + \sum_{i =1}^{n-1} x_{i+1}((\delta_i,\alpha_i^p\delta_i)).
\]
It is clear to see that $d^0(u) = 0$ if and only if $x_i = x_{i+1}$ for all $i =1, \ldots, n-1$, hence $u = x_1 \cdot \mathbbm{1}_q$.
Therefore $\coHH{0,q} = \langle \mathbbm{1}_q \rangle$ for $0 > q \in w\bZ$ and even, or $\mathrm{char}\, k = 2$.

Finally, to see that $\coHH{0,q} = \langle \mathbbm{1}_q \rangle$ for $q > 0$ and even, or $\mathrm{char}\, k = 2$, if and only if $B = \emptyset$, then consider that $x_i$ in the above argument must be zero if $i \in B$.
This is because $(\iota_i, \beta^{\frac{q}{w}})$ does not exist if $i \in B$.
\end{proof}

\begin{lemma}\label[lem]{Lem: loop cocycle}
	Suppose that $P_m = 0$ and let $n \geq m$. Then $x \in (\mathcal{P}_n|| \mathcal{B})$ is a cocycle if and only if $x$ is a linear combination of $((\alpha\beta)^i,\alpha q) + (-1)^{\lvert q \rvert}((\beta\alpha)^i,q\alpha)$ and $((\alpha\beta)^i,\beta q) + (-1)^{\lvert q \rvert}((\beta\alpha)^i,q\beta)$ if $n=2i$, resp. $((\alpha\beta)^i\alpha,\alpha q) + (-1)^{\lvert q \rvert}((\beta\alpha)^i\beta,q\beta )$ and $((\alpha\beta)^i\alpha,\beta q) + (-1)^{\lvert q \rvert}((\beta\alpha)^i\beta,q\alpha )$ if $n=2i+1$.
\end{lemma}

\begin{proof}
	Suppose $n=2i$, then 
	\[
	d^n(((\alpha\beta)^i,\alpha q) + (-1)^{\lvert q \rvert}((\beta\alpha)^i,q\alpha)) = (-1)^{\lvert q \rvert}((\beta\alpha)^i\beta,q) - ((\alpha\beta)^i\alpha,\alpha q \alpha) + ((\alpha\beta)^i\alpha,\alpha q\alpha ) - (-1)^{\lvert q \rvert}((\beta\alpha)^\beta,q)
	=0.
	\]
	Hence if $n=2i$ and $x$ is a linear combination $((\alpha\beta)^i,\alpha q) + (-1)^{\lvert q \rvert}((\beta\alpha)^i,q\alpha)$ and $((\alpha\beta)^i,\beta q) + (-1)^{\lvert q \rvert}((\beta\alpha)^i,q\beta)$, then $x$ is a cocycle.
	A similar computation shows that if $x$ is a linear combination of$((\alpha\beta)^i\alpha,\alpha q) + (-1)^{\lvert q \rvert}((\beta\alpha)^i\beta,q\beta )$ and $((\alpha\beta)^i\alpha,\beta q) + (-1)^{\lvert q \rvert}((\beta\alpha)^i\beta,q\alpha )$ and $n=2i+1$, then $x$ is a cocycle.
	
	We must now show that if $x$ is a cocycle, then it is a linear combination of these elements.
	Again, we prove the statement for $n=2i$, where the case of $n=2i+1$ is analogous.
	Suppose that $((\alpha\beta)^i,\alpha q)$ is a summand of $x$, then $x = ((\alpha\beta)^i,\alpha q) + x'$, and so 
	\[
	d^n(x') = -d^n((\alpha\beta)^i,\alpha q) = (-1)^{\lvert q \rvert + 1}((\beta\alpha)^i\beta,q) + ((\alpha\beta)^i\alpha,\alpha q \alpha).
	\]
	However, by the definition of $d^n$, $(-1)^{\lvert q \rvert + 1}((\beta\alpha)^i\beta, q) + ((\alpha\beta)^i\alpha,\alpha q\alpha)$ is in the image of $d^n(x')$ if and only if $(-1)^{\lvert q \rvert}((\beta\alpha)^i\beta,q\alpha)$ is a summand of $x'$.
	Hence, if $x$ is a cocycle for $n=2i$, then $x$ is a linear combination of $((\alpha\beta)^i,\alpha q) + (-1)^{\lvert q \rvert}((\beta\alpha)^i,q\alpha)$ and $((\alpha\beta)^i,\beta q) + (-1)^{\lvert q \rvert}((\beta\alpha)^i,q\beta)$.
	
	Our claim is proven by considering an analogous argument for when $n=2i+1$.
\end{proof}

\begin{lemma}\label[lem]{Lem: loops are exact}
	Suppose that $n \geq 2$, and let $x = (p,q) \in k (\mathcal{P}_n|| \mathcal{B})$ such that $p \in \mathcal{L}_n$.
	Then $x$ is a coboundary if and only if $x$ is a cocycle.
\end{lemma}

\begin{proof}
	As $(\mathcal{P}_{\bullet} || \mathcal{B})$ is a complex, it is clear that $x$ is a cocycle if it is a coboundary.
	Hence it is left to show that $x$ is a coboundary if it is a cocycle.
	
	Suppose $n=2i$, then we have 
	\[
	d^{n-1}(-(\beta\alpha)^{i-1}\beta, q) = -(- ((\alpha\beta)^i,\alpha q) + (-1)^{\lvert q \rvert-1}((\beta\alpha)^i,q\alpha)) = ((\alpha\beta)^i,\alpha q) + (-1)^{\lvert q \rvert}((\beta\alpha)^i,q\alpha),
	\]
	and 
	\[
	d^{n-1}(-(\alpha\beta)^{i-1}\alpha, q) = -(- ((\beta\alpha)^i,\beta q) + (-1)^{\lvert q \rvert-1}((\alpha\beta)^i,q\beta)) = ((\beta\alpha)^i,\beta q) + (-1)^{\lvert q \rvert}((\alpha\beta)^i,q\beta).
	\]
	Hence, by \Cref{Lem: loop cocycle}, $x$ is a cocycle if and only if it is a coboundary.
	
	An analogous calculation holds for $n=2i+1$, and so $x$ is a cocycle if and only if it is a coboundary for all $n \geq 2$.
\end{proof}

\begin{proposition}\label[prop]{Prop: HH 1}
	Let $\Lambda^w_n$ for some $w$-linear generator $G$.
	Then the vector space $\coHH{1,q}(\Lambda^w_n)$, for $q \in w\bZ$ is freely generated by the elements
	\begin{itemize}
		\item $(\alpha_i, \alpha_i \rho) + (\beta_i, \beta_i \rho)$ for each $n \neq i \not \in B$,
		\item $(\alpha_n, \alpha_n \rho) + (\beta_n, \beta_n \rho)$ if $n \not \in B$, and $\lvert \rho \rvert \in 2\bZ$ or $\mathrm{char}\, k =2$,
		\item $(\beta_n, \beta_n \rho)$ if $n \not \in B$, and $\lvert \rho \rvert \not \in 2\bZ$ and $\mathrm{char}\, k \neq 2$,
		\item $(\alpha_i,\alpha_i \rho)$ for each $n \neq i \in B$,
		\item $(\alpha_n,\alpha_n \rho)$ for $i = n \in B$, and $\lvert \rho \rvert \not \in 2\bZ$ and $\mathrm{char}\, k \neq 2$,
	\end{itemize}
	where $\rho \in \{\alpha^{\frac{-q}{w}},\beta^{\frac{-q}{w}}\}$.
	
	Moreover, if $q \not \in w\bZ$, then $\coHH{1,q}(\Lambda^w_n) = 0$.
\end{proposition}

\begin{proof}
	It is clear that if $q \not \in w\bZ$, then $\coHH{0,q}(\Lambda) = 0$, as for any pair of parallel paths $(\sigma, \varepsilon)$ we have $\lvert \sigma \rvert , \lvert \varepsilon \rvert \in w\bZ$, and so $\lvert \varepsilon \rvert - \lvert \sigma \rvert \in w\bZ$.
	Hence $q \neq \lvert \varepsilon \rvert - \lvert \sigma \rvert$.
	Therefore, we may continue by considering the case where $q \in w\bZ$.
	
	Let us start by computing the elements in $\mathrm{ker}\, d^1$.
	Consider the pair $(a,\varepsilon)$ for $\delta \in Q_1$.
	As there are no paths in $\mathcal{P}_2$ containing $\delta$, then clearly we have $d^1(a,\varepsilon) = 0$ for all $\delta \in Q_1$, and all $\varepsilon$.
	Additionally, as $\varepsilon$ is parallel to $\delta$, and the underlying quiver without loops is of linear type $A$, then $\varepsilon$ is equivalent to a path with $a$ as a subpath, and a fixed number of loops depending upon $q$.
	Hence, $\varepsilon$ is unique under the relations in $I^G$ for each $q \in w\bZ$.
	
	Next, we shall consider the pair $(\alpha_i, \alpha_i \rho)$.
	If $i \in B$, then there does not exist an arrow $\beta_i$, hence $d^1(\alpha_i, \alpha_i \rho) = 0$, and $\rho$ is again unique under the relations in $I^G$ for each $q \in w\bZ$.
	However, if $i \not\in B$, then there does exist an arrow $\beta_i$, and so we get
	\[
	d^1 (\alpha_i, \rho) = (-1)^{\lvert \rho \rvert -w} (\beta_i \alpha_i, \beta_i \rho) + (\alpha_i \beta_i, \rho \beta_i).
	\]
	So $(\alpha_i, \rho) \not \in \mathrm{ker}\, d^1$ if $i \not \in B$, and an analogous calculation holds for $(\beta_i, \rho)$ to see that is not in $\mathrm{ker}\, d^1$ either.
	
	However, consider the element $(\alpha_i, \alpha_i \rho) - (-1)^{\lvert \rho \rvert}(\beta_i, \beta_i \rho)$, then we get
	\[
	d^1 ((\alpha_i, \alpha_i \rho) + (\beta_i, \beta_i \rho)) = (-1)^{\lvert \rho \rvert} (\beta_i \alpha_i, \beta_i \alpha_i \rho) + (\alpha_i \beta_i, \alpha_i \rho \beta_i) - (-1)^{\lvert \rho \rvert} (\beta_i \alpha_i, \beta_i \rho \alpha_i) - (\alpha_i \beta_i, \alpha_i \beta_i \rho) = 0,
	\]
	where the equality comes from $\alpha_i \beta_i \rho = \alpha_i \rho \beta_i$ and $\beta_i \rho \alpha_i = \beta_i \alpha_i \rho$ with respect to the relations $I^G$.
	Hence $(\alpha_i, \alpha_i \rho) - (-1)^{\lvert \rho \rvert -1}(\beta_i, \beta_i \rho) \in \mathrm{ker}\, d^1$ if $i \not \in B$.
	
	To see that every element of $\mathrm{ker}\, d^1$ is a linear combination of the elements we have considered so far, suppose that we have some element $u \in \mathrm{ker}\, d^1$ that is not a linear combination of the above elements.
	Without loss of generality, we may assume that $(\delta,\varepsilon)$ and $(\alpha_i, \alpha_i \rho)$ are not summands of $u$ for any $\delta \in Q^G_1$ or $i \in B$.
	Therefore, $u$ must be a linear combination of elements of the form $(\alpha_i, \rho)$ and $(\beta_i, \rho)$ for $i \not \in B$, as these are the only other irreducible elements of $k(\mathcal{P}_2 || \mathcal{B})$.
	Therefore, suppose that $u = u' + (\beta_i,\beta_i \rho)$ and compute
	\[
	d^1(u) = d^1(u') + d^1(\beta_i,\beta_i \rho) =0,
	\]
	hence we get
	\[
	d^1(u') = -d^1(\beta_i, \beta_i \rho) = (-1)^{\lvert \rho \rvert}(\alpha_i \beta_i, \rho) + (\beta_i \alpha_i, \rho).
	\]
	However, the only elements for which $(\alpha_i \beta_i, \rho)$ appears as a summand after applying $d^1$ are the elements $(\beta_i, \beta_i \rho)$ and $(-1)^{\lvert \rho \rvert}(\alpha_i,\alpha_i \rho)$, as no other elements in $k(\mathcal{P}_1 || \mathcal{B})$ contain $\alpha_i$ or $\beta_i$ in $\mathcal{P}_1$ and are of the correct degree.
	Hence either
	\[
	u = u + (\beta_i,\beta_i \rho) - (\beta_i,\beta_i \rho), \quad \mathrm{or} \quad u = u'' + (\beta_i,\beta_i \rho) + (-1)^{\lvert \rho \rvert}(\alpha_i,\alpha_i \rho),
	\]
	and so $u$ is a linear combination of the above elements via an argument of infinite descent.
	That is
	\[
	\mathrm{ker}\, d^1 = \langle (\delta,\varepsilon), (\alpha_i, \alpha_i \rho), (\beta_j,\beta_j \rho) + (-1)^{\lvert \rho \rvert}(\alpha_j,\alpha_j \rho) \mid i \in B, j \not in B \rangle.
	\]
	
	We now consider the image of $d^0$.
	As all elements of $k(\mathcal{P}_0|| \mathcal{B})$ are linear combinations of elements of the form $(\iota_i,\rho)$, then we may only consider the images of each of these elements.
	Then we get
	\[
	d^0 (\iota_i, \rho) = (\delta_i, \delta_i \rho) - (\delta_{i-1},\rho \delta_{i-1}) + ((-1)^{\lvert \rho \rvert} - 1)((\alpha_i,\alpha_i \rho) + (\beta_i, \beta_i \rho)),
	\]
	for all $i \in Q^G_0$, where $(\delta_0, \delta_0 \rho) = 0 = (\delta_n, \delta_n \rho)$ and $(\beta_i, \beta_i \rho) = 0$ if $i \in B$.
	Hence, if $\lvert \rho \rvert \in 2\bZ$ or $\mathrm{char}\, k =2$, we have
	\[
	\mathrm{Im}\, d^0 = \langle (\delta_1,\delta_1 \rho), (\delta_{n-1},\rho \delta_{n-1}), (\delta_i, \delta_i \rho) - (\delta_{i-1},\rho \delta_{i-1})) \mid i = 2,\ldots, n-1 \rangle.
	\]
	Otherwise we have 
	\begin{align*}
	\mathrm{Im}\, d^0 = & \langle (\delta_1,\delta_1 \rho) + (\alpha_1,\alpha_1 \rho) + (\beta_1, \beta_1 \rho), (\delta_{n-1},\rho \delta_{n-1}) + (\alpha_n,\alpha_n \rho) + (\beta_n, \beta_n \rho),\\
	&(\delta_i, \delta_i \rho) - (\delta_{i-1},\rho \delta_{i-1}) +2((\alpha_i,\alpha_i \rho) + (\beta_i, \beta_i \rho)) \mid i = 2,\ldots, n-1 \rangle.
	\end{align*}
	
	Finally, it is left for us to compute the vector space $\mathrm{ker}\, d^1 / \mathrm{Im}\, d^0$.
	In the case of $\lvert \rho \rvert \in 2\bZ$ or $\mathrm{char}\, k =2$ it is straightforward, as all elements $(\delta,\varepsilon) \in \mathrm{ker}\, d^1$ are sent to zero by the relations in $\mathrm{Im}\, d^0$ and no other elements are affected, so we get
	\[
	\coHH{1,q}(\Lambda^w_n) = \langle  (\alpha_i, \alpha_i \rho), (\beta_j,\beta_j \rho) + (-1)^{\lvert \rho \rvert}(\alpha_j,\alpha_j \rho) \mid i \in B, j \not \in B \rangle,
	\]
	where $q = \lvert \rho \rvert$.
	
	However, if $\lvert \rho \rvert \not \in 2\bZ$ or $\mathrm{char}\, k \neq 2$, then we find the relation
	\[
	(\delta_2,\delta_2 \rho) + (\alpha_1,\alpha_1 \rho) + (\beta_1,\beta_1 \rho) + 2((\alpha_2,\alpha_2 \rho) + (\beta_2, \beta_2 \rho)),
	\]
	by combining the relations $(\delta_1,\delta_1 \rho) + (\alpha_1,\alpha_1 \rho) + (\beta_1, \beta_1 \rho)$ and $(\delta_2, \delta_2 \rho) - (\delta_1,\rho \delta_1) +2((\alpha_2,\alpha_2 \rho) + (\beta_2, \beta_2 \rho))$.
	By further combining relations inductively, we find that we have the relation
	\[
	(\delta_i,\delta_i \rho) + (\alpha_1,\alpha_1 \rho) + (\beta_1,\beta_1 \rho) + 2 \sum_{j=2}^i((\alpha_j,\alpha_j \rho) + (\beta_j, \beta_j \rho)),
	\]
	for all $i = 2,\ldots, n-1$.
	Hence we find that 
	\[
	(\delta_i,\delta_i \rho) = -(\alpha_1,\alpha_1 \rho) - (\beta_1,\beta_1 \rho) - 2 \sum_{j=2}^i((\alpha_j,\alpha_j \rho) + (\beta_j, \beta_j \rho))
	\]
	in $\mathrm{ker}\, d^1 / \mathrm{Im}\, d^0$.
	
	However, we must consider the final relation $ (\delta_{n-1},\rho \delta_{n-1}) + (\alpha_n,\alpha_n \rho) + (\beta_n, \beta_n \rho)$.
	By combining this relation with the above identification of $ (\delta_{n-1},\rho \delta_{n-1}) $ we get the relation,
	\[
	(\alpha_n,\alpha_n \rho) + (\beta_n, \beta_n \rho) -(\alpha_1,\alpha_1 \rho) - (\beta_1,\beta_1 \rho) - 2 \sum_{j=2}^i((\alpha_j,\alpha_j \rho) + (\beta_j, \beta_j \rho)).
	\]
	Hence we find our final relation
	\[
	(\alpha_n,\alpha_n \rho) = (\alpha_1,\alpha_1 \rho) - (\beta_1,\beta_1 \rho) + 2 \sum_{j=2}^i((\alpha_j,\alpha_j \rho) + (\beta_j, \beta_j \rho)) - (\beta_n, \beta_n \rho).
	\]
	Therefore, we get
	\[
	\coHH{1,q}(\Lambda^w_n) = \langle  (\alpha_i, \alpha_i \rho), (\beta_j,\beta_j \rho) + (-1)^{\lvert \rho \rvert}(\alpha_j,\alpha_j \rho), (\beta_n, \beta_n \rho) \mid n \neq i \in B, n \neq j \not \in B \rangle,
	\]
	where $q = \lvert \rho \rvert$, and $(\beta_n, \beta_n \rho) \neq 0$ if and only if $n \not \in B$.
\end{proof}

\begin{theorem}\label{Thm: HH}
	Let $G$ be a $w$-linear generator in a triangulated category $\cT$ satisfying \Cref{setup:w cluster cat}, and let $\Lambda^w_n$ be the graded endomorphism ring.
	Then the Hochschild cohomology $\coHH{r,q}(\Lambda^w_n)$ is generated by the following elements;
	\begin{itemize}
		\item $\mathbbm{1}_q:  = \sum_{i \in Q_0} (\iota_i,\alpha_i^{\frac{-q}{w}})$ for $q \leq 0$ and even, or $\mathrm{char}\, k = 2$, and $r=0$,
		\item $\mathbbm{1}_q:  = \sum_{i \in Q_0} (\iota_i,\beta^{\frac{q}{w}})$ for $q > 0$ and even, or $\mathrm{char}\, k = 2$, $B = \emptyset$, and $r=0$,
		\item $(\alpha_i, \alpha_i \rho) + (\beta_i, \beta_i \rho)$ for each $n \neq i \not \in B$ and $r =1$,
		\item $(\alpha_n, \alpha_n \rho) + (\beta_n, \beta_n \rho)$ if $n \not \in B$, and $\lvert \rho \rvert \in 2\bZ$ or $\mathrm{char}\, k =2$ and $r =1$,
		\item $(\beta_n, \beta_n \rho)$ if $n \not \in B$, and $\lvert \rho \rvert \not \in 2\bZ$ and $\mathrm{char}\, k \neq 2$ and $r =1$,
		\item $(\alpha_i,\alpha_i \rho)$ for each $n \neq i \in B$ and $r =1$,
		\item $(\alpha_n,\alpha_n \rho)$ for $i = n \in B$, and $\lvert \rho \rvert \not \in 2\bZ$ and $\mathrm{char}\, k \neq 2$ and $r =1$.
	\end{itemize}
	Moreover, $\coHH{r,q}(\Lambda^w_n) = 0$ for all the cases not covered by the above list.
\end{theorem}

\begin{proof}
	This is a combination of \Cref{Prop: HH 0,Prop: HH 1} and \Cref{Lem: loops are exact}.
\end{proof}

\begin{corollary}\label[cor]{Cor: Intrinsically formal}
	The vector space $\mathrm{HH}^{r,m}(\Lambda^w_n) = 0$ for all $r \geq 2$.
	Moreover, the algebra $\Lambda^w_n$ is intrinsically formal, and so $\cT$ is equivalent as triangulated categories to $\mathrm{perf}\, \Lambda^w_n$.
\end{corollary}

\begin{proof}
	Given the degree zero component of $\Lambda^w_n$ is isomorphic to the path algebra over the quiver of Dynkin type $A$, then there exists no anti-paths in $P_{\bullet}$ of length greater than 1.
	Hence we need only consider the parallel paths with anti-paths in $\mathcal{L}_r$, which we have shown to be exact in \Cref{Lem: loops are exact}, and so has trivial cohomology.
	Therefore $\mathrm{HH}^{r,m}(\Lambda^w_n) = 0$ for all $r \geq 2$.
	
	Our second claim follows from a result of Kadeishvili (see Roitzheim--Whitehouse \cite{RoitzheimWhitehouse} for an English version).
\end{proof}

\begin{corollary}
	Let $\ocC$ be the Paquette--Y\i ld\i r\i m completion of the Igusa--Todorov cluster category of Dynkin type $A_{\infty}$ with $n$ accumulation points.
	Then there exists a triangulated equivalence,
	\[
	\ocC \xrightarrow{\sim} \mathrm{perf}(\Lambda_n).
	\]
\end{corollary}

\begin{proof}
	This follows from \Cref{Cor: Intrinsically formal} and \cite[Theorem 4.3.10]{MurphyThesis}
\end{proof}

%% file: CombinatorialModel.tex
\section{Combinatorial Model}\label{Sec: comb}

In this section we show that a triangulated category admitting a $w$-linear generator admits a combinatorial model.

\subsection{The Combinatorial Model of $w$-Linear Generator}\label[sec]{Subsec: comb}

Let $G$ be a $w$-linear generator for a triangulated category $\cT$ satisfying \Cref{setup:w cluster cat}.
Suppose that $G \cong \bigoplus_{i=1}^n P_i$, with each $P_i$ indecomposable and not pairwise isomorphic up to suspension.
Suppose that $P_i \cong P_i[w]$ for all $i \in I$, and let $m = \mathrm{sup}(\lvert I \rvert +1, n - \lvert I \rvert)$.

Let $\Sigma$ be a disc with an anticlockwise orientation on the boundary, and let $\sM_m \subset \partial \Sigma$ be a collection of marked points on the boundary of our disc satisfying;
\begin{enumerate}
	\item $\sM_m$ has infinitely many elements,
	\item $\sM_m \subset S^1$ is a discrete subset,
	\item $\sM_m$ satisfies the \textit{two-sided limit condition}, i.e.\ each $x \in S^1$ which is the limit of a sequence in $\sM_m$ is a limit of both an increasing and decreasing sequence from $\sM_m$ with respect to the cyclic order, we call these points \textit{accumulation points},
	\item there are exactly $m$ accumulation points.
\end{enumerate}

Note that these accumulation points are not in $\sM_m$, as they do not have defined successors and predecessors, meaning $\sM_m$ would not be discrete if they were to be included.
As in \cite{Gratz2017}, we may think of $\sM_m$ as the vertices of the $\infty$-gon.
For an admissible subset $\sM_m \in S^1$, we label the set of accumulation points $L(\sM_m)$, and give them a cyclic ordering induced by the orientation of $S^1$.

Let $\mathfrak{a} \in L(\sM_m)$ be an accumulation point, we define $\mathfrak{a}_+ \in L(\sM_m)$ (resp.\ $\mathfrak{a}_- \in L(\sM_m)$) to be the accumulation point such that $\mathfrak{a} \leq \mathfrak{b} \leq \mathfrak{a}^+$ with $\mathfrak{b} \in L(\sM_m)$ (resp.\ $\mathfrak{a}_- \leq \mathfrak{c} \leq \mathfrak{a}$ with $\mathfrak{c} \in L(\sM_m)$) implies $\mathfrak{b}=\mathfrak{a}$ or $\mathfrak{b}=\mathfrak{a}_+$ (resp.\ $\mathfrak{c}=\mathfrak{a}_-$ or $\mathfrak{c}=\mathfrak{a}$).
Moreover, we say that $\mathfrak{a}$ is its own \textit{successor} and \textit{predecessor}, that is $\mathfrak{a}=\mathfrak{a}^+ = \mathfrak{a}^-$.
Let $\osM[m] := L(\sM_m) \cup \sM_m$.

\begin{figure}[h]
	\centering
	\begin{tikzpicture}
		\draw (0,0) circle (2cm);
		
		\foreach[count = \i] \txt in {1,...,3}
		\draw[fill=white] ({2*sin(47+120*\i)},{2*cos(47+120*\i)}) circle (0.1cm);
		
		\foreach[count = \i] \txt in {1,...,7}
		\draw[thin] ({1.95*sin(47-(30/\i))},{1.95*cos(47-(30/\i))}) -- ({2.05*sin(47-(30/\i))},{2.05*cos(47-(30/\i))});
		
		\foreach[count = \i] \txt in {1,...,7}
		\draw[thin] ({1.95*sin(47+(30/\i))},{1.95*cos(47+(30/\i))}) -- ({2.05*sin(47+(30/\i))},{2.05*cos(47+(30/\i))});
		
		\foreach[count = \i] \txt in {1,...,7}
		\draw[thin] ({1.95*sin(167-(30/\i))},{1.95*cos(167-(30/\i))}) -- ({2.05*sin(167-(30/\i))},{2.05*cos(167-(30/\i))});
		
		\foreach[count = \i] \txt in {1,...,7}
		\draw[thin] ({1.95*sin(167+(30/\i))},{1.95*cos(167+(30/\i))}) -- ({2.05*sin(167+(30/\i))},{2.05*cos(167+(30/\i))});
		
		\foreach[count = \i] \txt in {1,...,7}
		\draw[thin] ({1.95*sin(287-(30/\i))},{1.95*cos(287-(30/\i))}) -- ({2.05*sin(287-(30/\i))},{2.05*cos(287-(30/\i))});
		
		\foreach[count = \i] \txt in {1,...,7}
		\draw[thin] ({1.95*sin(287+(30/\i))},{1.95*cos(287+(30/\i))}) -- ({2.05*sin(287+(30/\i))},{2.05*cos(287+(30/\i))});
		
		\foreach[count = \i] \txt in {1,2,3}
		\draw[thin] ({1.95*sin(77+15*\i)},{1.95*cos(77+15*\i)}) -- ({2.05*sin(77+15*\i)},{2.05*cos(77+15*\i)});
		
		\foreach[count = \i] \txt in {1,2,3}
		\draw[thin] ({1.95*sin(197+15*\i)},{1.95*cos(197+15*\i)}) -- ({2.05*sin(197+15*\i)},{2.05*cos(197+15*\i)});
		
		\foreach[count = \i] \txt in {1,2,3}
		\draw[thin] ({1.95*sin(317+15*\i)},{1.95*cos(317+15*\i)}) -- ({2.05*sin(317+15*\i)},{2.05*cos(317+15*\i)});
		
		\node at ({2.3*sin(47)},{2.3*cos(47)}) {\footnotesize $\mathfrak{a}$};
		\node at ({2.3*sin(92)},{2.3*cos(92)}) {\footnotesize $x^+$};
		\node at ({2.3*sin(107)},{2.3*cos(107)}) {\footnotesize $x$};
		\node at ({2.3*sin(122)},{2.*cos(122)}) {\footnotesize $x^-$};
	\end{tikzpicture}
	\caption{An admissible subset $\sM_m$ of $\Sigma$. The marked points in $\sM_m$ converge to the accumulation points represented as small circles, and each marked point $x$ has both a predecessor and a successor, labelled $x^-$ and $x^+$ respectively.}
	\label{fig:w admissable}
\end{figure}
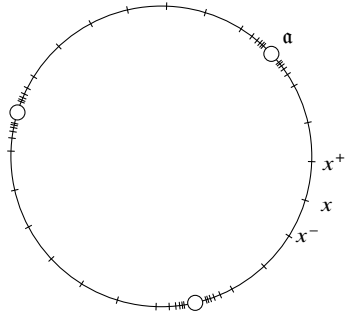

We wish to fix a labelling on the collection of marked points and accumulation points $\osM[m]$ according to the $w$-linear generator $G$, and to do this we begin by choosing an accumulation point $\mathfrak{a}$.
We now fix our labelling iteratively on the indecomposable direct summands $P_1,\ldots,P_n$, working in an anticlockwise direction from $\mathfrak{a}$.

Suppose that $P_1 \cong P_1[w]$, then we label the accumulation point $\mathfrak{a}^+$ by $P_1$.
We shall also label this accumulation point by $P_1[1], \ldots, P_1[w-1]$, but regard these labels as distinct from each other.

Else, if $P_1 \not\cong P_1[w]$, then we fix a marked point in the open interval $(\mathfrak{a},\mathfrak{a}^+)$, and label this point $P_1$.
The successor is labelled by $P_1[-1]$, and the predecessor by $P_1[1]$, with each successive predecessor (resp. successor) labelled by an extra power of the suspension (resp. desuspension) functor until all of the marked points in the open interval $(\mathfrak{a},\mathfrak{a}^+)$ are labelled.

Next, we label points according to the indecomposable direct summand $P_2$.
If $P_2 \cong P_2[w]$, we choose the first unlabelled accumulation point in an anticlockwise direction of the point labelled by $P_1$, and label this unlabelled accumulation point by $P_2$.
Similarly, if $P_2 \not\cong P_2[w]$, then we choose the next open interval between accumulation points with unlabelled marked points in an anticlockwise direction to the point labelled $P_1$.
Fix a marked point in this interval and label it $P_2$, and label the rest of the marked points in the interval accordingly.

We repeat this process for every indecomposable direct summand $P_i$.
Let $\wsM[m]$ be the collection of marked points and accumulation points labelled by an indecomposable object of $\lang[1]{G}$, plus the fixed accumulation point $\mathfrak{a}$.

\begin{definition}\label[def]{def: admissible arcs}
An \textit{arc} on $(\Sigma,\wsM[m])$ is a pair of labels of marked points $\{P[1],Q\}$ such that $P[1]$ is not the same point as $Q$, nor the successor or predecessor of $Q$.

An arc $\{P[1],Q\}$ will be called $w$-\textit{admissible} if either $Q <P$ or $\mathfrak{a} \in \{P[1],Q\}$.
The points $P,Q$ will be called \textit{endpoints} of the arc $\{P,Q\}$.
\end{definition}

\begin{proposition}\label[prop]{Prop: bijection of arcs}
Let $\cT$ be a triangulated category satisfying \Cref{setup:w cluster cat}, with a $w$-linear generator $G$.
Then there exists a bijection between $w$-admissible arcs on $(\Sigma,\wsM[m])$, and isomorphism classes of indecomposable objects in $\cT$.
\end{proposition}

\begin{proof}
	Let $P \in \mathrm{ind}\lang[1]{G}$, then we assign $P$ to the $w$-admissible arc $\{\mathfrak{a},P\}$.
	It is clear by the construction of the collection of labels $\wsM[m]$ that there is a bijection between isomorphism classes of indecomposable objects in $\lang[1]{G}$ and $w$-admissible arcs with an endpoint at $\mathfrak{a}$.
	
	Further, an arc $\{P[1],Q\}$ without $\mathfrak{a}$ as an endpoint is only $w$-admissible if and only if $Q<P$.
	Then there is a clear bijection taking the $w$-admissible arc $\{P[1],Q\}$ to the isomorphism class of indecomposable objects in $\cT$ containing $X_{P,Q}$.
\end{proof}

\begin{lemma}\label[lem]{Lem: cones share endpoints}
	Let $X,Y \in \cT$, and let $\mathscr{X} \subset \wsM[m]$ be the collection of endpoints of the indecomposable direct summands of $X \oplus Y$.
	Consider the distinguished triangle,
	\[
	X[-1] \rightarrow Y \rightarrow C \rightarrow X,
	\]
	and let $\mathscr{C}$ be the collection of endpoints of indecomposable direct summands of $C$.
	Then $\mathscr{C} \subset \mathscr{X}$.
\end{lemma}

\begin{proof}
	This follows the same argument as \cite[Lemmas 3.3, 3.4]{GZ2021}, combined with \Cref{lem:w P to X cone,lem:w X to P cone,lem:w X to X cone}.
\end{proof}

\begin{corollary}\label[cor]{Cor: thick subcats}
	Let $\mathscr{X} \subset \wsM[m]$ be a collection of labelled points, and let $\mathcal{S} \subset \cT$ be the full subcategory with indecomposable objects corresponding to $w$-admissible arcs with endpoints in $\mathscr{X}$.
	
	Then $\mathscr{S}$ is thick if and only if $P \in \mathscr{X}$ implies $P[1],P[-1] \in \mathscr{X}$.
\end{corollary}

\subsection{A Conjecture by Franchini}

In her PhD thesis \cite{FranchiniThesis}, Franchini conjectured the existence of a $(w+1)$-Calabi--Yau discrete cluster category of Dynkin type $A_{\infty}$, and provided a family of candidate categories.
These candidate categories are non-full subcategories of the Igusa--Todorov discrete cluster categories of Dynkin type $A_{\infty}$, which themselves are $2$-Calabi--Yau.
See \cite{Gratz2017} for further details on Igusa--Todorov discrete cluster categories.

Under the assumption of a triangulated structure on these candidate categories, Franchini proved that they are $w$-Calabi--Yau, described the Auslander--Reiten quiver, and classified the torsion pairs and $(w+1)$-cluster-tilting subcategories \cite{FranchiniThesis}.
However, the primary blockage for the assumption of a triangulated structure was in proving that the candidate triangles are closed under cones.

\begin{remark}
	There is a difference in convention used between this work and Franchini's PhD thesis.
	What we call a $(w+1)$-Calabi--Yau triangulated category, is a $w$-Calabi--Yau triangulated category in \cite{FranchiniThesis}.
\end{remark}

We prove here that there exists a thick subcategory of a triangulated category admitting a $w$-linear generator that is additively equivalent to Franchini's candidate categories.
Hence, there exists a triangulated structure on Franchini's candidate categories.

Let $\ocC$ be the Paquette--Y\i ld\i r\i m category with $n$ accumulation points, which is equivalent as triangulated categories to $\mathrm{perf}\, \Lambda_n$, with $\Lambda_n$ the algebra discussed in \Cref{Sec: Lambda} for $w=1$.
Recall that $G \cong \bigoplus_{i=1}^{2n-1} P_i$ such that $P_i \cong P_i[1]$ if and only if $i \equiv 0 \mathrm{mod} 2$, and $P_i < P_j$ if and only if $j < i$.
The Igusa--Todorov discrete cluster category $\cC_n$ is equivalent to the thick subcategory of $\ocC$ comprising of indecomposable objects whose endpoints are not accumulation points.
Let $\mathscr{Z}_n$ denote the set of marked points in the combinatorial model of $\cC_n$.

For some $w \geq 2$, consider the subcategory of $\mathrm{perf}\, \Lambda^w_n$ with the following collection of indecomposable objects,
\[
\cD_{n,w} : = \{X_{P^w,Q^w} \mid P^w \not\cong P^w[w], Q^w \not\cong Q^w[w]\}.
\]
Let $\mathscr{X}_n$ denote the set of marked points in the combinatorial model of $\cD_{n,w}$.

Notice that both $\cC_n$ and $\cD_{n,w}$ have $n$ accumulation points in their combinatorial model.
Therefore there exists a bijection on the sets of marked points,
\begin{align*}
	\phi &: \mathscr{Z}_n \rightarrow \mathscr{X}_n,\\
	& P_{2i-1} \mapsto P^w_i.
\end{align*}

Let $P,Q$ be two marked points in $\mathscr{Z}_n$, then we say that $P-Q \equiv a \, \mathrm{ mod }\, w$ if and only if $\phi(Q[a]) = Q^w[a] \leq P^w = \phi(P)$.
A indecomposable arc $X = \{P[1],Q\} \in \cC_n$ is said to be $w$-\textit{Franchini admissible} if $P[1] - Q \equiv 1 \, \mathrm{ mod } \, w$ \cite{FranchiniThesis}.

\begin{definition}\label[def]{Def: Franchini compatible}\cite{FranchiniThesis}
	Let $X =\{P[1],Q\}$ and $Y = \{R[1],S\}$ be two $w$-Franchini admissible indecomposable objects.
	We say that $Y$ is $w$-\textit{Franchini compatible} with $X$ if either of the following holds:
	\begin{itemize}
		\item $S \leq Q < R \leq P$, and $R - P \equiv 0\, \mathrm{mod}\, w$,
		\item $Q[1] < S \leq P[1] < R$, and $P - R \equiv 1\, \mathrm{mod}\, w$.
	\end{itemize}
	
	Further, a morphism $f : X \rightarrow Y \in \cC_n$ is called $w$-\textit{Franchini admissible} if one of the following holds:
	\begin{itemize}
		\item $f=0$ and $X,Y$ are $w$-Franchini admissible,
		\item $f \neq 0 $ and $Y$ is $w$-Franchini compatible with $X$.
	\end{itemize}
	The \textit{Franchini candidate category} $\cC_{n,w}$ is the additive subcategory of $\cC_n$ comprised of $w$-Franchini admissible indecomposable objects and $w$-Franchini admissible morphisms.
\end{definition}

\begin{lemma}
	The subcategory $\cD_{n,w} \subset \mathrm{perf}\, \Lambda^w_n$ is a thick subcategory.
\end{lemma}

\begin{proof}
	An indecomposable object of $\mathrm{perf}\, \Lambda^w_n$ is in $\cD_{n,w}$ if and only if it is isomorphic to $X_{P^w,Q^w}$ such that $P^w \not \cong P^w[w]$ and $Q^w \not\cong Q^w[w]$.
	Let $\mathscr{X}$ be the collection of endpoints in $\wsM[m]$ corresponding to the endpoints of indecomposable objects in $\cD_{n,w}$.
	Then $\mathscr{X}$ is equal to the collection of marked points labelled by $P^w[i]$ for all $P^w \not\cong P^w[w]$, hence $P^w \in \mathscr{X}$ if and only if $P^w[1],P^w[-1] \in \mathscr{X}$.
	Therefore $\cD_{n,w}$ is a thick subcategory of $\mathrm{perf}\, \Lambda^w_n$ by \Cref{Cor: thick subcats}.
\end{proof}

Let $\cD_{n,w}$ be the subcategory of $\mathrm{perf}\, \Lambda^w_n$ as above, and let
\[
\mathscr{G} := \Lambda_n \otimes_{\Lambda_n^w}^L - : \mathrm{perf}\, \Lambda^w_n \rightarrow \mathrm{perf}\, \Lambda_n
\]
be the derived functor induced by the injective homomorphism in \Cref{Prop: morph of lambda}.
As $\cD_{n,w}$ is a thick subcategory of $\mathrm{perf}\, \Lambda^w_n$, then the image of $\cD_{n,w}$ under $\mathscr{G}$ is also triangulated.

\begin{theorem}\label{Thm: Franchini Conj}
	The essential image of $\cD_{n,w}$ under $\mathscr{G}$ is equivalent to $\cC_{n,w}$. Hence $\cC_{n,w}$ has a triangulated structure compatible with $\ocC \simeq \mathrm{perf}\, \Lambda_n$.
\end{theorem}

\begin{proof}
	Recall that an indecomposable object in $\cD_{n,w}$ is isomorphic to the cone of a morphism $f\colon P^w \rightarrow Q^w \in \lang[1]{\Lambda^w_n}$ such that $P^w \not \cong P^w[w]$ and $Q^w \not\cong Q^w[w]$.
	
	Consider that the functor $\mathscr{G}$ takes the object $\Lambda^w_n \in \mathrm{perf}\, \Lambda_n^w$ to $\Lambda_n \in \mathrm{perf}\, \Lambda_n$, and consider the indecomposable objects $P^w,Q^w \in \lang[1]{\Lambda^w_n}$ such that $P^w \not \cong P^w[w]$ and $Q^w \not\cong Q^w[w]$.
	Then the triangle
	\[
	P^w \rightarrow Q^w \rightarrow X_{P^w,Q^w} \rightarrow P^w[1],
	\]
	is taken to 
	\[
	\mathscr{G}(P^w) \rightarrow \mathscr{G}(Q^w) \rightarrow \mathscr{G}(X_{P^w,Q^w}) \rightarrow \mathscr{G}(P^w[1]).
	\]
	However, as $P^w,Q^w \in \lang[1]{\Lambda^w_n}$, then $\mathscr{G}(P^w) \cong P$ and $\mathscr{G}(Q^w) \cong Q$, hence $\mathscr{G}(X_{P^w,Q^w}) \cong X_{P,Q} \in \cC_{n,w}$.
	
	Now suppose we have an indecomposable object $X_{P,Q} \in \cC_{n,w}$, then $P - Q \equiv 0\, \mathrm{mod} \, w$, and so $\phi(Q) = Q^w \leq P^w = \phi(P)$.
	Hence there exists a triangle
	\[
	P^w \rightarrow Q^w \rightarrow X_{P^w,Q^w} \rightarrow P^w[1],
	\]
	in $\mathrm{perf}\, \Lambda^w_n$ with $X_{P^w,Q^w} \in \cD_{n,w}$ such that $\mathscr{G}(X_{P^w,Q^w}) \cong X_{P,Q}$.
	Hence there exists a bijection between indecomposable objects in $\mathscr{G}(\cD_{n,w})$ and $\cC_{n,w}$.
	
	To check that $\mathrm{perf}\, \Lambda_n (\mathscr{G}(X_{P^w,Q^w}),\mathscr{G}(X_{R^w,S^w})) \cong \mathrm{perf}\, \Lambda_n (X_{P,Q},X_{R,S})$, consider a nonzero morphism $g \colon X_{P^w,Q^w} \rightarrow X_{R^w,S^w}$ in $\cD_{n,w}$.
	Then there exists a triangle
	\[
	X_{P^w,Q^w} \xrightarrow{g} X_{R^w,S^w} \rightarrow Y \oplus Z \rightarrow X_{P^w,Q^w}[1],
	\]
	where $Y \oplus Z$ is determined by \Cref{lem:w X to X cone}.
	By applying $\mathscr{G}$ to the above triangle we see that $\mathscr{G}(g) \neq 0$ as $\mathscr{G}(Y \oplus Z) \not\cong \mathscr{G}(X_{P^w,Q^w}[1] \oplus X_{R^w,S^w})$ by the above bijection on indecomposable objects.
	
	It is now only left to show that a $w$-Franchini admissible morphism is necessarily in $\mathscr{G}(\cD_{n,w})$.
	However, it follows from the definition of a $w$-Franchini admissible morphism that $S \leq Q < R \leq P$ and $R - P \equiv 0 \, \mathrm{mod} \, w$ if and only if $S^w \leq Q^w < R^w \leq P^w$, or $Q[1] < S \leq P[1] < R$ and $P -R \equiv 1 \, \mathrm{mod}\, w$ if and only if $Q^w[1] < S^w \leq P^w[1] < R^w$.
	Hence for each $w$-Franchini admissible morphism $h: X_{P,Q} \rightarrow X_{R,S}$, there exists a nonzero morphism $h' : X_{P^w,Q^w} \rightarrow X_{R^w,S^w} \in \cD_{n,w}$ such that $\mathscr{G}(h') = h$.
	
	Therefore $\cC_{n,w}$ is equivalent to $\mathscr{G}(\cD_{n,w})$, and hence carries the triangulated structure induced by the triangulated functor $\mathscr{G}$ on the thick subcategory $\cD_{n,w}$.
\end{proof}

%% file: ChoiceOfGeneratorsII.tex
\section{Choice of Generators for Hom-sets in Categories with Linear Generators} \label[appendix]{sec: choice of gen}

The goal of this section is to present a proof and explore some corollaries of the following result.

\begin{proposition}
	Let $\mathrm{ind} \, \cT$ denote the set of indecomposable objects in $\cT$. Then, there is a choice of morphisms $\{ \alpha^{Y, Z} \colon Y \to Z \}_{Y, Z \in \mathrm{ind} \, \cT}$ such that for any pair of indecomposable objects $Y,Z \in \cT$ 
	\[
	\alpha^{Y,Z} = \alpha^{W,Z} \alpha^{Y,W},
	\]
	whenever $\alpha^{Y, Z} \colon Y \rightarrow Z$ factors through an indecomposable object $W$.
\end{proposition}

\begin{proof}
	Since $\cT$ is Krull-Schmidt, then $G$ is a finite direct sum of indecomposable objects $\{P_i\}_{i =1}^n$. Throughout this proof, let $\iP = P_j$ and $\iQ = P_k$ where $\iQ = P_k \leq P_i \leq P_j =\iP$ for all $1 \leq i \leq n$. 
	\begin{enumerate}
	\item We first choose $\{ \alpha^{P, Q} \}_{P, Q \in \mathrm{ind} \, \lang[1]{G}}$.

	Let $\alpha^{\iP, \iP}$ be the identity, and fix nonzero morphisms $\alpha^{\iP, Q}$ for $Q < \iP$. For $0< j \in \bZ$, choose nonzero morphisms $\alpha^{\iP[j], \iP}$, and let $\alpha^{\iP[j], Q} := \alpha^{\iP[j], \iP} \circ \alpha^{\iP, Q}$. Extend this choice to include all shifts of the morphisms defined.

	Let $P, Q \in \lang[1]{G}$ such that there is a nonzero morphism $P \to Q$. Equivalently, $Q \leq P$. Then, there are integers $i, j$ and $s,t$ such that $P \cong P_i[s]$ and $Q \cong P_j[t]$. Let $m > \max(s, t)$. Then, by \Cref{prop:w Factoring arcs} there is a morphism $h$ making the diagram 
\[ \begin{tikzcd}
	{} & \iP[m] \arrow[dl, "{\alpha^{\iP[m-s], P_i}[s]}", swap] \arrow[dr, "{\alpha^{\iP[m-t], P_j}[t]}"] & {} \\
	P \cong P_i[s] \arrow[rr, "h", dashed] & {} & Q \cong P_j[t] 
   \end{tikzcd}
\]
commute. We let $\alpha^{P, Q} := h$. In fact, $h$ is independent of choice of $m$. To see this, let $m_1, m_2 > \max(s, t)$ and let $h_1 \colon P \to Q$ and $h_2 \colon P \to Q$ be the morphisms induced by $m_1$ and $m_2$, respectively. Without loss of generality, suppose $m_2 < m_1$ so that there is a morphism $\alpha^{\iP[m_1], \iP[m_2]}$. Then, we have the following diagram.
	\[
	\begin{tikzcd}
		&& \iP[m_1] \arrow[llddd,swap,"\alpha^{\iP[m_1],P}", bend right=3em] \arrow[dd,"\alpha^{\iP[m_1],\iP[m_2]}",pos=0.75] \arrow[rrddd,"\alpha^{\iP[m_1],Q}", bend left=3em] &&\\
		\\
		&& \iP[m_2] \arrow[rrd,"\alpha^{\iP[m_2],Q}",pos=0.1] &&\\
		P  \arrow[rrrr,swap,"h_2", yshift=-0.5em] \arrow[rrrr,swap,"h_1", yshift=0.1em, swap] \arrow[rru,"\alpha^{\iP[m_2],P}", pos=0.9, leftarrow] &&&& Q.
	\end{tikzcd}
	\]
	By construction of the maps $\alpha^{\iP[k], R}$, the top two triangles commute. Hence, 
\[	h_1 \circ \alpha^{\iP[m_1], P} = \alpha^{\iP[m_1], Q} = \alpha^{\iP[m_2], Q} \circ \alpha^{\iP[m_1], \iP[m_2]} = h_2 \circ \alpha^{\iP[m_2], P} \circ \alpha^{\iP[m_1], \iP[m_2]} = h_2 \circ \alpha^{\iP[m_1], P} \]  
	Since the Hom spaces are one-dimensional, $\hcT{\alpha^{\iP[m_1], P} }{Q}$ is either zero or an isomorphism. Since $\alpha^{\iP[m_1], Q}$ is nonzero, then it must be an isomorphism. Hence, $h_1 = h_2$, as required.

	\item The choice of $\alpha^{P, Q}$ is compatible with composition. 

	To see this, note first that $\alpha^{P, Q}$ cannot factor through $W \notin \lang[1]{G}$ by \Cref{prop:w Factoring arcs}. So, suppose that $\alpha^{P, Q}$ factors through $R \in \lang[1]{G}$. Then, $Q \leq R \leq P$ and we may choose a large enough $m$ such that diagram
	\[
	\begin{tikzcd}
		&& \iP[m] \arrow[lldd,swap,"\alpha^{\iP[m],P}", bend right=3em] \arrow[d,"\alpha^{\iP[m],R}",pos=0.35] \arrow[rrdd,"\alpha^{\iP[m],Q}", bend left=3em] &&\\
		&& R \arrow[rrd,"\alpha^{R,Q}",pos=0.1] &&\\
		P  \arrow[rrrr,swap,"\alpha^{P, Q}"] \arrow[rru,"\alpha^{P, R}", pos=0.9] &&&& Q.
	\end{tikzcd}
	\]
 	exists. The top two triangles and the outer triangle commute by construction of $\alpha^{P, R}$, $\alpha^{R, Q}$ and $\alpha^{P, Q}$. Therefore, 
	\[ \alpha^{R, Q} \circ \alpha^{P, R} \circ \alpha^{\iP[m], P}  = \alpha^{\iP[m], Q} = \alpha^{P, Q} \circ \alpha^{\iP[m], P}.\]
	As before, this is enough to conclude that $\alpha^{R, Q} \circ \alpha^{P, R} = \alpha^{P, Q}$, as required.
	
	\item We next choose morphisms $\{\alpha^{P, Z}\}_{P \in \mathrm{ind} \, \lang[1]{G}, Z \in \mathrm{ind} \, \cT \setminus \mathrm{ind} \, \lang[1]{G}}$.

	For each indecomposable object $Z \notin \lang[1]{G}$, \Cref{Thm:w Perf 2 is Tri} implies that $Z \cong X_{P, Q}$ for $P, Q \in \lang[1]{G}$. For each such $Z$, choose a nonzero morphism $\alpha^{Q, Z} \in \hcT{Q, Z}$. Moreover, for $R \in \lang[1]{G}$ such that there is a nonzero morphism $R \to Z$, then there is a morphism $R \to Q$ by \Cref{prop:w Factoring arcs}. Thus, define $\alpha^{R, Z} := \alpha^{Q, Z} \circ \alpha^{R, Q}$. 

	\item Let us choose forwards morphisms $\{\alpha^{Y, Z}\}_{Y, Z \in \mathrm{ind} \, \cT \setminus \mathrm{ind} \, \lang[1]{G}}$. 
	
	Since $Y, Z \in \mathrm{ind} \, \cT \setminus \mathrm{ind} \, \lang[1]{G}$, then $Y \cong X_{P, Q}$ and $Z \cong X_{P', Q'}$. Therefore, by \Cref{prop:w Factoring arcs}, there exists for any $Q \leq R < P$ a morphism $h$ making the diagram
	\[
	\begin{tikzcd}
		& R \arrow[ld,"\alpha^{R, Y}",swap] \arrow[rd,"\alpha^{R, Z}"] &\\
		X_{P,Q} \arrow[rr,dashed, "h",swap] && X_{P',Q'}.
	\end{tikzcd}
	\]
commute. We let $\alpha^{Y, Z} := h$. As in the proof of (1), $h$ is actually independent of the choice of $R$ by construction of the morphisms $\alpha^{R, Y}$ and $\alpha^{R, Z}$. 

	We will postpone showing compatibility of this choice with composition until we have defined our choices for backwards morphisms.  

	\item We now choose morphisms $\{\alpha^{Y, Q}\}_{Q \in \mathrm{ind} \, \lang[1]{G}, Y \in \mathrm{ind} \, \cT \setminus \mathrm{ind} \, \lang[1]{G}}$.

	To ease notation, let $\iX^t := X_{\iP[t], \iQ}$. Choose a nonzero morphism $\alpha^{\iX^0, \iP[1]} \colon \iX^0 \to \iP[1]$. By \Cref{prop:w Factoring arcs}, for $S \in \lang[1]{G}$ and $t \geq 0$, there exist morphisms $\alpha^{\iX^t, \iP[t+1]}$ and $\alpha^{X_{\iP[t],S}, \iP[t+1}]$ making the following diagram commute.
\[\begin{tikzcd}[column sep=6em]
	X_{\iP[t], S} \arrow[r, "{\alpha^{X_{\iP[t], S}, \iX^t}}"] \arrow[d, "{\alpha^{X_{\iP[t], S},\iP[t+1]}}", swap, dashed]  & X_{\iP[t], \iQ} \arrow[r, "{\alpha^{\iX^t, \iX^0}}"] \arrow[d, "{\alpha^{\iX^t, \iP[t+1]}}", swap, dashed]  & X_{\iP, \iQ} \arrow[d, "{\alpha^{\iX, \iP[1]}}", swap] \\
	\iP[t+1] \arrow[r, "{\alpha^{\iP, \iP}[t+1]}"] & \iP[t+1] \arrow[r, "{\alpha^{\iP[t+1], \iP[1]}}"] & \iP[1]
\end{tikzcd}
\]
	
	Now, let $Y \notin \lang[1]{G}$, so that we may assume $Y = X_{R, S}$, where $R = P_i[t]$ for integers $i$ and $t$. By choosing $\kmax > \max(t, 0)$, we have that $R \leq \iP[\kmax]$ so that, by \Cref{lem:w Morphisms X to X}, there is a forwards morphism $X^{\iP[\kmax], S} \to X_{R, S}$. Thus, by \Cref{prop:w Factoring arcs}, there is a nonzero morphism $\alpha^{Y, R[1]}$ making the following diagram commute.
\[\begin{tikzcd}[column sep=6em]
	X_{\iP[\kmax], S} \arrow[d, "{\alpha^{X_{\iP[\kmax], S},\iP[\kmax+1]}}"] \arrow[r, "{\alpha^{X_{\iP[\kmax], S}, Y}}"] & X_{R, S} \arrow[d, "{\alpha^{Y, R[1]}}", dashed] \\
	\iP[\kmax+1] \arrow[r, "{\alpha^{\iP[\kmax+1], R[1]}}"] & R[1]
\end{tikzcd}
\]
	As before, our choice of $\alpha^{Y, R[1]}$ does not depend on which $\kmax$ we choose. 
 
	Next, suppose that there is a nonzero morphism $Y \to Q$. Then, $S \leq Q < R$ by \Cref{lem:w Hom X to P}, and so we let $\alpha^{Y, Q} := \alpha^{R[1], Q} \circ \alpha^{Y, R[1]}$.

	\item It remains to check that our choice of $\{\alpha^{Y, Q}\}_{Q \in \mathrm{ind} \, \lang[1]{G}, Y \in \mathrm{ind} \, \cT \setminus \mathrm{ind} \, \lang[1]{G}}$ is compatible with composition. 

	To see this, suppose that a morphism $Y \to Q$ factors through an indecomposable object $W$. If $W \in \lang[1]{G}$, then it is easy to check that, by construction, $\alpha^{Y, Q} = \alpha^{W, Q} \circ \alpha^{Y, W}$. So, suppose that $W \notin \lang[1]{G}$. Then, $W \cong X_{P', Q'}$ and, by \Cref{prop:w Factoring arcs}, morphisms $Y \to W$ are forwards morphisms.  Moreover, by taking low enough $\kmin$ and high enough $\kmax$, there are forwards morphisms $X_{\iP, \iQ[\Delta]}[\kmax] \to Y$ and $X_{\iP, \iQ[\Delta]}[\kmax] \to X$. For brevity, let $Z:= X_{\iP, \iQ[\Delta]}[\kmax]$. We thus have the diagram
	\[
	\begin{tikzcd}
		&& Z \arrow[llddd,swap,"\alpha^{Z, Y}"] \arrow[dd,"\alpha^{Z,Q}",pos=0.75] \arrow[rrddd,"\alpha^{Z,W}"] &&\\
		\\
		&& Q \arrow[rrd,"\alpha^{W,Q}",pos=0.1, leftarrow] &&\\
		Y  \arrow[rrrr,swap,"\alpha^{Y, W}"] \arrow[rru,"\alpha^{Y, Q}", pos=0.9] &&&& W.
	\end{tikzcd}
	\]
	where the top two triangles and the outer triangle commute by construction of the morphisms in the diagram. It is straightforward to check that
\[ \alpha^{Y, Q} \circ \alpha^{Z, Y} = \alpha^{Z,Q} = \alpha^{W, Q} \circ \alpha^{Z, W} = \alpha^{W, Q} \circ \alpha^{Y, W} \circ \alpha^{Z, Y} \]
which, as in the proof of (1), is enough to conclude that $\alpha^{Y, Q} = \alpha^{W, Q} \circ \alpha^{Y, W}$, as required.

\item We complete our choice by choosing backwards morphisms $\{\alpha^{Y, Z}\}_{Y, Z \in \mathrm{ind} \, \cT \setminus \mathrm{ind} \, \lang[1]{G}}$.

	Write $Y = X_{R,S}$ and $Z = X_{P', Q'}$, and suppose that $Y \to Z$ is a backwards morphism. Then, for any $Q' \leq P \leq R[1]$, \Cref{prop:w Factoring arcs} implies the existence of the diagram
\[\begin{tikzcd}
	Y \arrow[rr, "\alpha^{Y, Z}", dashed] \arrow[dr, "\alpha^{Y, P}"] && Z \arrow[dl, "\alpha^{P, Z}", leftarrow] \\
        {} & P & {}
\end{tikzcd}
\] 
where we define $\alpha^{Y, Z}$ to be the morphism making the diagram commute. As before, this choice is independent of $P$. 

\item Finally, we show that our choices are compatible with composition. 

	Consider $\alpha^{Y, Z} \colon Y \to Z$, and suppose it factors through an indecomposable object $W$. If both $Y, Z \in \lang[1]{P}$, then $W \in \lang[1]{G}$ by \Cref{prop:w Factoring arcs}, and we have shown in (2) that compatibility of composition holds. 

	Suppose that $Y \in \lang[1]{G}$ and $Z =X_{R, S}$. If and $W \in \lang[1]{G}$,  then by (1) and the definition of $\alpha^{Y, Z}$ and $\alpha^{W, Z}$,  
\[ \alpha^{W, Z} \circ \alpha^{Y, W} = \alpha^{S, Z} \circ \alpha^{W, S} \circ \alpha^{Y, W} = \alpha^{S, Z} \circ \alpha^{Y, S} = \alpha^{Y, Z} \]
as required. If, on the other hand, $W = X_{P, Q}$, then morphisms $W \to Z$ are forwards by \Cref{prop:w Factoring arcs}. Hence,  $\alpha^{W, Z} \circ \alpha^{Y, W} = \alpha^{Y, Z}$ by definition of forwards morphisms $\alpha^{W, Z}$ in (4).

	The case where $Y = X_{P, Q}$ and $Z \in \lang[1]{G}$ was shown in (6). Hence, it remains to consider the setting of $Y = X_{P, Q}$ and $Z = X_{R, S}$.
	
	Under the assumption $Y = X_{P, Q}$ and $Z = X_{R, S}$, suppose that $\alpha^{Y, Z}$ is a forwards morphism. Then by 
\Cref{prop:w Factoring arcs}, $W=X_{P',Q'}$ and the morphisms $\alpha^{Y, W}$ and $\alpha^{W, Z}$ are forwards. It follows from \Cref{lem:w Hom P to X,lem:w Morphisms X to X}, as well as the choice of forwards morphisms in (4) that there is that there is a diagram
	\[
	\begin{tikzcd}
		&& Q \arrow[llddd,swap,"\alpha^{Q,Y}"] \arrow[dd,"\alpha^{Q,W}",pos=0.75] \arrow[rrddd,"\alpha^{Q,Z}"] &&\\
		\\
		&& W \arrow[rrd,"\alpha^{W,Z}",pos=0.1] &&\\
		Y  \arrow[rrrr,swap,"\alpha^{Y, Z}"] \arrow[rru,"\alpha^{Y, W}", pos=0.9] &&&& Z.
	\end{tikzcd}
	\]
	where the top two triangles and the outer triangle commute. It is easily checked that this implies that the entire diagram commutes (see e.g. \. (2) ), which proves compatibility with composition. 

	Finally, suppose that $Y = X_{P, Q}$, $Z = X_{R, S}$, and $\alpha^{Y, Z}$ is a backwards morphism. If $W \in \lang[1]{G}$, then compatibility follows from the definition of $\alpha^{Y, Z}$ in (7). If $W = X_{P', Q'}$ and $\alpha^{Y, W}$ is a backwards morphism, then $\alpha^{W, Z}$ is forwards and the following diagram exists by \Cref{lem:w Hom P to X,lem:w Morphisms X to X} and \Cref{prop:w Factoring arcs}.
	\[
	\begin{tikzcd}
		&& Q' \arrow[llddd,swap,"\alpha^{Q',Y}", leftarrow] \arrow[dd,"\alpha^{Q',W}",pos=0.75] \arrow[rrddd,"\alpha^{Q',Z}"] &&\\
		\\
		&& W \arrow[rrd,"\alpha^{W,Z}",pos=0.1] &&\\
		Y  \arrow[rrrr,swap,"\alpha^{Y, Z}"] \arrow[rru,"\alpha^{Y, W}", pos=0.9] &&&& Z.
	\end{tikzcd}
	\]
	Here, the top left triangle and the outer triangle commute by the choice of backwards morphisms in (7), and the top right triangle commutes by the choice of forwards morphism in (4). As before, it is easily checked that this implies that the entire diagram commutes, as required.

	On the other hand, if, under the assumption that $W = X_{P', Q'}$, $\alpha^{Y, W}$ is a forwards morphism, then $\alpha^{W, Z}$ is backwards and the following diagram exists by \Cref{lem:w Hom P to X,lem:w Morphisms X to X} and \Cref{prop:w Factoring arcs}.
	\[
	\begin{tikzcd}
		&& Q \arrow[lldddd,swap,"\alpha^{Q,Y}", bend right = 3em] \arrow[d,"\alpha^{Q,S}",pos=0.75, leftarrow] \arrow[rrdddd,"\alpha^{Q,W}", bend left = 3em] &&\\
		&& S \arrow[llddd,swap,"\alpha^{Y,S}", leftarrow] \arrow[dd,"\alpha^{S,Z}",pos=0.75] \arrow[rrddd,"\alpha^{W, S}", leftarrow] &&\\
		\\
		&& Z \arrow[rrd,"\alpha^{W,Z}",pos=0.1, leftarrow] &&\\
		Y  \arrow[rrrr,swap,"\alpha^{Y, W}"] \arrow[rru,"\alpha^{Y, Z}", pos=0.9] &&&& W.
	\end{tikzcd}
	\]
	The triangle with vertices $Y, Q, W$ commutes by construction of forwards morphisms in (4), and the triangle with vertices $Y, S, W$ commutes by (6). Moreover, the triangles with vertices $Y, S, Z$ and $Z, S, W$ commute by choice of backwards morphisms in (7). Hence, 
\[ \alpha^{W, Z} \circ \alpha^{Y, W} \circ \alpha^{Q, Y} = \alpha^{W, Z} \circ \alpha^{Q, W} = \alpha^{S, Z} \circ \alpha^{W, S} \circ \alpha^{Q, W} = \alpha^{S, Z} \circ \alpha^{Y, S} \circ \alpha^{Q, Y} = \alpha^{Y, Z} \circ \alpha^{Q, Y} \]
which implies that $\alpha^{W, Z} \circ \alpha^{Y, W} = \alpha^{Y, Z}$, as required. 
	\end{enumerate}
\end{proof}

\begin{corollary} \label[cor]{cor: extending basis}
	Let $\cT$ be a triangulated category satisfying \Cref{setup:w cluster cat} and admitting a linear generator $G$. Let $G'$ be a summand of $G$, let $\cS = \lang{G'} \subseteq \cT$ and suppose that $\{\alpha^{X, Y}\}_{X, Y \in \mathrm{ind} \, \cS}$ is a choice of generators for $\cS$ as in \Cref{prop:w Generators of Hom}. Then, there is a choice of generators $\{\beta^{X, Y}\}_{X, Y \in \mathrm{ind} \, \cT}$ such that $\beta^{X, Y} = \alpha^{X, Y}$ whenever $X, Y \in \cS$.
\end{corollary}
\begin{proof}
	As in the proof of \Cref{prop:w Generators of Hom}, $\iP$ and $\iP'$ be the biggest (with respect to the order on $\lang[1]{G}$) indecomposable summands of $G$ and $G'$, respectively. Observe that $\iP' \leq \iP$. 
	\begin{enumerate}
	\item We first construct a collection $\{\beta^{P, Q}\}_{P, Q \in \mathrm{ind} \, \lang[1]{G}}$ compatible with composition such that $\beta^{P, Q} = \alpha^{P, Q}$ whenever $P, Q \in \lang[1]{G'}$. 
 	
	We are given morphisms $\alpha^{\iP', Q}$ for all $Q \leq \iP'$ and $Q \in \lang[1]{G'}$. Extend this collection by making choices of morphisms $\alpha^{\iP', Q}$ for all $Q \leq \iP'$ and $Q \in \lang[1]{G} \setminus \lang[1]{G'}$. Moreover, we are also given morphisms $\alpha^{\iP'[j],\iP'}$ for $0 < j \in \bZ$.

 	Next, fix a morphism $\beta^{\iP, \iP'}$. Note that for $0 < j \in \bZ$, $\iP' \leq \iP \leq \iP[j]$. Hence, by the definition of linear generators, there is an $h$ such that 
\[ \alpha^{\iP'[j], \iP'} \circ \beta^{\iP, \iP'}[j] = \beta^{\iP, \iP'} \circ h. \]
We let $\beta^{\iP[j], \iP} := h$.

	Furthermore, for all $Q \leq \iP' \in \lang[1]{G}$, let $\beta^{\iP, Q} = \alpha^{\iP', Q} \circ \beta^{\iP, \iP'}$. On the other hand, for $Q \in \lang[1]{G}$ with $\iP' < Q < \iP$, just fix a choice of morphisms $\beta^{\iP, Q}$. Finally, let $\beta^{\iP, \iP}$ be the identity.  With this initial choice of morphisms, apply \Cref{prop:w Generators of Hom}(1) to get a collection $\{\beta^{P, Q}\}_{P, Q \in \mathrm{ind} \, \lang[1]{G}}$. 

	It remains to prove that $\beta^{P, Q} = \alpha^{P, Q}$ whenever $P, Q \in \lang[1]{G'}$. Well, by the proof of \Cref{prop:w Generators of Hom} and our choices for $\alpha^{P, Q}$ and $\beta^{P, Q}$, there is an integer $m$ such that there is a diagram
\[
\begin{tikzcd}[column sep=8em]
	& \iP[m] \arrow[d, "{\beta^{\iP, \iP'}[m]}", pos=0.7, yshift=0.2em] \arrow[ddl, "{\beta^{\iP[m], P}}", swap, pos=0.3, xshift=-0.5em, yshift=0.2em] \arrow[ddr, "{\beta^{\iP[m], Q}}", pos=0.3, xshift=0.5em, yshift=0.2em]  & \\[1.5em] 
	& \iP'[m] \arrow[dl, "{\alpha^{\iP'[m], P}}", swap, pos=0.15, yshift=0.2em] \arrow[dr, "{\alpha^{\iP'[m], Q}}", pos=0.15, yshift=0.2em] & \\
	P \arrow[rr, "\alpha^{P, Q}", xshift=0.2em] \arrow[rr, "\beta^{P, Q}", yshift=-0.5em, xshift=0.2em, swap] & & Q
\end{tikzcd}
\]
where the triangle with vertices $P, Q, \iP[m]$ and edge $\beta^{P, Q}$, as well as the triangle with vertices $P, Q, \iP'[m]$ and edge $\alpha^{P, Q}$ commute by construction of $\beta^{P, Q}$ and $\alpha^{P, Q}$. Moreover, the top left triangle commutes since
\begin{align*}
\alpha^{\iP'[m], P} \circ \beta^{\iP, \iP'}[m] & = \alpha^{\iP', P} \circ \alpha^{\iP'[m], \iP'} \circ \beta^{\iP, \iP'}[m] \tag{\text{Since $\{\alpha^{P, Q}\}$ is compatible with composition} }\\
& = \alpha^{\iP', P} \circ \beta^{\iP, \iP'} \circ \beta^{\iP[m], \iP} \tag{ \text{By definition of $\beta^{\iP[m], \iP}$}} \\
& = \beta^{\iP[m], P} \tag{\text{By definition of $\beta^{\iP[m], P}$}}. 
\end{align*}
Similarly, the top right triangle also commutes. Hence,
\[ \alpha^{P, Q} \circ \beta^{\iP[m], P} = \alpha^{P, Q} \circ \alpha^{\iP'[m], P} \circ \beta^{\iP, \iP'}[m] = \alpha^{\iP'[m], Q} \circ \beta^{\iP, \iP'}[m] = \beta^{\iP[m], Q} = \beta^{P, Q} \circ \beta^{\iP[m], P}\]
which is enough to conclude that $\alpha^{P, Q} = \beta^{P, Q}$, as required.

\item  For the morphisms $\{\beta^{P, Z}\}_{P \in \mathrm{ind} \, \lang[1]{G}, Z \in \mathrm{ind} \, \cT \setminus \lang[1]{G}}$, we make the obvious choice: If $Z = X_{R, S} \in \cS$, then we let $\beta^{S, Z} = \alpha^{S, Z}$, otherwise if $Z = X_{R, S} \notin \cS$, then we fix some morphism $\beta^{S, Z}$. For any $R \in \lang[1]{G}$ such that there is a nonzero morphism $R \to S$, we let $\beta^{R, Z} = \beta^{S, Z} \circ \beta^{R, S}$. Thus, it is clear that $\beta^{P, Z} = \alpha^{P, Z}$ whenever $P, Z \in \cS$.

\item The construction forwards morphisms $\{\beta^{Y, Z}\}_{Y, Z \in \mathrm{ind} \, \cT \setminus \lang[1]{G}}$ follows exactly as in \Cref{prop:w Generators of Hom}(4). It is clear from (1) that $\beta^{Y, Z} = \alpha^{Y, Z}$ for $Y, Z \in \cS$.

\item To choose morphisms $\{\beta^{Y, P}\}_{P \in \mathrm{ind} \, \lang[1]{G}, Y \in \mathrm{ind} \, \cT \setminus \lang[1]{G}}$, consider the object $\iX' := X_{\iP', \iQ'}$ and note that we are given a morphism $\alpha^{\iX', \iP'[1]} \colon \iX' \to \iP'[1]$. Since there are forwards morphisms $X_{\iP, \iQ'} \to X_{\iP, \iQ}$ and $X_{\iP, \iQ'} \to \iX'$, \Cref{prop:w Factoring arcs} implies that there are morphisms $g$ and $h$ making the diagram 
\[\begin{tikzcd}[column sep=4em] \label{cd: extending basis}
X_{\iP, \iQ} \arrow[d, "g", dashed, swap] \arrow[r, "{\beta^{X_{\iP, \iQ'}, X_{\iP, \iQ} } }", leftarrow] & X_{\iP, \iQ'} \arrow[r, "{\beta^{X_{\iP, \iQ'}, \iX'} }"] \arrow[d, "h", dashed, swap] & \iX'  \arrow[d, "{ \alpha^{\iX', \iP'[1]} }"] \\
\iP[1] \arrow[r, "{\beta^{\iP, \iP}[1]}", leftarrow] & \iP[1] \arrow[r, "{\beta^{\iP, \iP'}[1]}"] & \iP'[1]
\end{tikzcd}
\]
commute. Let $\iX := X_{\iP, \iQ}$ and $\beta^{\iX, \iP[1]} := g$. Then, construct the morphisms $\{\beta^{Y, P}\}_{P \in \mathrm{ind} \, \lang[1]{G}, Y \in \mathrm{ind} \, \cT \setminus \lang[1]{G}}$ as in \Cref{prop:w Generators of Hom}(5). 

It remains to show that this choice of morphisms $\beta^{Y, P}$ restricts to $\alpha^{Y, P}$. To see this, let $S \in \lang[1]{G'}$. We may construct the diagram
\[ \begin{tikzcd}[column sep=4em, row sep=2em] \label{cd: extending basis 2}
X_{\iP, \iQ} \arrow[d, "{\beta^{X_{\iP, \iQ}, \iP[1]}}"] & X_{\iP, \iQ'}  \arrow[d, "{\beta^{X_{\iP, \iQ'}, \iP[1]}}"] \arrow[l, "{\beta^{X_{\iP, \iQ'},X_{\iP, \iQ}}}", swap] & X_{\iP, S}  \arrow[d, "{\beta^{X_{\iP, S}, \iP[1]}}"] \arrow[l, "{\beta^{X_{\iP, S},X_{\iP, \iQ'}}}",swap]  \arrow[r, "{\beta^{X_{\iP, S},X_{\iP', S}}}"] & X_{\iP', S}  \arrow[d, "{\alpha^{X_{\iP', S}, \iP[1]}}"] \arrow[r, "{\beta^{X_{\iP', S},X_{\iP', \iQ'}}}"] & \iX' \arrow[d, "{\alpha^{iX', \iP[1]}}"] \\
\iP[1] \arrow[r, "{\beta^{\iP, \iP}[1]}", swap] & \iP[1] \arrow[r, "{\beta^{\iP, \iP}[1]}", swap] & \iP[1]  \arrow[r, "{\beta^{\iP, \iP'}[1]}", swap]   & \iP'[1]  \arrow[r, "{\beta^{\iP', \iP'}[1]}", swap] & \iP'[1]
\end{tikzcd}
\]
where all the morphisms in the first row are forwards. The two leftmost squares and the rightmost square commute by construction of the vertical morphisms. The outermost square commutes by \eqref{cd: extending basis}. It is then straighforward to check that the third inner square commutes. 
 
Next, let $X_{R, S} \in \cS$, then, as in the proof of \Cref{prop:w Generators of Hom}(5), there is an integer $t$ such that there are forwards morphisms $X_{\iP'[t],S}\to X_{R, S}$ and $X_{\iP[t], S} \to X_{R,S}$, and so following diagram exists
\[\begin{tikzcd}[column sep=5em, row sep=2em]
	X_{\iP'[t], S} \arrow[d, equals] & {} & X_{\iP[t], S} \arrow[ll, "{\beta^{X_{\iP[t], S}, X_{\iP'[t], S}}}", swap] \\
	X_{\iP'[t], S} \arrow[d, "{\alpha^{X_{\iP'[t], S}, \iP'[t+1]}}", swap] \arrow[r, "\beta^{X_{\iP'[t], S}, X_{R,S}}"] & X_{R, S}  \arrow[d, "{\alpha^{X_{R, S}, R[1]}}", xshift=-0.25em, swap]  \arrow[d, "{\beta^{X_{R, S}, R[1]}}", xshift=0.25em] & X_{\iP[t], S} \arrow[l, "\beta^{X_{\iP[t], S}, X_{R,S}}", swap] \arrow[u, equals] \arrow[d, "{\beta^{X_{\iP[t], S}, \iP[t+1]}}"] \\
	\iP'[t+1] \arrow[r, "{\beta^{\iP'[t+1], R[1]}}", swap]  & R[1] & \iP[t+1]  \arrow[l, "{\beta^{\iP[t+1], R[1]}}"] \arrow[d, equals] \\
	\iP'[t+1] \arrow[u, equals]     & {}       &  \iP[t+1]. \arrow[ll, "{ \beta^{\iP[t+1], \iP'[t+1]} }"] 
\end{tikzcd}
\]
Here, the outer square commutes commutes by \eqref{cd: extending basis}. Moreover, the two middle squares commute by construction of the morphisms $\alpha^{X_{R, S}, R[1]}$ and $\beta^{X_{R, S}, R[1]}$. It is, thus, straightforward to check that this implies $\alpha^{X_{R, S}, R[1]} = \beta^{X_{R, S}, R[1]}$, as required.

\item Finally, we construct backwards morphisms $\{\beta^{Y, Z}\}_{Y, Z \in \mathrm{ind} \, \cT \setminus \lang[1]{G}}$ exactly as in \Cref{prop:w Generators of Hom}(7). It is clear from the construcition that if $Y, Z \in \cS$ then $\beta^{Y, Z} = \alpha^{Y, Z}$.
\end{enumerate}
\end{proof}